\documentclass[reqno, a4paper,12pt]{amsart}
\usepackage{natbib}
\usepackage{amsmath}
\renewcommand{\l}{\ell}
\usepackage{amssymb}
\usepackage{amsthm}
\usepackage{mathrsfs}
\usepackage{mathabx} 
\usepackage{mathrsfs}
 \usepackage{dsfont}
 \usepackage{bm}
 \usepackage[bmargin=4cm,tmargin=4cm,rmargin=2.7cm,lmargin=2.6cm]{geometry}
\newtheorem{theorem}{Theorem}[section]
\newtheorem{lemma}[theorem]{Lemma}
\newtheorem{example}[theorem]{Example}
\newtheorem{definition}[theorem]{Definition}
\newtheorem{corollary}[theorem]{Corollary}
\newtheorem{remark}[theorem]{Remark}
\newtheorem{proposition}[theorem]{Proposition}
\newtheorem{question}[theorem]{Question}
\newtheorem{problem}[theorem]{Problem}
\newcommand{\esssup}{\operatorname{ess\:sup}}
\DeclareBoldMathCommand\balpha{\alpha}
\DeclareSymbolFont{largesymbolsA}{U}{txexa}{m}{n}
\SetSymbolFont{largesymbolsA}{bold}{U}{txexa}{bx}{n}
\DeclareFontSubstitution{U}{txexa}{m}{n}
\DeclareMathSymbol{\varprod}{\mathop}{largesymbolsA}{16}

\title{The Grothendieck inequality revisited}

\author{Ron Blei}

\email{Ron Blei, blei@math.uconn.edu}

\address{Department of Mathematics, University of Connecticut, Storrs, CT 06269}

\date{\today}



\begin{document}

\begin{abstract}
The classical Grothendieck inequality is viewed as a statement about representations of functions of two variables over discrete domains by integrals of two-fold products of functions of one variable.  An analogous statement is proved, concerning  continuous functions of two variables over general topological domains.  The main result is the construction of a continuous map $\Phi$ from  $l^2(A)$ into $L^2(\Omega_A, \mathbb{P}_A)$,  where $A$ is a set,  \  $\Omega_A = \{-1,1\}^A$, \ and $\mathbb{P}_A$ is the uniform probability measure on $\Omega_A$, such that\\
 \begin{equation} \label{ab1}
 \sum_{\alpha \in A} {\bf{x}}(\alpha) \overline{{\bf{y}}(\alpha)} = \int_{\Omega_A} \Phi({\bf{x}})\Phi(\overline{{\bf{y}}})d \mathbb{P}_A, \ \ \ {\bf{x}} \in l^2(A), \ {\bf{y}} \in l^2(A),\\
 \end{equation}
 and
 \begin{equation} \label{ab2}
  \|\Phi({\bf{x}})\|_{L^{\infty}} \leq K \|{\bf{x}}\|_2, \ \ \ {\bf{x}} \in l^2(A),\\
  \end{equation}
for an absolute constant $K > 1$. ($\Phi$ is non-linear, and does not commute with complex conjugation.) The Parseval-like formula in \eqref{ab1} is obtained by iterating the usual Parseval formula in a framework of harmonic analysis on dyadic groups.  A modified construction implies a similar integral representation of the dual action between $l^p$ and $l^q$, \ $\frac{1}{p} + \frac{1}{q} = 1$.

Variants of the Grothendieck inequality in higher dimensions are derived.  These variants involve representations of functions of $n$ variables in terms of functions of $k$ variables, $0 < k < n.$  Multilinear Parseval-like formulas are obtained, extending the bilinear formula in \eqref{ab1}.  The resulting formulas imply multilinear extensions of the Grothendieck inequality, and are used  to characterize the feasibility of integral representations of multilinear functionals on a Hilbert space.

\end{abstract}

\maketitle

\numberwithin{equation}{section}
\newpage
\tableofcontents 
\section{\bf{Introduction}}\label{s0}
 \subsection{The inequality} \ We start with an infinite-dimensional Euclidean space, whose coordinates are indexed by a set $A$,  \\\
\begin{equation}\label{i0}
l^2 =  l^2(A) := \big\{{\bf{x}} = \big({\bf{x}}(\alpha)\big)_{ \alpha \in A} \in \mathbb{C}^{A}:  \sum_{\alpha \in A} |{\bf{x}}(\alpha)|^2 < \infty \big\},
\end{equation}
equipped with the usual dot product 
\begin{equation*}
\langle {\bf{x}}, {\bf{y}} \rangle  \ :=  \ \sum_{\alpha \in A} {\bf{x}}(\alpha) \overline{{\bf{y}}(\alpha)}, \ \ \ {\bf{x}} \in l^2, \ {\bf{y}} \in l^2,
\end{equation*}
and the Euclidean norm
\begin{equation*}
\|{\bf{x}}\|_2 \ :=  \ \sqrt{\langle {\bf{x}}, {\bf{x}} \rangle} \ = \  \left ( \sum_{\alpha \in A} |{\bf{x}}(\alpha)|^2 \right )^{\frac{1}{2}}, \ \ \ {\bf{x}} \in l^2.
\end{equation*}
We let $B_{l^2}$ denote the closed unit ball,
\begin{equation*}
B_{l^2} := \big \{ {\bf{x}} \in l^2:  \|{\bf{x}}\|_2 \leq 1 \big \}.
\end{equation*}
The Grothendieck inequality in this setting is the assertion that there exists  $1 < K < \infty$ such that for every finite scalar array  $(a_{jk})$,\\\
 \begin{equation} \label{i1}
 \sup \bigg \{\big |\sum_{j,k}a_{jk}   \langle{\bf{x}}_j, {\bf{y}}_k \rangle \big |: ({\bf{x}}_j, {\bf{y}}_k) \in (B_{l^2})^2  \bigg \} \ \leq  \ K \sup \bigg \{\big |\sum_{j,k}a_{jk}s_jt_k \big |:  (s_j,t_k) \in [-1,1]^2 \bigg \}.\\\\
 \end{equation}\\

An assertion equivalent to \eqref{i1}, couched in a setting of topological tensor products,  had appeared first in Alexandre Grothendieck's landmark  \emph{Resum\'e}  \cite{Grothendieck:1956}, and had remained largely unnoticed until it was deconstructed and reformulated in ~\cite{Lindenstrauss:1968} -- another classic -- as the  inequality above.  Since its reformulation,  which became known as  \emph{the Grothendieck inequality}, it has been duly recognized as a fundamental statement, with diverse appearances and applications in functional, harmonic, and stochastic analysis, and recently also in theoretical physics and theoretical computer science.  (See ~\cite{pisier2012grothendieck}, and also Remark \ref{R2}.ii in this work.) 

The numerical value of the "smallest" $K$ in \eqref{i1}, denoted by $\mathcal{K}_G$ and dubbed \emph{the Grothendieck constant}, is an open problem that to this day continues to attract interest. (For the latest on $\mathcal{K}_G$, see ~\cite{braverman2011grothendieck}.)
 
\subsection{An integral representation} \ 
Consider the infinite product space 
\begin{equation} \label{i2}
\Omega_{B_{l^2}} := \{-1,1\}^{B_{l^2}},
\end{equation}
equipped with the usual product topology and the sigma field  generated by it, and consider the coordinate functions  $r_{{\bf{x}}}: \Omega_{B_{l^2}} \rightarrow \{-1,1\}$,  defined by
\begin{equation} \label{i4}
r_{{\bf{x}}}(\omega) = \omega({\bf{x}}), \ \ \ \omega \in  \{-1,1\}^{B_{l^2}}, \ \ \ {\bf{x}} \in B_{l^2}.
\end{equation}
We refer to $R_{B_{l^2}}  :=  \{r_{{\bf{x}}}: {\bf{x}} \in B_{l^2} \}$ as a \emph{Rademacher system indexed by} $B_{l^2}$, and to its members as \emph{Rademacher characters}; see \S \ref{s3} of this work. 
\begin{proposition} \label{P0}
The Grothendieck inequality \eqref{i1} is equivalent to the existence of a complex measure $$\lambda \in M\big (\Omega_{B_{l^2}} \times \Omega_{B_{l^2}} \big ) \  (= \big \{\text{complex measures on} \  \ \Omega_{B_{l^2}} \times \Omega_{B_{l^2}}\big \})$$ such that 

\begin{equation} \label{i6}
\langle {\bf{x}},{\bf{y}} \rangle = \int_{\Omega_{B_{l^2}} \times \Omega_{B_{l^2}}}r_{{\bf{x}}}(\omega_1)r_{{\bf{y}}}(\omega_2) \lambda(d\omega_1, d\omega_2), \ \ \ ({\bf{x}},{\bf{y}}) \in B_{l^2} \times B_{l^2},
\end{equation}\\
and the Grothendieck constant  $ \mathcal{K}_G$  is the infimum of $\|\lambda\|_M$ over all the representations of the dot product by \eqref{i6}.
($\|\lambda\|_M = $ total variation norm of $\lambda$.)
\end{proposition}
\begin{proof} \  We first verify \eqref{i6}  $\Rightarrow$  \eqref{i1}.  \  Let $a = (a_{jk})$ be a finite scalar array, and denote by $\|a\|_{\mathcal{F}_2}$ the supremum on the right side of \eqref{i1}.  Assuming \eqref{i6}, let $({\bf{x}}_j)$ and  $({\bf{y}}_k)$  be arbitrary sequences of vectors in $B_{l^2}$, and then estimate
\begin{equation} \label{i5}
\begin{split}
\big |\sum_{j,k}a_{jk}   \langle{\bf{x}}_j, {\bf{y}}_k \rangle \big | &= \big |\sum_{j,k}a_{jk}  \int_{\Omega_{B_{l^2}} \times \Omega_{B_{l^2}}}r_{{\bf{x}}_j}(\omega_1)r_{{\bf{y}}_k}(\omega_2) \lambda(d\omega_1, d\omega_2)  \big | \\\\
& \leq \   \int_{\Omega_{B_{l^2}} \times \Omega_{B_{l^2}}} \big | \sum_{j,k}a_{jk}r_{{\bf{x}}_j}(\omega_1)r_{{\bf{y}}_k}(\omega_2)\big | \  |\lambda |(d\omega_1, d\omega_2) \\\\
& \leq \ \|a\|_{\mathcal{F}_2} \|\lambda\|_M.
\end{split}
\end{equation}
We thus obtain \eqref{i1} with $\mathcal{K}_G \leq \|\lambda\|_M$.

To verify  \eqref{i1} $\Rightarrow$ \eqref{i6}, we first associate with a finite scalar array $a = (a_{jk})$, and sequences of vectors $({\bf{x}}_j)$ and  $({\bf{y}}_k)$  in $B_{l^2}$,  the \emph{Walsh polynomial}
\begin{equation}
\hat{a} = \sum_{j,k} a_{jk} \ r_{{\bf{x}}_j} \otimes r_{{\bf{y}}_k}. 
\end{equation} 
Note that $\hat{a}$ is  a continuous function on $\Omega_{B_{l^2}}\times \Omega_{B_{l^2}}$, and 
\begin{equation}\label{i7}
 \|\hat{a}\|_{\infty} = \|a\|_{\mathcal{F}_2},
 \end{equation}
 where $\|\hat{a}\|_{\infty}$  is the supremum of $\hat{a}$ over $\Omega_{B_{l^2}}\times \Omega_{B_{l^2}}$.  Such polynomials are norm-dense in the space of continuous functions on $\Omega_{B_{l^2}}\times \Omega_{B_{l^2}}$ with spectrum in $R_{B_{l^2}} \times R_{B_{l^2}}$, which is denoted by  $C_{R_{B_{l^2}}  \times R_{B_{l^2}}}(\Omega_{B_{l^2}}\times \Omega_{B_{l^2}})$;  e.g., see   ~\cite[Ch. VII, Corollary 9]{blei:2001}.  Then,  \eqref{i1} becomes the statement that
\begin{equation}
\sum_{j,k} a_{jk} \ r_{{\bf{x}}_j} \otimes r_{{\bf{y}}_k}  \ \mapsto \ \sum_{j,k}a_{jk}   \langle{\bf{x}}_j, {\bf{y}}_k \rangle
\end{equation}
determines a bounded linear functional on $C_{R_{B_{l^2}}  \times R_{B_{l^2}}}(\Omega_{B_{l^2}}\times \Omega_{B_{l^2}})$, with norm bounded by $\mathcal{K}_G$.  Therefore, by the Riesz Representation theorem and by the Hahn-Banach theorem, there exists $\lambda \in M\big (\Omega_{B_{l^2}} \times \Omega_{B_{l^2}})$ such that \eqref{i6} holds, and $\|\lambda\|_M \leq \mathcal{K}_G$.

\end{proof}

\noindent

The implication $\eqref{i6} \Rightarrow \eqref{i1}$  provides a template for direct proofs of the Grothendieck inequality:  first establish an integral representation of the dot product, like the one in \eqref{i6}, and then verify the inequality by an "averaging" argument, as in \eqref{i5};  e.g., ~\cite{Lindenstrauss:1968}, \cite{rietz1974proof}, \cite{Blei:1977uq}, \cite{Krivine:1977}.    
Whereas there are other equivalent formulations of the inequality (e.g., see  ~\cite[\S1, \S2]{pisier2012grothendieck}), a representation of the dot product by an integral with uniformly bounded integrands is, arguably, the "closest" to  it.  

Building on  ideas in \cite{Blei:1977uq}, \cite{Blei:1979fk}, and \cite{Blei:1980fk}, we establish here integral representations that go a little further than  \eqref{i6}, and imply a little more than \eqref{i1}.\

\noindent
\subsection{Parseval-like formulas}
Without a requirement that integrands be uniformly bounded, integral representations of the $l^2$-dot product are ubiquitous and indeed easy to produce:   let  $\{f_{\alpha}: \alpha \in A\}$ be an orthonormal system in  $L^2(\Omega, \mu)$, where  $(\Omega, \mu)$   is a probability space, and let $U$  be the map from  $l^2$   into  $L^2(\Omega, \mu)$  given by
 \begin{equation}
U({\bf{x}}) = \sum_{\alpha} {\bf{x}}(\alpha) f_{\alpha}, \ \ \ {\bf{x}} \in l^2, 
 \end{equation}
 whence (Parseval's formula),
\begin{equation}
\langle {\bf{x}}, {\bf{y}} \rangle = \int_{\Omega} U({\bf{x}})\overline{U({\bf{y}}}) d \mu, \ \ \ {\bf{x}} \in l^2, \ {\bf{y}} \in l^2,
\end{equation}
and
\begin{equation}
\|U({\bf{x}})\|_{L^2} = \|{\bf{x}}\|_2, \ \ \ \ {\bf{x}} \in l^2.
\end{equation}\\
The map  $U$ is linear and (therefore) continuous.  

The Grothendieck inequality (via Proposition \ref{P0})  is ostensibly the more stringent assertion,  that there exist a probability space  $(\Omega, \mu)$ and a map $\Phi$ from $l^2$ into  $L^{\infty}(\Omega, \mu),$ such that
\begin{equation} \label{par}
\sum_{\alpha} {\bf{x}}(\alpha) {\bf{y}}(\alpha) = \int_{\Omega} \Phi({\bf{x}}) \Phi({\bf{y}}) d \mu, \ \ \ {\bf{x}} \in l^2, \ {\bf{y}} \in l^2,
\end{equation} 
and
\begin{equation}
\|\Phi({\bf{x}})\|_{L^{\infty}} \leq K \|{\bf{x}}\|_2, \ \ \ \ {\bf{x}} \in l^2,
\end{equation}\\ 
 for some $K > 1$. Specifically, assuming \eqref{i1}, we take  $\lambda \in M(\Omega_{B_{l^2}} \times \Omega_{B_{l^2}})$ supplied by Proposition \ref{P0}, and  let $\psi$ be the Radon-Nikodym derivative of $\lambda$ with respect to its total variation measure $|\lambda|$.   For ${\bf{x}} \in l^2$, define \begin{equation} \label{nor}
 \sigma \textbf{x} := \left \{
\begin{array}{lcc}
\textbf{x}/\|\textbf{x}\|_2 & \quad  \text{if} \quad  \textbf{x} \not= \textbf{0}  \\
\textbf{0} & \quad   \text{ if} \quad  \textbf{x} = \textbf{0},
\end{array} \right.
\end{equation}\\ 
and let $\pi_i$ ($i =1, 2$) denote the canonical projections from $\Omega_{B_{l^2}} \times \Omega_{B_{l^2}}$ onto $\Omega_{B_{l^2}}$,
\begin{equation*}
\pi_i(\omega_1,\omega_2) = \omega_i, \ \ \ (\omega_1,\omega_2) \in \Omega_{B_{l^2}} \times \Omega_{B_{l^2}}.
\end{equation*}
Define \\  $$\phi_i: l^2 \rightarrow L^{\infty}\big(\Omega_{B_{l^2}} \times \Omega_{B_{l^2}},  |\lambda|/\|\lambda\|_M\big)$$  by
 \begin{equation}\label{rep}
 \phi_i({\bf{x}}) =(\|\lambda\|_M \ \psi)^{\frac{1}{2}} \ \|{\bf{x}}\|_2 \  r_{\sigma{\bf{x}}} \circ \pi_i, \ \ \ \ {\bf{x}} \in l^2, \ \ \ i =1, 2,
 \end{equation}\\
where $r_{\sigma{\bf{x}}}$ are the coordinate functions given by \eqref{i4}, whence
\begin{equation}
\sum_{\alpha} {\bf{x}}(\alpha) {\bf{y}}(\alpha) = \int_{\Omega} \phi_1({\bf{x}}) \phi_2({\bf{y}}) d \mu, \ \ \ {\bf{x}} \in l^2, \ {\bf{y}} \in l^2,
\end{equation}
with $\Omega = \Omega_{B_{l^2}} \times \Omega_{B_{l^2}}$ and $\mu = |\lambda|/\|\lambda\|_M.$  Now let
\begin{equation*}
\Phi = \frac{\phi_1 + \phi_2}{2} + i \frac{\phi_1 - \phi_2}{2},
\end{equation*}
which, notably, is neither linear nor continuous. \\\

The Grothendieck inequality \emph{proper} implies  \eqref{par} with a probability measure $\mu$ on $\Omega = \Omega_{B_{l^2}} \times \Omega_{B_{l^2}}$,  and a non-linear map $$\Phi: l^2 \rightarrow L^{\infty}(\Omega_{B_{l^2}} \times \Omega_{B_{l^2}}, \mu),$$ which is not continuous with respect to the ambient topologies of  $l^2$ and $L^{\infty}(\Omega,\mu)$.   A question arises: can we do better?  That is, can a Parseval-like formula be derived with a more wieldy and canonical $(\Omega, \mu)$, and with  a map $\Phi$ that somehow "detects" the structures of its domain and range?

In the first part (\S \ref{s1} - \S \ref{Snonh}), given sets $X$ and $Y$, we study representations of scalar-valued functions of two variables, $x \in X$ and  $y \in Y$,  by integrals whose integrands are two-fold products of functions of one variable, $x \in X$ and $y \in Y$ separately.   If $X$ and $Y$ are merely sets,  then the Grothendieck inequality  --  seen as an integral representation of an inner product -- plays a natural role. If $X$ and $Y$ are topological spaces, then to play  that same role, the Grothendieck inequality needs an upgrade.  The main result (Theorem \ref{T2}) is an integral representation of the dot product 
\begin{equation} \label{par2}
\sum_{\alpha \in A} {\bf{x}}(\alpha) {\bf{y}}(\alpha) = \int_{\Omega_A} \Phi({\bf{x}}) \Phi({\bf{y}}) d \mathbb{P}_A, \ \ \ {\bf{x}} \in B_{l^2(A)}, \ {\bf{y}} \in B_{l^2(A)},
\end{equation} \\
where $A$ is an infinite set,  
\begin{align*}
 \Omega_A &:=  \{-1,1\}^A,\\\
  \mathbb{P}_A &:=  \text{uniform product measure (normalized Haar measure)},
 \end{align*} 
   and the map  
\begin{equation}\label{map}
\Phi: B_{l^2(A)} \rightarrow L^{\infty}(\Omega_A,\mathbb{P}_A)
\end{equation}\\
 is uniformly bounded, and also continuous in a prescribed sense.

 The proof of Theorem \ref{T2} is carried out in a setting of harmonic analysis on dyadic groups.  The construction of  $\Phi$ implicitly uses $\Lambda(2)$-\emph{uniformizability} \cite{Blei:1977uq}, a property of sparse spectral sets manifested here through the use of Riesz products.  Facts and tools drawn from harmonic analysis are reviewed as we move along (e.g.,  \S4). 
 
  In  \S \ref{Snonh}, a modification of the proof of Theorem \ref{T2} yields a Parseval-like formula for  $\langle {\bf{x}}, {\bf{y}}\rangle, \  {\bf{x}} \in l^p, \ {\bf{y}} \in l^q$, $1 \leq p \leq  2 \leq q \leq \infty$, \ $\frac{1}{p} + \frac{1}{q} = 1$. (Theorem \ref{T2pq}).

\subsection{Multilinear Parseval-like formulas} \ In the second part (\S \ref{s6} - \S \ref{s8}),  we study representations of  functions of $n$ variables ($n \geq 2$) by integrals whose integrands involve functions of $k$ variables,  $k < n$.  In this context we consider extensions of the (two-dimensional) Grothendieck inequality to dimensions greater than two. 

 First, following the "dual" view of  \eqref{i1} as the Parseval-like formula in \eqref{par}, we derive analogous formulas in arbitrary dimensions.  The general result  (Theorem \ref{T4}) is cast in a framework of \emph{fractional Cartesian products}.  It is proved by induction, with key steps provided by a variant of Theorem \ref{T2}. 

  We illustrate the multilinear formulas of Theorem \ref{T4} with two archetypal instances.  In the first,  we take a simple extension of the (bilinear) dot product in $l^2(A)$, 
\begin{equation}
\Delta_n({\bf{x}}_1, \ldots, {\bf{x}}_n) = \sum_{\alpha \in A} {\bf{x}}_1(\alpha) \cdots {\bf{x}}_n(\alpha), \ \ \ {\bf{x}}_i \in l^2(A), \ i = 1, \ldots, n.
\end{equation}
With no additional requirements, integral representations of $\Delta_n$, which extend the usual Parseval formula, arise typically as follows.  Consider the bounded linear map
\begin{equation}\label{repo}
U: l^2(A) \rightarrow L^2(\Omega_A,\mathbb{P}_A),
\end{equation} 
 defined by
\begin{equation}
U({\bf{x}}) = \sum_{\alpha \in A} {\bf{x}}(\alpha) r_{\alpha}, \ \ \ {\bf{x}} \in l^2(A),
\end{equation}
where the $r_{\alpha}$ are Rademacher characters,
\begin{equation}
r_{\alpha}(\omega) = \omega(\alpha), \ \ \ \omega \in \Omega_A := \{-1,1\}^A, \ \ \alpha \in A.
\end{equation}
Then,
\begin{equation} \label{dn}
\begin{split}
\Delta_n({\bf{x}}_1, \ldots, {\bf{x}}_n) &= \int_{(\omega_1, \ldots, \omega_{n-1}) \in \Omega_A^{n-1}} \bigg(\prod_{i=1}^{n-1}U({\bf{x}}_i)\big(\omega_i\big) \bigg)U({\bf{x}}_n)\big(\omega_1\cdots\omega_{n-1}\big)\ d\omega_1 \cdots d\omega_{n-1}\\\\
&= \big(U({\bf{x}}_1) \convolution \cdots \convolution U({\bf{x}}_n) \big)(\boldsymbol{\omega}_0), \ \ \ \ {\bf{x}}_1 \in l^2(A), \ldots, {\bf{x}}_n \in l^2(A),
\end{split}
\end{equation}\\
where the integral above is performed with respect to the $(n-1)$-fold product of the Haar measure $\mathbb{P}_A$  of the compact group $\Omega_A$, and is the $n$-fold convolution of $U({\bf{x}}_1), \ldots, U({\bf{x}}_n)$, evaluated at the identity element $\boldsymbol{\omega}_0 \in \Omega_A$,
$$\boldsymbol{\omega}_0(\alpha) = 1, \ \ \ \alpha \in A.$$
 Adding the requirement that integrands be uniformly bounded, we obtain (Lemma \ref{L6}) 
 \begin{equation} \label{gn}
\begin{split}
\Delta_n({\bf{x}}_1, \ldots, {\bf{x}}_n) &= \int_{(\omega_1, \ldots, \omega_{n-1}) \in \Omega_A^{n-1}} \bigg(\prod_{i=1}^{n-1}\Phi_n({\bf{x}}_i)\big(\omega_i\big) \bigg)\Phi_n({\bf{x}}_n)\big(\omega_1\cdots\omega_{n-1}\big)\ d\omega_1 \cdots d\omega_{n-1}\\\\
&= \big(\Phi_n({\bf{x}}_1) \convolution \cdots \convolution \Phi_n({\bf{x}}_n) \big)(\boldsymbol{\omega}_0), \ \ \ \ {\bf{x}}_1 \in l^2(A), \ldots, {\bf{x}}_n \in l^2(A),
\end{split}
\end{equation}\\
where 
 \begin{equation}\label{mapn}
\Phi_n: l^2(A) \rightarrow L^{\infty}(\Omega_A,\mathbb{P}_A)
\end{equation} \\ 
is $(l^2 \rightarrow L^2)$-continuous, and
\begin{equation} \label{mapnbd}
\|\Phi_n({\bf{x}})\|_{L^{\infty}} \leq K\|{\bf{x}}\|_2, \ \ \ {\bf{x}} \in l^2(A),
\end{equation} 
for an absolute constant $K > 0$.

The second instance is the trilinear functional $\eta$ on $l^2(A^2) \times l^2(A^2) \times l^2(A^2)$,
\begin{equation} \label{tri}
\eta({\bf{x}}, {\bf{y}},{\bf{z}}) = \sum_{(\alpha_1, \alpha_2, \alpha_3) \in A^3}{\bf{x}}(\alpha_1,\alpha_2) {\bf{y}}(\alpha_2,\alpha_3){\bf{z}}(\alpha_1,\alpha_3), \ \ \ \ ({\bf{x}}, {\bf{y}},{\bf{z}}) \in l^2(A^2) \times l^2(A^2) \times l^2(A^2).
\end{equation}\\\
To obtain a generic integral representation of $\eta$, with no restrictions on integrands, we take the bounded linear map $U_2$ from $l^2(A^2)$ into $L^2(\Omega_A^2, \mathbb{P}_A^2)$ given by\\
\begin{equation}
U_2({\bf{x}}) = \sum_{(\alpha_1,\alpha_2) \in A^2} {\bf{x}}(\alpha_1,\alpha_2)r_{\alpha_1}\otimes r_{\alpha_2}, \ \  \ \ {\bf{x}} \in l^2(A^2),
\end{equation}
and observe\\
\begin{equation}
\begin{split}
\eta({\bf{x}}, {\bf{y}},{\bf{z}}) = \int_{(\omega_1,\omega_2,\omega_3)\in \Omega_A^3}U_2({\bf{x}})(\omega_1,\omega_2) & U_2({\bf{y}})(\omega_2,\omega_3)U_2({\bf{z}})(\omega_1,\omega_3)\ d\omega_1d\omega_2 d\omega_3,\\\
& {\bf{x}} \in l^2(A^2), \  {\bf{y}} \in l^2(A^2), \  {\bf{z}} \in l^2(A^2).
\end{split}
\end{equation}\\\
Adding the requirement that integrands be uniformly bounded, we have 
\begin{equation*}
\begin{split}
\eta({\bf{x}}, {\bf{y}},{\bf{z}}) = \int_{(\omega_1,\omega_2,\omega_3)\in \Omega_A^3}\Phi_{(2,2)}({\bf{x}})(\omega_1,\omega_2) & \Phi_{(2,2)}({\bf{y}})(\omega_2,\omega_3)\Phi_{(2,2)}({\bf{z}})(\omega_1,\omega_3) \ d\omega_1 d\omega_2 d\omega_3,\\\
& {\bf{x}} \in l^2(A^2), \  {\bf{y}} \in l^2(A^2), \ {\bf{z}} \in l^2(A^2),\\\
\end{split}
\end{equation*}
so that  
\begin{equation}
\Phi_{(2,2)}: l^2(A^2) \rightarrow L^{\infty}(\Omega_A^2, \mathbb{P}_A^2)
\end{equation} 
 is  $\big(l^2(A^2) \rightarrow L^2(\Omega_A^2, \mathbb{P}_A^2)\big)$ - continuous, 
 and
 \begin{equation}
 \|\Phi_{(2,2)}({\bf{x}})\|_{L^{\infty}} \leq K^2\|{\bf{x}}\|_2, \ \ \ {\bf{x}} \in l^2(A^2),
 \end{equation}
where $K > 1$ is the absolute constant in \eqref{mapnbd}.  
The map $\Phi_{(2,2)}$ is obtained by a two-fold iteration of Lemma \ref{L6} in the base case $n = 2$.   (See Remark \ref{R7}.i.)

An $n$-linear version ($n \geq 3$) of the trilinear $\eta$ in \eqref{tri} is given by \\
\begin{equation}\label{hsp}
\begin{split}
\eta_n({\bf{x}}_1, {\bf{x}}_2, \ldots, {\bf{x}}_n) = \sum_{(\alpha_1, \alpha_2, \ldots, \alpha_n) \in A^n} {\bf{x}}_1(\alpha_1,\alpha_2) & {\bf{x}}_2(\alpha_2,\alpha_3) \cdots {\bf{x}}_n(\alpha_n,\alpha_1),\\\\
& ({\bf{x}}_1, {\bf{x}}_2, \ldots, {\bf{x}}_n) \in l^2(A^2)\times l^2(A)^2 \times \cdots \times  l^2(A^2),\\\
\end{split}
\end{equation}\\
whose subsequent integral representation is given by\\
\begin{equation} \label{hst}
\begin{split}
\eta_n&({\bf{x}}_1, {\bf{x}}_2,  \ldots, {\bf{x}}_n) \\\\
&= \int_{(\omega_1,\omega_2,\ldots,\omega_n)\in \Omega_A^n}\Phi_{(2,2)}({\bf{x}}_1)(\omega_1,\omega_2)  \Phi_{(2,2)}({\bf{x}}_2)(\omega_2,\omega_3)\cdots\Phi_{(2,2)}({\bf{x}}_n)(\omega_n,\omega_1) \ d\omega_1 d\omega_2 \cdots d\omega_n,\\\\
& \ \ \ \ \ \ \ \ \ \ \ \ \ \ \ \ \ \ \ \ ({\bf{x}}_1, {\bf{x}}_2, \ldots, {\bf{x}}_n) \in l^2(A^2)\times l^2(A) \times  \cdots \times  l^2(A^2).
\end{split}
\end{equation}\\\\
The $n$-linear functional in \eqref{hsp} together with its integral representation in \eqref{hst} play key roles in the Banach algebra of Hilbert-Schmidt operators.  (See Remark \ref{R10}.)\\\

\subsection{Projective boundedness and projective continuity} \ The Grothendieck inequality implies (via Proposition \ref{P0}) that if $\eta$ is any bounded bilinear functional on a Hilbert space $H$, then there
exist  bounded maps $\phi_1$ and $\phi_2$ from $B_H$ (= closed unit ball in $H$) into $L^{\infty}(\Omega,\mu)$, for some probability space $(\Omega, \mu)$, such that
\begin{equation} \label{bil}
\eta({\bf{x}},{\bf{y}}) = \int_{\Omega} \phi_1({\bf{x}})\phi_2({\bf{y}}) d \mu, \ \ \ \ \ ({\bf{x}},{\bf{y}}) \in B_H \times B_H.
\end{equation}
Theorem \ref{T2} provides $\phi_1$ and $\phi_2$ that are also ($H \rightarrow L^2$)-continuous.  

In the three-dimensional case, there exist nontrivial trilinear functionals on an infinite-dimensional Hilbert space that admit integral representations with uniformly bounded and continuous integrands (e.g.,  $\Delta_3$ in \eqref{dn} and  $\eta$ in \eqref{tri}).  But there exist also bounded trilinear functionals  $\eta$ (produced first in \cite{varopoulos1974inequality}) that fail a trilinear version of \eqref{i1} (with $\eta$  in the role of the dot product), and therefore do not admit integral representations with uniformly bounded integrands.  (See Remark \ref{R8}.)

 We are led to a definition:  an $n$-linear functional  $\eta$ on a Hilbert space $H$ is \emph{projectively bounded} or (the stronger property) \emph{projectively continuous}, if there exist a probability space $(\Omega, \mu)$ and bounded maps or, respectively (the more stringent requirement), bounded and  $\big(H \rightarrow L^2(\Omega,\mu)$\big)-continuous maps, $$\phi_i: B_H \rightarrow L^{\infty}(\Omega,\mu), \ \ \ i = 1, \ldots, n,$$ such that
\begin{equation}
\eta({\bf{x}}_1, {\bf{x}}_2, \ldots, {\bf{x}}_n) = \int_{\Omega} \phi_1({\bf{x}}_1) \cdots \phi_n({\bf{x}}_n) \ d \mu, \ \ \ \ \ ({\bf{x}}_1, {\bf{x}}_2, \ldots, {\bf{x}}_n) \in (B_H)^n.
\end{equation}
(See Definition \ref{Q7}.) 

 All bounded linear functionals and all bounded bilinear functionals on a Hilbert space are projectively continuous, and \emph{a fortiori} projectively bounded (Proposition \ref{P8}, Theorem \ref{T2}).  But for $n > 2$,  how to distinguish between bounded, projectively bounded, and projectively continuous $n$-linear functionals are open-ended questions.  We tackle these here through the analysis of kernels of multilinear functionals.  

Let $H$ be a Hilbert space, and let $A$ be an orthonormal basis in it.   Given a bounded $n$-linear functional $\eta$ on  $H$, we consider its \emph{kernel relative to} $A$,
 \begin{equation} \label{ker}
 \theta_{A, \eta}(\alpha_1, \ldots,\alpha_n) := \eta(\alpha_1, \ldots, \alpha_n), \ \ \ \ \ (\alpha_1, \ldots, \alpha_n) \in A^n,
 \end{equation}
 and ask:
\begin{question} \label{Q0}
  What properties of $\theta_{A, \eta}$ imply that $\eta$ is projectively continuous,  or projectively bounded?
  \end{question} 
\noindent
If $\eta$ is projectively bounded, then its kernel relative to a basis $A$ can be represented as an integral of $n$-fold products of uniformly bounded functions of one variable, but whether this condition alone is also sufficient is an open question (Problem \ref{q11}). 

 Building on multilinear Parseval-like formulas (Theorem \ref{T4}),  we show that if the support of  $\theta_{A,\eta}$ is a \emph{fractional Cartesian product} of a certain type, and $\theta_{A,\eta}$ has an integral representation with uniformly bounded integrands, then $\eta$ is projectively continuous  (Theorem \ref{T3}).  In this case, the respective spaces of projectively bounded and projectively continuous functionals are equal --  distinguished, possibly, by different numerical constants (Problem \ref{num}).   I do not know whether this is always true, that a multilinear functional is projectively bounded if and only if it is projectively continuous (Problem \ref{q10}).

We list some unanswered questions  in  \S \ref{s11}.   Some are stated in prior sections, and others will almost surely occur to the reader in the course of the narrative.  The subject is open, and also open-ended, with loose ends and different directions to follow.  

Unless stated otherwise, the underlying scalars will be the \emph{complex numbers}.  There will be places in the work where distinctions are made between the \emph{real} field and the \emph{complex} field, and other places where \emph{reals}, and only  \emph{reals} are used; these places are clearly marked.  

The essential prerequisites for the work are relatively minimal: basic knowledge of classical functional analysis and classical harmonic analysis should suffice.  To the extent possible, I tried to be inclusive and self-contained (with apologies to experts). \

\subsection{A personal note and acknowledgements} \ My involvement with the Grothendieck inequality began in 1975, in connection with its roles in harmonic analysis.  Since then, over the decades, I gradually meandered away from the subject.  In 2005, I learnt of  the interest in the Grothendieck inequality in the computer science community, and in particular of the results in ~\cite{Charikar:2004}.  My own interest in the inequality was rekindled, and reinforced by a subsequent result in ~\cite{Alon:2006}, that the Grothendieck-like inequalities in ~\cite{Charikar:2004}  were in a certain sense optimal.  The optimality, which at first blush appeared counter-intuitive to me, made crucial use of a phenomenon that had been first verified in  ~\cite{Kashin:2003}. (See Remark \ref{R2}.ii.)  I learnt more about these results from Noga Alon's lecture, and also from separate conversations with Assaf Naor and Boris Kashin, at the Lindenstrauss-Tzafriri retirement conference in June 2005.  The integral representation formula in \eqref{e20}  was then derived, without the continuity of $\Phi$, and noted specifically as a counterpoint to the result in ~\cite{Kashin:2003}.  Soon after my seminar at IAS in February 2006, where the aforementioned counterpoint was presented, I received an email from Noga Alon, that a similar formula could be deduced also from the proof of the Grothendieck inequality in ~\cite{Krivine:1977}. 

Two and a half years later, following a seminar about integral representations of multilinear functionals (Paris, June 2008), Gilles Pisier pointed out to me the implication $\eqref{i1} \Rightarrow \eqref{i6}$ (with a proof different from the argument given here), underlining that the integral representation of the dot product in \eqref{i6} and the Grothendieck inequality are in fact equivalent. These two post-seminar remarks, by Alon and Pisier, sent me back to the drawing board, to study and further analyze integral representations of general functions, and indeed in large part determined the direction of this work. 

Finally, I thank my son David Blei, first for pointing out to me, in 2005,  the interest in the Grothendieck inequality in the computer science community, and then for his help with the optimization problems in \S \ref{sf} concerning "best" constants.

 \section{\bf{Integral representations: the case of discrete domains}} \label{s1} 
 \subsection{First question}\ \  
  Whether it is feasible to represent a given function in terms of simpler functions is a recurring theme that appears in various guises throughout mathematics.  The precise formulation of the problem of course depends on the context.  Here we consider the issue in a quasi-classical setting of functional analysis.
  
  We begin with a question that -- mainly during the 1960's and 1970's -- motivated studies of tensor products in harmonic analysis;  e.g.,~\cite{Varopoulos:1967},~\cite[Ch. VIII]{Katznelson:1976},~\cite[Ch. 11]{Graham&McGehee:1979}.

\begin{question} \label{Q1} Let $X$ and $Y$ be sets, and let $f$ be a bounded scalar-valued function of two variables,  $x\in X$ and $y\in Y.$   Are there families of scalar-valued functions of one variable,  $\{g_{\omega}\}_{\omega \in \Omega}$ and $\{h_{\omega}\}_{\omega \in \Omega}$  defined respectively on $X$ and $Y,$  such that 
\begin{equation} \label{Q1-1}
\sum_{\omega \in \Omega} \|g_{\omega}\|_{\infty}\|h_{\omega}\|_{\infty} < \infty,
\end{equation}
and
\begin{equation} \label{Q1-2}
f(x,y) = \sum_{\omega \in \Omega} g_{\omega}(x)h_{\omega}(y), \ \ \ \  (x,y) \in X \times Y?
\end{equation}
($\|\cdot\|_\infty$ denotes, here and throughout, the usual sup-norm over an underlying domain.)
\end{question}

\noindent We refer to the functions $g_{\omega}$ and  $h_{\omega},$ indexed by $\omega \in \Omega$ and defined respectively on $X$ and $Y,$  as \emph{representing functions} of $f.$ \    In the next section  $X$ and $Y$ will have additional structures, but at this point, and until further notice, these domains are viewed merely as sets.  The constraint  in \eqref{Q1-1}  implies that the sum on the right side of \eqref{Q1-2} converges absolutely,  and also that the indexing set $\Omega$  may as well be at most countable.  Later we will consider more general representations and constraints, wherein sums in  \eqref{Q1-1} and \eqref{Q1-2} are replaced by integrals, and indexing sets are measure spaces.  If constraints on representing functions are removed altogether, then the answer to Question \ref{Q1} will always be affirmative (and also uninteresting):  for, in the absence of constraints, given a scalar-valued function $f$ on  $X \times Y$, we can take the indexing set  $\Omega$ to be $Y$, and then let
\begin{equation*}
  g_y(x) = f(x,y),  \ \ \ x \in X, \ y \in Y,
  \end{equation*}
   and for $y  \in Y$ and  $y^{\prime} \in Y,$
   \begin{equation*}
 h_{y^{\prime}}(y) = \left \{
\begin{array}{ccc}
0 & \quad  \text{if} \quad  y^{\prime}  \not= y  \\
1 & \quad   \text{ if} \quad  y^{\prime} = y.
\end{array} \right.
\end{equation*} \\\

Nowadays it is practically folklore  -- arguably originating in Littlewood's classic paper ~\cite{Littlewood:1930}  --  that if  $X$ and $Y$  are infinite, then there is an abundance of functions that cannot be represented by \eqref{Q1-2} under the constraint in \eqref{Q1-1}.  To illustrate how such functions can arise,  we take
\begin{equation} \label{V1}
  f(x,k) = e^{ixk}, \ \ \ (x,k) \in[0,2\pi] \times \mathbb{Z}.
     \end{equation}
An affirmative answer here to Question \ref{Q1} would mean that  there exist families of representing functions $\{g_n\}_{n\in\mathbb{N}}$ and $\{h_n\}_{n\in\mathbb{N}}$ defined respectively on $[0,2\pi]$  and $\mathbb{Z},$ and indexed by $\mathbb{N}$,  such that

  \begin{equation}\label{A1}
 e^{ixk} = \sum_{n  \in \mathbb{N}} g_n(x) h_n(k), \ \ \ \  (x,n) \in [0,2\pi] \times \mathbb{Z},
 \end{equation}
 and
 \begin{equation}\label{A2}
 \sum_{n  \in \mathbb{N}}  \|g_n\|_{\infty} \|h_n \|_{\infty} = c < \infty.
 \end{equation}
 With this assumption,  if a finite scalar array  $(a_{jk})_{j,k}$  satisfies\\\
 \begin{equation} \label{e1}
 \sup \big \{\big |\sum_{j,k}a_{jk}s_jt_k \big |: |s_j| \leq 1, |t_k| \leq 1 \big \} \leq 1,
 \end{equation}
then for  all $x_j \in [0, 2\pi], \ j = 1, \ldots,$
 \begin{equation} \label{e2}
 \big |\sum_{j,k}a_{jk}e^{ix_jk} \big | \leq c.
\end{equation}
 Now observe that for every positive integer $N$, the array (a rescaled Fourier matrix)
\begin{equation} \label{e17}
a_{jk} = N^{-\frac {3}{2}}e^{\frac{-2\pi i j k}{N}}, \ \ \ (j,k) \in[N]^2, 
\end{equation}
 satisfies \eqref{e1}.  Taking $N > c^2$ and $x_j = 2\pi i j/N$ for $ j\in [N]$, we have
\begin{equation*}
\sum_{j,k = 1}^N a_{jk}e^{ix_jk} =  N^{\frac{1}{2}} > c,
\end{equation*}
which contradicts \eqref{e2}, thus proving that $f$ in \eqref{V1}  cannot be represented by \eqref{A1} under the constraint in \eqref{A2}. 

The preceding argument, verifying a negative answer to Question \ref{Q1},  made implicit use of a duality between two  norms.  The first norm, $\|f\|_{\mathcal{V}_2(X \times Y)}$ for a given  scalar-valued function  $f$ on $X\times Y$, is defined to be the infimum of the left side of \eqref{Q1-1} taken over all representations in \eqref{Q1-2}.  (Here and throughout, in the absence of ambiguity, the underlying $X$ and $Y$ will be suppressed from the notation; e.g.,  $\|f\|_{\mathcal{V}_2}$ will stand for $\|f\|_{\mathcal{V}_2(X \times Y)}$.)  With this norm and point-wise multiplication on  $X \times Y,$  the resulting space    
\begin{equation*}
\mathcal{V}_2(X \times Y) = \big\{f:  \|f\|_{\mathcal{V}_2} < \infty\big\}
\end{equation*} \\
is a Banach algebra.  ($\mathcal{V}$ is for Varopoulos; e.g., see ~\cite{Varopoulos:1967},  ~\cite[ Ch. 11]{Graham&McGehee:1979}.)   The second norm, $\|a\|_{\mathcal{F}_2}$ for a given  scalar array  $a = (a_{xy})_{x \in X, y \in Y}$, is defined as the supremum of 
\begin{equation*}
\big |\sum _{(x,y) \in S \times T}a_{xy} g(x)h(y) \big|
\end{equation*}
taken over all finite rectangles $S \times T \subset  X\times Y$,  $g \in B_{l^{\infty}(X)}$  and $h \in B_{l^{\infty}(Y)}.$  (Here and throughout,  $B_E$  denotes the closed unit ball in a normed linear space $E$, and $l^{\infty}(D)$ denotes the space of bounded scalar-valued functions on a domain $D$.)  Equipped with the $\mathcal{F}_2$-norm and point-wise multiplication on $X \times Y$,
\begin{equation} \label{b1}
\mathcal{F}_2(X \times Y) := \big\{a = (a_{xy})_{x \in X, y \in Y}: \|a\|_{\mathcal{F}_2} < \infty \big\}
\end{equation}
is a Banach algebra.  ($\mathcal{F}$ is for Fr\'echet; e.g., see ~\cite{Frechet:1915},  ~\cite[Ch. I]{blei:2001}.)
 Within the broader context of topological tensor products, the $\mathcal{V}_2$-norm  and the $\mathcal{F}_2$-norm are instances, respectively,  of \textit{projective} and \textit{injective} cross-norms; e.g., see~\cite{Schatten:1950},~\cite{Grothendieck:1955}.     

 The duality between the $\mathcal{V}_2$-norm and the $\mathcal{F}_2$-norm is expressed by Lemma \ref{L1} below, and is at the foundation of a large subject.  Noted and investigated first by Schatten in the 1940's, e.g., ~\cite{Schatten:1943}, this duality was extended and expansively studied by Grothendieck during the 1950's, e.g., ~\cite{Grothendieck:1956}, ~\cite{Grothendieck:1955}.  Since the mid-1960's, Lemma \ref{L1}  has reappeared in various forms and settings, and has become folklore.  I learnt it nearly four decades ago in a specific context of harmonic analysis, e.g., ~\cite{Blei:1975}. 
\begin{lemma} \label{L1}
If  $f \in \mathcal{V}_2(X \times Y)$, then

\begin{equation} \label{e3}
 \|f\|_{\mathcal{V}_2}  = \sup \big \{ \big |\sum_{(x,y) \in S\times T}a_{xy}f(x,y)\big |:  a \in B_{\mathcal{F}_2}, \  \text{\rm{finite rectangles}} \  \ S \times T \subset  X\times Y \big \}. \ \ \  \ \ \ \ \ \ 
  \end{equation}
\end{lemma} 

\noindent  (For proof, see Remark \ref{R1}.ii.)
\ \

\subsection{Second question} \ \    The converse to Lemma \ref{L1} is false:   if $X$ and $Y$  are infinite, then there exist $f \in l^{\infty}(X \times Y)$ for which the right side of \eqref{e3} is finite, but $f \notin \mathcal{V}_2(X \times Y)$.  (How these $f$ arise is briefly explained in Remark \ref{R1}.iii below.) 
However, if we broaden Question \ref{Q1} (but  keep to its original intent),  then we are led exactly to those functions $f$  for which the right side of \eqref{e3} is finite.

\begin{question} \label{Q2}
 Let  $f \in l^{\infty}(X \times Y)$.  Can we find a probability space $(\Omega,  \mu)$ and representing functions indexed by it, $\{g_{\omega}\}_{ \in \Omega}$ and $\{h_{\omega}\}_{\omega \in \Omega}$  defined respectively on $X$ and $Y,$  such that for every  $(x,y) \in X \times Y,$

\begin{equation} \label{e7}
\omega \mapsto g_{\omega}(x), \ \ \ \omega \mapsto h_{\omega}(y), \ \ \  \omega \mapsto \|g_{\omega}\|_{\infty}, \ \ \  \omega \mapsto \|h_{\omega}\|_{\infty}, \ \ \ \omega \in \Omega,
\end{equation} \\\
determine $\mu$-measurable functions on $\Omega,$ 

\begin{equation} \label{e4}
\int_{\Omega} \|g_{\omega}\|_{\infty} \|h_{\omega}\|_{\infty} \mu(d\omega) < \infty,
\end{equation}\\\
and

\begin{equation} \label{e5}
f(x,y) = \int_{\Omega}g_{\omega}(x)h_{\omega}(y) \mu(d \omega)?
\end{equation}
(We refer to  the probability space $(\Omega, \mu)$ as an indexing space, and to the probability measure $\mu$ as an indexing measure.)
\end{question}

\noindent  There exist functions of two variables that cannot be represented by \eqref{e5} under the constraint in \eqref{e4};  e.g.,  the function $f$ in \eqref{V1}, with the same proof.  Question \ref{Q2} obviously subsumes Question \ref{Q1}:  an affirmative to Question \ref{Q1} implies an affirmative to Question \ref{Q2}.  But the converse is false, as per Remark \ref{R1}.iii  below.  To formalize matters,  given  $f \in l^{\infty}(X \times Y),$  we define $\|f\|_{\tilde{\mathcal{V}}_2}$  to be the infimum of the left side of \eqref{e4} over all representations of $f$ by \eqref{e5}, and then let
\begin{equation*}
\tilde{\mathcal{V}}_2(X \times Y) = \{f \in l^{\infty}(X \times Y): \|f\|_{\tilde{\mathcal{V}}_2} < \infty \}.
\end{equation*}
One can verify directly that  $\| \cdot \|_{\tilde{\mathcal{V}}_2}$ is a norm,  and that $(\tilde{\mathcal{V}}_2(X \times Y), \| \cdot \|_{\tilde{\mathcal{V}}_2})$  is a Banach algebra with point-wise multiplication on $X \times Y$.  (E.g., see Proposition \ref{P2} in the next section.)  This follows also from a characterization of  $\tilde{\mathcal{V}}_2(X \times Y) $ as the dual space of $\mathcal{F}_2(X \times Y).$   To state and prove this characterization,  we use an alternate definition of $\mathcal{F}_2(X \times Y)$ cast in a framework of harmonic analysis. 

 Given a set $A$, we define the Rademacher system indexed by $A$,
  \begin{equation*}
  R_A =\{r_{\alpha} \}_{\alpha \in A},
  \end{equation*}
  to be the set of coordinate functions on $\{-1,1\}^A := \Omega_{A}$,

  \begin{equation*}
 r_{\alpha}(\omega) = \omega(\alpha), \ \ \ \omega \in \{-1,1\}^A , \ \ \ \alpha \in A.
  \end{equation*}
 We view the product space $\Omega_A$ as a compact Abelian group with the product topology and coordinate-wise multiplication, and  view $R_A$ as  a set of independent characters on it.  (E.g., see~\cite[ p. 139]{blei:2001}.)
 
 Next we consider the set of characters  $$R_{X} \times R_{Y} = \{r_x \otimes  r_y \}_{(x,y) \in X \times Y}$$  on the compact Abelian group $\Omega_{X} \times \Omega_{Y}$, where $r_x \otimes r_y$ is the $\{-1,1\}$-valued function on  $\Omega_{X} \times \Omega_{Y}$, defined by $$r_x \otimes r_y(\omega_1,\omega_2) = r_x(\omega_1)r_y(\omega_2), \ \ \ (\omega_1,\omega_2) \in \Omega_{X} \times \Omega_{Y}.$$  Given a scalar array $a =(a_{xy})_{(x,y) \in X \times Y}$,  we  identify it formally with the Walsh series
   \begin{equation*}
   \hat{a}  \ \sim  \sum_{(x,y) \in X \times Y}a_{xy} r_x \otimes r_y
   \end{equation*}
   (a distribution on $\Omega_{X} \times \Omega_{Y}$), and observe that
   \begin{equation} \label{e6}
   a \in \mathcal{F}_2(X \times Y) \ \Longleftrightarrow \  \hat{a} \in  C_{R_{X} \times R_{Y}}(\Omega_{X} \times \Omega_{Y}), 
   \end{equation}
  where $C_{R_{X} \times R_{Y}}(\Omega_{X} \times \Omega_{Y})$ =   \{continuous functions on $\Omega_{X} \times \Omega_{Y}$ with spectra in $R_{X} \times R_{Y}$\};  e.g., see~\cite[Ch. IV,  Ch. VII]{blei:2001}, and also \S \ref{s3} in this work.  Specifically, 
   \begin{equation} \label{e8}
    \|\hat{a}\|_{\infty} \leq  \|a\|_{\mathcal{F}_2} \leq  4 \|\hat{a}\|_{\infty} .
    \end{equation}
    Therefore, the dual space of  $\mathcal{F}_2(X \times Y)$ can be identified (via \eqref{e6} and Parseval's formula) with the dual space of $C_{R_X \times R_Y}(\Omega_{X} \times \Omega_{Y})$, which (via the Riesz representation theorem and the Hahn-Banach theorem) is the algebra of restrictions to $R_X \times R_Y$ of transforms of regular complex measures on $\Omega_{X} \times \Omega_{Y}$.  This algebra is denoted by 
    \begin{equation*} 
  B(R_X \times R_Y)  := \widehat{M}(\Omega_{X} \times \Omega_{Y})/\{\hat{\lambda} \in \widehat{M}(\Omega_{X} \times \Omega_{Y}): \hat{\lambda} = 0 \ \text{on} \ R_X \times R_Y\}.  
    \end{equation*}\
     
    \begin{proposition} \label{P1}
 
  \begin{equation*}
    \tilde{\mathcal{V}}_2(X \times Y) = B(R_{X} \times R_{Y}).
   \end{equation*}\     
    \end{proposition}
  \begin{proof}
    If $f \in B(R_{X} \times R_{Y}),$  then (by definition) there exists a regular complex measure $\lambda$ on $ \Omega_{X} \times \Omega_{Y}$ such that
   \begin{equation*}
   f(x,y) = \int_{\Omega_{X} \times \Omega_{Y}}r_x(\omega_1)r_y(\omega_2) \lambda(d\omega_1,d\omega_2), \ \  (x,y) \in X \times Y,
  \end{equation*} 
   which implies $f \in \tilde{\mathcal{V}}_2(X \times Y)$:   the indexing space is $\big(\Omega_{X} \times \Omega_{Y},\frac{ |\lambda|}{\|\lambda\|}\big)$; \  $|\lambda|$ is the \textit{total variation measure} of $\lambda$; \  $\|\lambda\| = |\lambda|(\Omega_X \times \Omega_Y)$, and representing functions are given by
 \begin{align*}
  g_{\omega_1\omega_2}(x) & = r_x(\omega_1) \\
  h_{\omega_1\omega_2}(y)  & =  \|\lambda\| \frac{d \lambda}{d|\lambda|}(\omega_1,\omega_2) r_y(\omega_2), \ \ \ \ (\omega_1,\omega_2) \in   \Omega_{X} \times \Omega_{Y}, \ \ (x,y) \in X \times Y.
  \end{align*} 
  
   Suppose  $f \in \tilde{\mathcal{V}}_2(X \times Y)$.  Then, there exist an indexing space $(\Omega, \mu)$ and families of representing functions  $\{g_{\omega}\}_{ \in \Omega}$ and $\{h_{\omega}\}_{\omega \in \Omega}$  defined respectively on $X$ and $Y,$  such that \eqref{e7},  \eqref{e4},  and \eqref{e5} hold.  Let $\hat{a}$ be a Walsh polynomial in $C_{R_{X}\times R_{Y}}(\Omega_{X} \times \Omega_{Y})$,
   \begin{equation*}
   \hat{a} = \sum_{(x,y) \in X \times Y} a_{xy}r_x \otimes r_y,
   \end{equation*}   
  where $(a_{xy})_{(x,y) \in X \times Y}$ is a scalar array with finite support in $X \times Y$.  Then, by \eqref{e5}, \eqref{e4}, and  \eqref{e8},
  \begin{equation}\label{b2}
   \begin{split}
  \big|\sum_{(x,y) \in X \times Y} a_{xy}f(x,y)\big|  & = \big|\sum_{(x,y) \in X \times Y} a_{xy} \int_{\Omega} g_{\omega}(x) h_{\omega}(y) \mu(d\omega)\big| \\
   & \leq \int_{\Omega}\ \big{|}\sum_{(x,y) \in X \times Y}a_{xy}g_{\omega}(x) h_{\omega}(y)\big{|}\mu(d\omega) \\
   & \leq 4\|\hat{a}\|_{\infty} \int_{\Omega} \|g_{\omega}\|_{\infty}\|h_{\omega}\|_{\infty} \mu(d\omega) \ \leq \ 4\|\hat{a}\|_{\infty} \|f\|_{\mathcal{V}_2}.
   \end{split}
   \end{equation}
  Because Walsh polynomials are norm-dense in $C_{R_{X}\times R_{Y}}(\Omega_{X} \times \Omega_{Y})$, the estimate in \eqref{b2} implies that $f$ determines a bounded linear functional on $C_{R_X \times R_Y}(\Omega_{X} \times \Omega_{Y})$, and hence $f \in B(R_{X} \times R_{Y})$.
   
    \end{proof}

\begin{remark} \label{R1} \ \ \  \\\
\em{\textbf{i.} \  The simple proof of Proposition \ref{P1} is similar to the proof of Proposition \ref{P0}.  The compact Abelian group $ \{-1,1\}^A$  in the proof above can be replaced, with nearly identical effect, by the compact Abelian group $$\{e^{it}: t \in [0, 2\pi) \}^A$$ (equipped with coordinate-wise multiplication, and the usual topology on each coordinate).  The Rademacher system becomes the Steinhaus system (cf.~\cite[p. 29]{blei:2001}), and the norm equivalence in \eqref{e8} becomes the isometry
$$\|\hat{a}\|_{\infty} = \|a\|_{\mathcal{F}_2}.$$
A modified proof of Proposition \ref{P1} in this setting yields
\begin{equation*}
   \|f\|_{\tilde{\mathcal{V}}_2} = \sup \big \{ \big |\sum_{(x,y) \in S \times T} a_{xy}f(x,y) \big |:   \|a\|_{\mathcal{F}_2} \leq 1, \  \text{finite}  \ S \times T \subset X \times Y \big \}.   
   \end{equation*}\
\noindent
\textbf{ii.} \  By an elementary argument (involving measure theory),  
\begin{equation} \label{e9}
\|f\|_{ \mathcal{V}_2(X \times Y)} = \|f\|_{ \tilde{\mathcal{V}}_2(X \times Y)},  \ \ \  f \in  \mathcal{V}_2(X \times Y),
\end{equation}
and hence Lemma \ref{L1}.  The equality of norms in \eqref{e9} can be observed also by noting that $\mathcal{V}_2(X \times Y)$  equals  $B_d(R_X \times R_Y)$, the algebra of restrictions to $R_X \times R_Y$ of transforms of \textit{discrete} measures on $\Omega_{X} \times \Omega_{Y}.$ (See~\cite{Varopoulos:1969}.)\\\

\noindent
\textbf{iii.} \   In the language of harmonic analysis (via Proposition \ref{P1}),  the proper inclusion 
\begin{equation} \label{e10} 
 \tilde{\mathcal{V}}_2(X \times Y) \subsetneq l^{\infty}(X \times Y)
 \end{equation}
 for infinite $X$ and $Y$ is the assertion that $R_{X} \times R_{Y}$ is not a Sidon set, a classical fact that can be verified in several ways;  e.g., via the example in \eqref{V1}.  Also in this language, the proper inclusion 
  \begin{equation*}
  \mathcal{V}_2(X \times Y) \subsetneq  \tilde{\mathcal{V}}_2(X \times Y)
  \end{equation*}
  becomes the assertion
  \begin{equation}\label{e11}
  B_d(R_X \times R_Y) \subsetneq B(R_X \times R_Y),
  \end{equation}
 which is proved in ~\cite{Varopoulos:1969}, via~\cite{Varopoulos:1968}.  The proper inclusion in \eqref{e11} can be deduced also from \eqref{e10} and a characterization of Sidonicity found in~\cite[p. 118]{Bourgain:1984}. \\\\
 \noindent 
 \textbf{iv.} \  By Proposition \ref{P1}, if $f \in \tilde{\mathcal{V}}_2(X \times Y),$ then there exists $\lambda \in M(\Omega_X \times \Omega_Y)$ such that
\begin{equation*}
f(x,y) = \hat{\lambda}(r_x \otimes r_y), \ \ \ (x,y) \in X \times Y.
\end{equation*}\\\
This implies that in the definition of the $\|\cdot\|_{\tilde{\mathcal{V}}_2}$-norm, the constraint in \eqref{e4} (absolute integrability) can be replaced by the stronger condition ($L^{\infty}$-boundedness)
  \begin{equation} \label{e12}
  \sup_{x \in X, y \in Y} \esssup_{\omega \in \Omega}|g_{\omega}(x) h_{\omega}(y)| < \infty. 
  \end{equation}
    Specifically, the $\|\cdot\|_{\tilde{\mathcal{V}}_2}$-norm is equivalent to the norm obtained by taking the infimum of the left side of \eqref{e12} over all representations in \eqref{e5}.}
\end{remark} 
\  
   \subsection{Third question and Grothendieck's  \textit{th\'eor\`eme fondamental}} \ \    Next we replace \eqref{e4} in Question \ref{Q2} with a weaker constraint:

  \begin{question} \label{Q3} Let  $f \in l^{\infty}(X \times Y)$.  Can we find a general measure space  $(\Omega,  \mu)$, and representing functions indexed by it,  $\{g_{\omega}\}_{ \in \Omega}$ and $\{h_{\omega}\}_{\omega \in \Omega}$ defined respectively on $X$ and $Y,$  so that for every $(x,y) \in X \times Y,$
   \begin{equation*}
\omega \mapsto g_{\omega}(x), \ \ \ \omega \mapsto h_{\omega}(y), \ \ \ \ \omega \in \Omega,
\end{equation*} \\\
determine elements in  $L^2(\Omega,\mu),$ 
\begin{equation} \label{e13}
\sup_{x\in X, y\in Y} \left( \int_{\Omega} |g_{\omega}(x)|^2 \mu(d \omega) \right)^{\frac{1}{2}}   \left( \int_{\Omega} |h_{\omega}(y)|^2 \mu(d \omega) \right)^{\frac{1}{2}}  < \infty,
\end{equation}
and   
\begin{equation} \label{e14}
f(x,y) = \int_{\Omega}g_{\omega}(x)h_{\omega}(y) \mu(d \omega)?
\end{equation}
\end{question}

\noindent 
 The constraint in \eqref{e13}  insures that the right side of \eqref{e14} is well defined, and appears weaker  (via Remark \ref{R1}.iv above)  than the condition in \eqref{e4}.  The two conditions  \eqref{e4}  and  \eqref{e13}  are, respectively,  the "maximal" and "minimal" constraints that guarantee convergence of the respective integral representations in \eqref{e5} and \eqref{e14}.  (See also Remark \ref{R2}.i below.)  

Grothendieck's celebrated theorem, dubbed  \textit{le th\'eor\`eme fondamental de la th\'eorie metrique des produit tensoriels}  in ~\cite[p. 59]{Grothendieck:1956},  is in effect the statement that the aforementioned constraints are equivalent.  That is, for every function on $X \times Y,$ the respective answers to Questions \ref{Q2} and \ref{Q3} are identical:   a function of  two variables can be represented by \eqref{e5} under the constraint in \eqref{e4}  if and only if it can be represented by \eqref{e14} under the constraint in \eqref{e13}.  
 To state this precisely, given $f \in l^{\infty}(X \times Y)$,  we define  $\|f\|_{\mathcal{G}_2}$  to be the infimum of the left side of \eqref{e13} over all representations of $f$ by \eqref{e14}, and let 
\begin{equation*}
\mathcal{G}_2(X \times Y) = \{f \in l^{\infty}(X \times Y): \|f\|_{\mathcal{G}_2} < \infty \}.
\end{equation*}
(We omit the verification that $(\mathcal{G}_2(X \times Y), \|\cdot\|_{\mathcal{G}_2})$ is a Banach algebra;  see the theorem below, and also Proposition \ref{P2} in the next section.)
\vskip 0.5 cm

\begin{theorem}  [\text{a version of Grothendieck's} \textit{th\'eor\`eme fondamental}] \label{T1}
\begin{equation*}
\mathcal{G}_2(X \times Y) = \tilde{\mathcal{V}}_2(X , Y).
\end{equation*}
\end{theorem}
\vskip 0.5 cm
\noindent To verify  $\tilde{\mathcal{V}}_2(X , Y) \subset  \mathcal{G}_2(X \times Y),$  note that if $f \in \tilde{\mathcal{V}}_2(X \times Y),$ then (by Proposition \ref{P1}) there exists $\lambda \in M(\Omega_X \times \Omega_Y),$ such that
\begin{equation*}
f(x,y) = \hat{\lambda}(r_x \otimes r_y), \ \ \ (x,y) \in X \times Y,
\end{equation*}
and (by Remark \ref{R1}.iv) we have the norm estimate
\begin{equation*}
\|f\|_{\mathcal{G}_2} \leq \|f\|_{\tilde{\mathcal{V}}_2},  \ \ \ f \in  \tilde{\mathcal{V}}_2(X , Y).
\end{equation*}
  The reverse inclusion
\begin{equation*}
\mathcal{G}_2(X \times Y) \subset  \tilde{\mathcal{V}}_2(X , Y), 
\end{equation*}
or the assertion equivalent to it, that 
\begin{equation} \label{e15}
 \|f\|_{\mathcal{G}_2} \leq 1 \ \Rightarrow \ \|f\|_{\tilde{\mathcal{V}}_2} \leq K, 
\end{equation}
 for a universal constant $1 < K < \infty,$  is the essence of Grothendieck's theorem.  

 Focusing on \eqref{e15}, observe first that  $\|f\|_{\mathcal{G}_2} \leq 1$ means that  there exist maps from $X$ and $Y$  into the unit ball  $B_H$  of a Hilbert space $H$,
  \begin{equation}\label{e27}
\begin{array} {cc}
 x \mapsto \textbf{u}_x, &  x\in X, \ \textbf{u}_x \in B_H, \\
 y \mapsto \textbf{v}_y, &  y\in Y, \ \textbf{v}_y \in B_H,
\end{array}
\end{equation}
 such that
\begin{equation}\label{e28}
 f(x,y) = \langle \textbf{u}_x, \textbf{v}_y \rangle_H, \ \ \ (x,y) \in X \times Y,
  \end{equation}\\
  where $\langle \cdot, \cdot \rangle_H$ is the inner product on $H$.  Then  \eqref{e15} becomes, via Proposition \ref{P1},  the Grothendieck inequality,  
  \begin{equation*}
  |\sum_{x,y} a_{xy} \langle \textbf{u}_x,\textbf{v}_y \rangle_H| \leq K\|a\|_{\mathcal{F}_2}\\\\
  \end{equation*}
for all finitely supported scalar arrays $(a_{xy})_{x \in X, y \in Y}$.   Therefore,  \eqref{e15} asserts  (again through Proposition \ref{P1})  that the inner product on $H$, viewed as a function on   $B_H \times B_H$,  is in $\tilde{\mathcal{V}}_2(B_H,B_H)$.   
  
\begin{theorem}[\text{yet another version of} \textit{le th\'eor\`eme fondamental}] \label{G2}
  Let $H$ be an infinite-dimensional Hilbert space with inner product $\langle \cdot, \cdot \rangle_H$.  Then,
 \begin{equation} \label{e16}
 \|\langle \cdot, \cdot \rangle_H\|_{\tilde{\mathcal{V}}_2(B_H,B_H)}  :=   \mathcal{K}_G < \infty,
  \end{equation}
  where $\mathcal{K}_G$ is independent of $H.$
  \end{theorem} 
 \noindent The constant in \eqref{e16} is  \textit{the Grothendieck constant}.  Determinations of its two distinct values $\mathcal{K}_G^{\mathbb{C}}$  and $\mathcal{K}_G^{\mathbb{R}}$, which correspond to  a choice of a scalar field in the definition of $\tilde{\mathcal{V}}_2$,  are  open problems. To date, the best published estimates of  $\mathcal{K}_G^{\mathbb{C}}$ and $\mathcal{K}_G^{\mathbb{R}}$  are found in ~\cite{Haagerup:1987} and ~\cite{braverman2011grothendieck}, respectively.\\\
 \begin{remark} \label{R2}  \ \ \\\
{\bf{i.}} \ \em{The $L^2-L^2$ constraint in \eqref{e13}, which guarantees (via Cauchy-Schwarz) the existence of the integral in \eqref{e14}, can be replaced by an $L^p-L^q$ constraint ($\frac{1}{p} + \frac{1}{q} = 1, \ \  2 \geq p \leq \infty$), which produces (via H\"older) the same effect.  Then, modifying accordingly Question \ref{Q3}, we look for integral representations of $f \in l^{\infty}(X \times Y)$ under the constraint 
\begin{equation*}
\sup_{x\in X, \ y\in Y} \|g_{\omega}(x)\|_{L^p(\mu)} \|h_{\omega}(y)\|_{L^q(\mu)}  < \infty.
\end{equation*}
We consider the space $\mathcal{G}_{(p,q)}(X\times Y)$, whose definition is analogous to that of $\mathcal{G}_2(X \times Y)$, and ask whether a Grothendieck-type theorem holds here as well.  That is, do we have $$\mathcal{G}_{(p,q)}(X\times Y) = \mathcal{V}_2(X\times Y)?$$
(Cf. Theorem \ref{T1}.)}

For $p > 2$, and infinite $X$ and $Y,$  the answer is  \emph{no}.  The proof,  like that of $$\mathcal{V}_2(X \times Y)  \subsetneq l^{\infty}(X \times Y),$$ makes use of the finite Fourier matrix.  For an integer  $N > 0$, we take  $X = Y = [N]$, and  indexing space $\Omega = [N]$ with the uniform probability measure $\mu$  on it, i.e.,  $$\mu(\{\omega\}) = \frac{1}{N},  \ \ \omega \in [N].$$  We take representing functions
\begin{equation*}
g_{\omega}(j) = e^{\frac{-2\pi  i j \omega}{N}}, \ \ \ \omega \in \Omega, \  j \in X,
\end{equation*}\
\begin{equation*}
\begin{split}
 h_{\omega}(k) =  \left \{
\begin{array}{ccc}
0 & \quad  \text{if} \quad  \omega  \not= k  \\
N^{\frac{1}{q}} & \quad   \text{ if} \quad  \omega = k, 
\end{array} \right.\\\
& \ \     \omega \in \Omega, \   k \in Y,
\end{split} 
\end{equation*} 
 and define 
\begin{equation*}
f(j,k) := \int_{\Omega}g_{\omega}(j) h_{\omega}(k) \mu(d \omega) = N^{-\frac{1}{p}}e^{\frac{-2\pi  i j k}{N}}, \ \ (j,k) \in X \times Y.
\end{equation*}
Because
\begin{equation*}
\sup_{j\in X, \  k\in Y} \left( \int_{\Omega} |g_{\omega}(j)|^p \mu(d \omega) \right)^{\frac{1}{p}}   \left( \int_{\Omega} |h_{\omega}(k)|^q \mu(d \omega) \right)^{\frac{1}{q}}  = 1,\\\
\end{equation*}\\ 
 we have $\|f\|_{\mathcal{G}_{(p,q)}} \leq 1.$  By duality, via Proposition \ref{P1} and the rescaled Fourier matrix in \eqref{e17}, we also have 
\begin{equation*}
\|f\|_{\tilde{\mathcal{V}}_2} \geq  \frac{1}{N^{\frac{3}{2}}}\sum_{j,k = 1}^N f(j,k) e^{\frac{2\pi  i j k}{N}} =  N^{\frac{1}{2} - \frac{1}{p}} \ \underset{N \to \infty}{\longrightarrow} \infty, 
\end{equation*}
and thus conclude that if $X$ and $Y$ are infinite sets, then 
\begin{equation*}
\mathcal{V}_2(X\times Y) \subsetneq \mathcal{G}_{(p,q)}(X\times Y).
\end{equation*}\\\
\noindent
{\bf{ii.}} \ \em{\em{Grothendieck's \textit{th\'eor\`eme fondamental} had been stated first as a "factorization" theorem in a framework of topological tensor products~\citep{Grothendieck:1956}, and later was reformulated  as an elementary inequality in a context of Banach space theory  ~\citep{Lindenstrauss:1968}.  Since its   reformulation, which made \textit{le th\'eor\`eme fondamental} accessible to a larger public, Grothendieck's theorem has evolved, migrating further and farther into various diverse settings.  (See ~\cite{pisier2012grothendieck}.)

Recently, variants of the Grothendieck inequality have appeared in studies of algorithmic complexity in a context of theoretical computer science; e.g.,  ~\cite{Alon:2004},~\cite{Charikar:2004},~\cite{Alon:2006},  \cite{so2007approximating}, \cite{khot2011grothendieck}.  These studies began with the following inequality  ~\cite{Charikar:2004}.}   Let  $(a_{jk})_{(j,k) \in \mathbb{N}^2}$ be an infinite matrix with real-valued entries, and $a_{jj} = 0$  for all $j \in \mathbb{N}.$ Then, for each  $n \in \mathbb{N},$ and $\textbf{v}_j \in \mathbb{R}^n,$  $\|\textbf{v}_j\| \leq 1$ for $j \in [n]$,
  \begin{equation} \label{e18}
   \sum_{(j,k) \in [n]^2} a_{jk}  \langle \textbf{v}_j,\textbf{v}_k \rangle  \ \ \leq \  K_n \max_{\epsilon_j = \pm 1, \  j \in [n] }\sum_{(j,k) \in [n]^2} a_{jk} \epsilon_j \epsilon_k,
  \end{equation}
 where $\langle \cdot, \cdot \rangle$ denotes the standard dot product in $\mathbb{R}^n,$ and $K_n = \mathcal{O}(\log n).$  Notice that if the "absolute value" is applied to both sides of \eqref{e18}, then (by the Grothendieck inequality),  $K_n = \mathcal{O}(1)$.  Notice also that if we take the "decoupled" version of \eqref{e18},
 \begin{equation*}
 \sum_{(j,k) \in [n]^2} a_{jk}  \langle \textbf{v}_j,\textbf{v}_k \rangle  \ \ \leq \  K_n \max_{\epsilon_j = \pm 1, \ \delta_k = \pm 1 \  j \in [n], \ k \in [n] }\sum_{(j,k) \in [n]^2} a_{jk} \epsilon_j \delta_k,
  \end{equation*}
 then (again by the Grothendieck inequality), $K_n = \mathcal{O}(1)$. 

 That $K_n = \mathcal{O}(\log n)$ is indeed optimal in \eqref{e18} was proved in ~\cite{Alon:2006} by sharpening the following result in ~\cite{Kashin:2003}:  for every $n \in \mathbb{N},$ there exist vectors $\textbf{v}_1, \ldots, \textbf{v}_n$ in the unit ball of $\mathbb{R}^n,$ such that whenever \textit{real-valued} functions $f_1, \ldots, f_n$  in $L^{\infty}([0,1])$ verify
   \begin{equation} \label{e19}
    \langle \textbf{v}_j,\textbf{v}_k \rangle \ = \int_{[0,1]}f_j(t)f_k(t)dt, \ \ \ 1 \leq j < k \leq n
   \end{equation}
 ($dt$  = Lebesgue measure), then
 
 \begin{equation*}
 \max_{j \in [n]}\|f_j\|_{L^{\infty}} \geq C(\log n)^{\frac{1}{4}},
 \end{equation*}
where $C$ is an absolute constant.  Notably, if  \textit{complex-valued} $f_1, \ldots, f_n$  in $L^{\infty}([0,1])$ are allowed in \eqref{e19}, then the opposite phenomenon holds;  cf. \eqref{par},  \eqref{e20}, and Remark \ref{R4e}.ii.}
 \end{remark}

\section{\bf{Integral representations: the case of topological domains}}\label{s2}
\subsection{$L^2$-continuous families} \ \
  If $H$ is a Hilbert space, then the Grothendieck inequality, through an application of Proposition \ref{P1}, guarantees that there exists a regular complex measure $\mu$ on the compact Abelian group
 \begin{equation*}
  \Omega_{B_H} \times \Omega_{B_H} := \{-1,1\}^{B_H} \times  \{-1,1\}^{B_H},
 \end{equation*}
  such that 
 \begin{align} \label{e21}
 \langle \textbf{u}, \textbf{v} \rangle_H \  & = \int_{\Omega_{B_H} \times \Omega_{B_H}} r_{\textbf{u}} \otimes r_{\textbf{v}} d\mu, \ \ \ (\textbf{u}, \textbf{v}) \in B_H \times B_H.
  \end{align}
  (Cf. Proposition \ref{P0}.)   However, whereas the left side of \eqref{e21}  is continuous on $B_H \times B_H$  (separately in each variable with respect to the weak topology, and jointly with respect to the norm topology), continuity cannot be independently inferred on the right side of \eqref{e21}.  To wit, if  $\lambda \in M(\Omega_{B_H} \times \Omega_{B_H})$ is arbitrary,  then the function on  $B_H \times B_H$    defined by
 \begin{equation*}
  \hat{\lambda}(r_{\textbf{u}} \otimes r_{\textbf{v}}), \ \ \ (\textbf{u}, \textbf{v}) \in B_H \times B_H,
 \end{equation*} 
  does not \textit{a priori} ``see" any of the structures (linear or topological) in  $H$.   This is inconsequential in the case of Theorem \ref{T1}:  in the definitions of   $\mathcal{V}_2(X \times Y)$,   $\tilde{\mathcal{V}}_2(X \times Y)$,  and $\mathcal{G}_2(X \times Y)$,  the underlying domains $X$ and $Y$  are merely sets.
 
  But now suppose that in Questions \ref{Q1}, \ref{Q2}, and \ref{Q3},  $X$ and $Y$ are topological spaces, and that $f$  is continuous on $X \times Y$ -- either jointly, or separately in each variable.  Then, in the search for integral representations of $f$, we would want to focus precisely on those representing functions that "respect" the underlying topologies of their domains;  specifically, that we search only amongst representations in \eqref{e5} and \eqref{e14} that \emph{a priori} determine continuous functions on $X \times Y$ -- either jointly, or separately in each variable. 
  
 \begin{definition} \label{D2}  \  Let  $\textbf{g} = \{g_{\omega} \}_{\omega \in \Omega}$ be a family of scalar-valued functions defined on a topological space $X$ and indexed by a finite measure space $(\Omega, \mu)$.\\  

\noindent 
 {\bf{i.}} \ \  $\textbf{g}$ \ is $L^2(\mu)$-\textit{continuous}  if for each $x \in X,$
\begin{equation*}
\omega \ \mapsto g_{\omega}(x), \ \ \omega \in \Omega,
\end{equation*}
determines an element of $L^2(\Omega,\mu),$  and the resulting map 
$\textbf{g} :  X  \rightarrow  L^2(\Omega,\mu)$  defined by
\begin{equation*}
\textbf{g}(x)(\omega) = g_{\omega}(x), \ \ \omega \in \Omega, \ x \in X,
\end{equation*}
is continuous with respect to the topology on $X$ (its domain), and the norm topology on $L^2(\Omega,\mu)$ (its range).\\

\noindent
{\bf{ii.}} \ \   $\textbf{g}$ \  is  \textit{weakly}-$L^2(\mu)$-\textit{continuous}  if for each $x \in X$,
\begin{equation*}
\omega \ \mapsto g_{\omega}(x), \ \ \omega \in \Omega,
\end{equation*}
determines an element of $L^2(\Omega,\mu),$  and the resulting map 
$\textbf{g} :  X  \rightarrow  L^2(\Omega,\mu)$  defined by
\begin{equation*}
\textbf{g}(x)(\omega) = g_{\omega}(x), \ \ \omega \in \Omega, \ x \in X,
\end{equation*}
is continuous with respect to the topology on $X$ and the weak topology on $L^2(\Omega,\mu)$. 
\end{definition}  
\noindent
If in Questions \ref{Q2} and \ref{Q3},  $X$ and $Y$ are topological spaces, and  families of representing functions are $L^2$-continuous, then subsequent integral representations in \eqref{e5} and \eqref{e14} determine functions that are jointly continuous on $X \times Y.$  Similarly, if families are weakly $L^2$-continuous, then corresponding integral representations determine functions that are continuous on $X \times Y$ separately in each coordinate. 
\vskip0.7cm  
     
  \begin{definition} \label{D1}
 Let $X$ and $Y$ be topological Hausdorff spaces, and let  $f$ be a scalar-valued function on $X \times Y.$ \\\\
  \noindent {\bf{i.}}    $f \in \tilde{V}_2(X \times Y)$ if there exist a probability space $(\Omega, \mu)$ and $L^2(\mu)$-continuous families of functions indexed by it, $\emph{\textbf{g}} = \{g_{\omega} \}_{\omega \in \Omega}$ \  and \ $\emph{\textbf{h}} = \{h_{\omega} \}_{\omega \in \Omega}$ defined on $X$ and $Y$ respectively, such that
 \begin{equation} \label{e22}
 \sup_{x \in X, y \in Y}\|\emph{\textbf{g}}(x)\|_{L^{\infty}(\mu)} \|\emph{\textbf{h}}(y)\|_{L^{\infty}(\mu)} < \infty,
 \end{equation}
and
\begin{equation} \label{e23}
f(x,y) = \int_{\Omega}g_{\omega}(x)h_{\omega}(y) \mu(d \omega), \ \ \ (x,y) \in X \times Y.
\end{equation}
If \  $\emph{\textbf{g}}$ and \ $\emph{\textbf{h}}$  are weakly-$L^2(\mu)$-continuous, and \eqref{e22} and \eqref{e23} are satisfied,  then $f$ is said to be in  $ \widetilde{\mathcal{W}V}_2(X \times Y).$\\

\noindent $\|f\|_{\tilde{V}_2(X \times Y)} = \|f\|_{\tilde{V}_2}$ denotes the infimum of the left side in \eqref{e22} over all indexing spaces $(\Omega, \mu),$ and $L^2(\mu)$-continuous families of functions that represent $f$  by \eqref{e23}.  Similarly, $\|f\|_{\widetilde{\mathcal{W}V}_2(X \times Y)} = \|f\|_{\widetilde{\mathcal{W}V}_2}$ denotes  the infimum of the left side in \eqref{e22}  over all weakly-$L^2(\mu)$-continuous families that represent $f$ by \eqref{e23}.\\\\
\noindent {\bf{ii.}}  $f \in G_2(X \times Y)$ if there exist a probability space $(\Omega, \mu)$ and $L^2(\mu)$-continuous families of functions indexed by it, $\emph{\textbf{g}} =\{g_{\omega} \}_{\omega \in \Omega}$  and $\emph{\textbf{h}} =\{h_{\omega} \}_{\omega \in \Omega}$ defined on $X$ and $Y$ respectively, such that 
\begin{equation} \label{e24}
\sup_{x \in X, y \in Y}\|\emph{\textbf{g}}(x)\|_{L^2(\mu)} \|\emph{\textbf{h}}(y)\|_{L^2(\mu)} < \infty,
\end{equation}
and   
\begin{equation} \label{e25}
f(x,y) = \int_{\Omega}g_{\omega}(x)h_{\omega}(y) \mu(d \omega), \ \ \ (x,y) \in X \times Y.
\end{equation}
If $\emph{\textbf{g}}$ and \ $\emph{\textbf{h}}$  are weakly-$L^2(\mu)$-continuous, and  \eqref{e24} and \eqref{e25} are satisfied, then $f$ is said to be in  $ \mathcal{W}G_2(X \times Y).$ \\

\noindent $\|f\|_{G_2(X \times Y)} = \|f\|_{G_2}$ is the infimum of the left side of \eqref{e24} taken over all indexing spaces $(\Omega, \mu),$ and $L^2(\mu)$-continuous families of functions that represent $f$  by \eqref{e25}, and $\|f\|_{\mathcal{W}G_2(X \times Y)} = \|f\|_{\mathcal{W}G_2}$ is the infimum taken over weakly-$L^2(\mu)$-continuous families.\\\\
 \end{definition}

 In the remainder of the section, $X$ and $Y$ denote topological Hausdorff  spaces.  We denote the space of bounded scalar-valued continuous functions on $X\times Y$ by  $C_b(X \times Y)$, and the space of bounded scalar-valued functions on $X\times Y$ continuous separately in each variable by  $C_b(X,Y)$.  Both $C_b(X \times Y)$ and $C_b(X,Y)$ are Banach algebras with sup-norm and point-wise multiplication on $X \times Y.$ 
 \begin{proposition} \label{P2}
 $\|\cdot\|_{\tilde{V}_2}$,   $\|\cdot\|_{G_2}$, $\|\cdot\|_{\widetilde{\mathcal{W}V}_2}$,  and  $\|\cdot\|_{\mathcal{W}G_2}$  are norms on $\tilde{V}_2(X \times Y)$,  $G_2(X \times Y),$  $\widetilde{\mathcal{W}V}_2(X \times Y)$, and \ $\mathcal{W}G_2(X \times Y),$ respectively.  With these norms and point-wise multiplication on $X \times Y$,  the spaces $\tilde{V}_2(X \times Y)$, $G_2(X \times Y)$,  $\widetilde{\mathcal{W}V}_2(X \times Y)$, and \  $\mathcal{W}G_2(X \times Y)$ are Banach algebras, and 
  \begin{equation} \label{e26}
  \begin{split}
  \tilde{V}_2(X \times Y) &\subset  G_2(X \times Y) \subset C_b(X \times Y)\\\
   \widetilde{\mathcal{W}V}_2(X \times Y) &\subset \mathcal{W}G_2(X \times Y) \subset C_b(X,Y),
   \end{split}
 \end{equation}
 where inclusions are norm-decreasing.
 \end{proposition}
 \begin{proof}[Sketch of proof] 
We only outline the arguments verifying assertions involving the first two spaces in first line in \eqref{e26};  the arguments verifying assertions involving the corresponding "weak" spaces in the second line are similar.

  In order to verify that $\tilde{V}_2(X \times Y)$  and  $G_2(X \times Y)$ are linear spaces, and that the triangle inequality holds for  $\|\cdot\|_{\tilde{V}_2}$ and $\|\cdot\|_{G_2},$  note that if  functions $f_1$ and $f_2$ on $X \times Y$  are represented, respectively,  by  $L^2(\Omega_1,\mu_1)$- and $L^2(\Omega_2,\mu_2)$-continuous families, 
  then $f_1 + f_2$ can be represented by "sums" of these families, appropriately indexed by the disjoint union of $(\Omega_1, \mu_1)$ and $(\Omega_2, \mu_2),$ properly defined and normalized.
  
  To prove completeness, verify that absolutely summable sequences in  $\tilde{V}_2(X \times Y)$  and  $G_2(X \times Y)$ are summable in their respective spaces, by applying a "countably infinite" version of the argument used to verify the triangle inequality.
  
  To prove that $\tilde{V}_2(X \times Y)$  and  $G_2(X \times Y)$ are Banach algebras under point-wise multiplication on $X \times Y,$  in the proof of the triangle inequality replace "sums" of $L^2$-continuous families by "products," and replace the disjoint union of $(\Omega_1, \mu_1)$ and $(\Omega_2, \mu_2)$ by the product $(\Omega_1 \times \Omega_2, \mu_1 \times  \mu_2).$ 
  
  The constraint in \eqref{e22} is stronger than \eqref{e24}, and thus the norm-decreasing left inclusion in \eqref{e26}.  The $L^2$-continuity of the representing families implies (via Cauchy-Schwarz) the norm-decreasing right inclusion.\\\
   \end{proof}
 
  \begin{remark} \label{R3}
 \em{\    The two left inclusions in \eqref{e26} are in fact equalities; this is the gist of the "upgraded" Grothendieck theorem that we prove here.  The right inclusions in \eqref{e26} are strict; e.g., see Remark \ref{R1}.iii.
 
 If $X$ and $Y$ are discrete, then by Remark \ref{R1}.iv,
 \begin{equation*}
 \tilde{V}_2(X \times Y) = \tilde{\mathcal{V}}_2(X \times Y).
  \end{equation*} 
  In general, if $X$ and $Y$ are topological Hausdorff spaces, then I know only the obvious inclusion 
\begin{equation*}
\tilde{V}_2(X \times Y) \subset C_b(X \times Y) \cap \tilde{\mathcal{V}}_2(X \times Y).
\end{equation*} 
  Specifically, if $X$ and $Y$ are compact Hausdorff spaces, and  $\tilde{V}(X \times Y)$ is the tilde algebra defined in~\cite{Varopoulos:1968} (and ~\cite[Ch. 11.9]{Graham&McGehee:1979}), then (by Proposition \ref{P1}, and the comment on top of p. 26 in~\cite{Varopoulos:1968}) we have $$\tilde{V}(X \times Y) = C(X \times Y) \cap \tilde{\mathcal{V}}_2(X \times Y),$$ and therefore
  \begin{equation*}
 \tilde{V}_2(X \times Y) \subset \tilde{V}(X \times Y).
 \end{equation*}
 I do not know whether the reverse inclusion holds.}
  \end{remark}

   \subsection{A "continuous" version of \textit{le th\'eor\`eme fondamental}} \ \ 
 The two pairs of integral representations in Definition \ref{D1} are equivalently feasible;  that  is,
 \begin{equation}
 G_2(X \times Y) = \tilde{V}_2(X \times Y),
 \end{equation}
 and
 \begin{equation}
 \mathcal{W}G_2(X \times Y) = \widetilde{\mathcal{W}V}_2(X \times Y),
 \end{equation}
  which (modulo "best constants") supersede Theorem \ref{T1}.  Specifically, we prove here that there exist absolute constants  $K_1 > 0$  and $K_2 > 0$, such that
  \begin{equation} \label{e29}
 \|f\|_{\tilde{V}_2} \leq K_1 \ \ \text{for all} \  \ f \in B_{G_2(X \times Y)} ,
  \end{equation} \   
 and
  \begin{equation} \label{e29a}
 \|f\|_{\widetilde{\mathcal{W}V}_2} \leq K_2 \ \ \text{for all} \  \ f \in B_{\mathcal{W}G_2(X \times Y)}.
  \end{equation} \

 To start, we observe that  $f \in B_{G_2(X \times Y)}$   means that there exist a Hilbert space $H$,  and functions
  \begin{equation} \label{zz1}
\begin{array} {ccc}
 \textbf{g}: X   \rightarrow   B_H, \\
 \textbf{h}: Y   \rightarrow   B_H,
\end{array}
\end{equation}
that are continuous with respect to the topologies on their respective domains $X$ and  $Y$, and the norm topology on  $B_H$,  such that
 \begin{equation*}
 f(x,y) = \langle \textbf{g}(x),\textbf{h}(y)\rangle_H, \ \ \ x \in X, \ y \in Y.
 \end{equation*}\\
 Similarly, $f \in B_{WG_2(X \times Y)}$ means that the functions in \eqref{zz1} are continuous with respect to the topologies on $X$ and $Y$, and the weak topology on $B_H$.   (Cf. \eqref{e27} and \eqref{e28}.)  Therefore,  \eqref{e29} is equivalent to the assertion that, when viewed as a function on $B_H \times B_H$ with the norm topology on $B_H$,  the inner product $ \langle \cdot, \cdot \rangle_H$ is in $\tilde{V}_2(B_H \times B_H)$.   Similarly, \eqref{e29a} is equivalent to the assertion that, when viewed as a function on $B_{\mathcal{W}H} \times B_{\mathcal{W}H}$, where $B_{\mathcal{W}H}$ denotes the  unit ball in $H$ with the weak topology,  the inner product  is in $\widetilde{\mathcal{W}V}_2(B_{\mathcal{W}H} \times B_{\mathcal{W}H})$.   Succinctly put, \eqref{e29} and \eqref{e29a} are equivalent to\\
 \begin{equation} \label{e30}
 \|\langle \cdot, \cdot \rangle_H\|_{\tilde{V}_2(B_H \times B_H)} \  \leq \ K_1\\\\ 
 \end{equation}\\
 and
 \begin{equation} \label{e30a}
 \|\langle \cdot, \cdot \rangle_H\|_{\widetilde{\mathcal{W}V}_2(B_{\mathcal{W}H} \times B_{\mathcal{W}H})} \ \leq \ K_2,
 \end{equation}\\
 with absolute constants $K_1 > 0 $ and $K_2 > 0$, respectively.

 To establish \eqref{e30} and \eqref{e30a}, we take an infinite-dimensional Hilbert space $H$,  and fix an orthonormal basis $A$ in it .  We then take  $H$  to be  $l^2(A)$ with the usual dot product, and construct (algorithmically) $L^2$-continuous as well as weakly-$L^2$-continuous families of functions, defined on $B_{l^2(A)}$ and  indexed by
 \begin{equation}
 (\Omega_{A}, \mathbb{P}_{A}),  \ \ \ \ \mathbb{P}_{A} = \text{\rm{Haar measure}},
 \end{equation}
  specifically intended as  integrands in integral representations of the dot product. 
  
   The main result is stated below.  Tools to prove it are collected in \S4, and the proof is given in \S5.   
  \begin{theorem} \label{T2}
  Let $A$ be an infinite set.  There exists a one-one map
  \begin{equation} \label{e32}
\Phi : B_{l^2(A)} \rightarrow L^{\infty}(\Omega_{A}, \mathbb{P}_{A}),
  \end{equation}
  
 with the following properties: \\
 
 \noindent
 {\bf{(i)}} \ there is an absolute constant $K > 0$, such that\\
  \begin{equation} \label{e31}
  \|\Phi(\textbf{\emph{x}})\|_{L^{\infty}} \leq K, \ \ \ \textbf{\emph{x}} \in B_{l^2(A)};
  \end{equation}\\
  
 \noindent
 {\bf{(ii)}} \ for all  \ ${\bf{x}} \in B_{l^2(A)}$, \\
 \begin{equation} \label{e43}
  \Phi(-\textbf{\emph{x}}) = - \Phi(\textbf{\emph{x}});
  \end{equation}\\

  \noindent
 {\bf{(iii)}}  \ for all \ $({\bf{x}}, {\bf{y}})\in  B_{l^2(A)} \times B_{l^2(A)}$,\\
  \begin{equation} \label{e20}
  \sum_{\alpha \in A} \textbf{\emph{x}}(\alpha)\overline{\textbf{\emph{y}}(\alpha)} = \int_{\Omega_{A}} \Phi(\textbf{\emph{x}})\Phi(\overline{\textbf{\emph{y}}})d\mathbb{P}_A;
  \end{equation}\\
 
 \noindent
 {\bf{(iv)}} \  $\Phi$  is continuous with the weak topology on $B_{l^2(A)}$ and the weak* topology on $L^{\infty}(\Omega_{A}, \mathbb{P}_{A})$;\\
  
 \noindent
 {\bf{(v)}} \  $\Phi$  is continuous with the $l^2$-norm on $B_{l^2(A)}$ and the $L^2(\Omega_{A}, \mathbb{P}_{A})$-norm on its range.\\
 \end{theorem}
 
 \skip 1.5 cm
  
   \begin{corollary}  \label{C1}\ \ \\\
   (1)   If $H$ is a Hilbert space with inner product $\langle \cdot, \cdot \rangle_H,$ then\\
 \begin{equation} \label{G2-1}
 \|\langle \cdot, \cdot \rangle_H\|_{\tilde{V}_2(B_H \times B_H)}  := K_H  \leq  K^2 ,
  \end{equation}
   \begin{equation} \label{G2-11}
 \|\langle \cdot, \cdot \rangle_H\|_{\widetilde{\mathcal{W}V}_2(B_{\mathcal{W}H} \times B_{\mathcal{W}H})}  := K_{\mathcal{W}H}  \leq  K^2 ,
  \end{equation}\\
  where $K$ is the absolute constant in \eqref{e31}.\\
  
  \noindent
  (2)  If $X$ and  $Y$  are topological Hausdorff spaces, then
 \begin{equation} \label{G2-2}
 G_2(X \times Y) = \tilde{V}_2(X \times Y),
 \end{equation}
 and
 \begin{equation} \label{G2-22}
 \mathcal{W}G_2(X \times Y) = \widetilde{\mathcal{W}V}_2(X \times Y).
 \end{equation}
   \end{corollary}
   \begin{proof} \ \ \\\\
   (1) \ To verify \eqref{G2-1},  we take the indexing space $(\Omega_A, \mathbb{P}_A)$, and then, to produce an integral representation of $\langle \cdot, \cdot \rangle_H$, we let $\Phi = \{\Phi(\cdot)(\omega)\}_{\omega \in \Omega_A}$  be the $L^2(\Omega_A, \mathbb{P}_A)$-continuous family in each coordinate.   Here we view $\langle \cdot, \cdot \rangle_H$ as a continuous function on $B_H \times B_H.$
   
 Similarly,  \eqref{G2-11} is verified by taking a weakly continuous $\Phi = \{\Phi(\cdot)(\omega)\}_{\omega \in \Omega_A}$  in a representation of $\langle \cdot, \cdot \rangle_H$,  viewed as a function on $B_{\mathcal{W}H} \times B_{\mathcal{W}H}$ that is separately continuous in each variable.\\
 
\noindent
(2)  \ To prove \eqref{G2-2}, suppose $f$ is an arbitrary scalar-valued function on $X \times Y$, and $\|f\|_{G_2(X \times Y)}  \leq 1$.  Then, by definition, there exist a probability space $(\Omega, \mu)$, and $L^2(\mu)$-continuous families of functions  $\textbf{g} =\{g_{\omega} \}_{\omega \in \Omega}$  and $\textbf{h} =\{h_{\omega} \}_{\omega \in \Omega}$ on $X$ and $Y$,  respectively, such that 
\begin{equation*}
\sup_{x \in X}\|\textbf{g}(x)\|_{L^2}  \leq 1, \ \ \sup_{ y \in Y} \|\textbf{h}(y)\|_{L^2}  \leq 1, 
\end{equation*}
and
\begin{equation*}
f(x,y) = \int_{\Omega}g_{\omega}(x)h_{\omega}(y) \mu(d \omega) \ = \  \langle \textbf{g}(x),\overline{\textbf{h}(y)}\rangle_H,  \ \ (x,y) \in X \times Y,
\end{equation*}
where $H = L^2(\Omega, \mu)$.   Fix an orthonormal basis  $A$ for $H$, and let $\Phi$ be the map in \eqref{e32}.  Then, the $(l^2 \rightarrow L^2)$-continuity of $\Phi$ implies that $\Phi \circ \textbf{g}$ and $\Phi \circ \bar{\textbf{h}}$ are $L^2(\Omega_A, \mathbb{P}_A)$-continuous families that represent $f$, with the estimate 
\begin{equation*}
 \sup_{x \in X,   y \in Y}\|(\Phi \circ \textbf{g})(x)\|_{L^{\infty}(\Omega_A, \mathbb{P}_A)} \  \|(\Phi \circ \bar{\textbf{h}})(y)\|_{L^{\infty}(\Omega_A, \mathbb{P}_A)} \ \leq K^2,
 \end{equation*}
 which verifies $\|f\|_{\tilde{V}_2(X \times Y)} \leq K^2$.
 
  The equality in \eqref{G2-22} is proved similarly, via the weak continuity of  $\Phi$.
   \end{proof} \  \\
     
  \begin{remark} \label {R4e}\ \ \\\
 \em{\textbf{i.} 
  \ To obtain an  $L^{\infty}(\Omega_A, \mathbb{P}_A)$-valued map analogous to $\Phi$,  defined on the full Hilbert space $l^2(A)$,  we can take \\
 \begin{equation} \label{e31e}
 {\bf{x}}   \ \mapsto \ \|{\bf{x}}\|_2 \ \Phi\big({\bf{x}}/\|{\bf{x}}\|_2\big), \ \ \ {\bf{x}}  \in l^2(A),  \ \ {\bf{x}} \neq 0,
 \end{equation}\\
 which is ($l^2 \rightarrow L^2$)-continuous, but, notably,  \emph{not}  continuous with respect to the weak topologies on $l^2(A)$ and  $L^2(\Omega_A, \mathbb{P}_A)$;  see Remark \ref{notw}.  On the upside, the map in \eqref{e31e} is homogeneous on $l^2_{\mathbb{R}}(A)$. 
 
 A map similar to \eqref{e31e}, with the same properties but sharper bounds, can be obtained also by slightly altering the construction of $\Phi$ in the proof of Theorem \ref{T2}; see Corollary \ref{full}.
 \\\\
 \textbf{ii.} \   It follows from \eqref{e20} and the result in~\cite{Kashin:2003} (stated in Remark \ref{R2}.ii) that  $\Phi$ does not commute with complex conjugation.  Specifically,  the image of $B_{l^2_{\mathbb{R}}(A)}$ under $\Phi$  contains elements in $L^{\infty}(\Omega_{A}, \mathbb{P}_{A})$ with non-zero imaginary parts.\\\\
\textbf{iii.} \ We have
\begin{equation} \label{e33}
\mathcal{K}_{G} \leq \mathcal{K}_{CG} := \sup\{K_H: \text{Hilbert space} \ H\} \leq K^2,
\end{equation}
as well as 
\begin{equation} \label{e33e}
\mathcal{K}_{G} \leq \mathcal{K}_{\mathcal{W}CG} := \sup\{K_{\mathcal{W}H}: \text{Hilbert space} \ H\} \leq K^2,
\end{equation}\\
where $\mathcal{K}_{G}$ is the Grothendieck constant defined in \eqref{e16}, $K_H$ and  $K_{\mathcal{W}H}$ are the constants in \eqref{G2-1} and \eqref{G2-11}, and  $K$ is the constant in \eqref{e31}, 
with the usual distinction between the real and complex cases.  The best estimates to date of $\mathcal{K}_{G}$ can be found in ~\cite{braverman2011grothendieck} and ~\cite{Haagerup:1987}.  In \S \ref{sf}, we note that   $K^2$ on the right sides of \eqref{e33} and \eqref{e33e} can in fact be replaced by a smaller bound, which is strictly greater than $\mathcal{K}_{G}$.  But, otherwise, I do not know which of the inequalities in \eqref{e33} and \eqref{e33e} are strict.}
\end{remark}

\section{\bf{Tools}}\label{s3}
Proofs of Theorem \ref{T2} and its extensions use harmonic analysis on dyadic groups. 
\subsection{The framework} \ \ Let $A$ be a set, and consider the product $$\Omega_A := \{-1,1\}^A,$$  equipped with the usual product topology.  The Borel field $\mathscr{B}_A$ in $\Omega_A$ is generated by the cylindrical sets
\begin{equation*}
C(F;\eta) = \big \{\omega \in \Omega_A: \omega(\alpha) = \eta(\alpha), \ \alpha \in F \big \}, \ \  \text{finite} \ F \subset A,\  \eta \in \Omega_F,
\end{equation*}
whence the uniform probability measure $\mathbb{P}_A$ on $(\Omega_A, \mathscr{B}_A)$ is determined by
\begin{equation*}
\mathbb{P}_A(C(F;\eta)) = \big(\frac{1}{2}\big)^{|F|}, \ \  \text{finite} \ F \subset A,
\end{equation*}\\
where $|F|$ denotes the cardinality of $F.$  Multiplication in $\Omega_A$ is defined by 

\begin{equation} \label{e34}
(\omega \cdot \omega')(\alpha) = \omega(\alpha) \omega'(\alpha), \ \ \ \omega \in \Omega_A, \  \omega' \in \Omega_A, \ \alpha \in A.
\end{equation}\\
We thus have a compact Abelian group $\Omega_A$ with normalized Haar measure $\mathbb{P}_A$.

\subsection{Rademacher and Walsh characters}  \ \ Let $\widehat{\Omega}_A$ denote the group of characters of $\Omega_A,$  wherein group operation is point-wise multiplication of functions.  Let $r_0$ denote the character that is identically $1$ on $\Omega_A$ (multiplicative identity in $\widehat{\Omega}_A$).  For   $\alpha \in A$, let $r_{\alpha}$ be the $\alpha^{th}$ coordinate function on $\Omega_A,$
\begin{equation*}
r_{\alpha}(\omega) = \omega(\alpha), \ \ \omega \in \Omega_A.
\end{equation*}
The set $R_A := \{r_{\alpha}\}_{\alpha \in A}$ is a subset of $\widehat{\Omega}_A,$ and is independent in the following two senses:

\noindent (i)  (\textit{Statistical independence}).  $R_A$ is a system of identically distributed independent random variables on the probability space $(\Omega_A, \mathscr{B}_A, \mathbb{P}_A).$

\noindent (ii)  (\textit{Algebraic independence}) $R_A$ is algebraically independent in $\widehat{\Omega}_A:$ \ for  $F \subset A \cup \{0\},$\\
 $$\prod_{\alpha \in F} r_{\alpha} = r_0  \ \  \Rightarrow \ \ F = \{0\}.$$\\

 \noindent  The system  $R_A$ generates the full character group $\widehat{\Omega}_A$.  Specifically, let $W_{A,0} = \{r_0\},$ and 
\begin{equation*}
W_{A,k} = \big\{\prod_{\alpha \in F}r_{\alpha}: F\subset A, \ |F| = k\big\}, \ \ k \in \mathbb{N}.
\end{equation*}
Then,
\begin{equation*}
W_A := \bigcup_{k=0}^{\infty} W_{A,k} = \widehat{\Omega}_A.
\end{equation*}
We refer to members of $W_A$ as \emph{Walsh characters}, to members of $R_A$ (=  $W_{A,1}$) as \emph{Rademacher characters}, and to members of $W_{A,k},$   $ k \geq 1,$ as \emph{Walsh characters of order $k.$}

\subsection{Walsh series}  \ \ At the very outset, \textit{Walsh series} \ \ 
\begin{equation*}
S  \sim  \sum_{w \in W_A} a_w w, \ \ \ (a_w)_{w \in W_A} \in \mathbb{C}^{W_A},
\end{equation*}
are merely formal objects.  We write $\hat{S}(w) = a_w,$ and 
\begin{equation*}
\text{spect}(S) := \{w \in W_A: \hat{S}(w) \ne 0 \}
\end{equation*}
(spectrum of $S$).  Let $M(\Omega_A)$ denote the space of regular complex measures on $(\Omega_A, \mathscr{B}_A)$ with the total variation norm.  The \textit{Walsh transform} of  $\mu \in M(\Omega_A)$  is
\begin{equation*}
\hat{\mu}(w) = \int_{\Omega_A} w(\omega) \mu(d\omega), \ \ \ w \in W_A,
\end{equation*}
and its Walsh series is
\begin{equation*}
S[\mu] \sim \sum_{w \in W_A} \hat{\mu}(w) w.
\end{equation*}
If $f \in L^1(\Omega_A, \mathbb{P}_A),$ then $\hat{f}  = \widehat{fd\mathbb{P}_A}.$\\\

 If $f$ is a Walsh polynomial,
\begin{equation*}
f = \sum_{w \in F} a_w w, \ \ \ F \subset A,  \ \  |F| < \infty, \ \ (a_w)_{w \in F} \in \mathbb{C}^F,
\end{equation*}
then $\text{spect}(f) \subset F,$ and $\hat{f}(w) = a_w$ for $w \in F.$  Moreover (Parseval's formula),
\begin{equation} \label{Par}
\int_{\Omega_A}f d \mu = \sum_{w \in F} \hat{f}(w) \hat{\mu}(w), \ \ \ \mu \in M(\Omega_A).
\end{equation}
Therefore, because Walsh polynomials are norm-dense in $C(\Omega_A)$, if $\mu \in M(\Omega_A)$ and $\hat{\mu} = 0$ on $\widehat{\Omega}_A,$ then $\mu = 0.$  

If $\mu \in M(\Omega_A),$ and  $$\sum_{w \in W_A} |\hat{\mu}(w)|^2 < \infty,$$ then  $\mu \ll \mathbb{P}_A,$ and $$\frac{d\mu}{d \mathbb{P}_A} \in L^2(\Omega_A, \mathbb{P}_A).$$  In particular, $W_A$ is a complete orthonormal system in $L^2(\Omega_A, \mathbb{P}_A),$ and (Plancherel's formula)
\begin{equation*}
\int_{\Omega_A}|f|^2 d \mathbb{P}_A = \sum_{w \in W_A} |\hat{f}(w)|^2, \ \ \ f \in L^2(\Omega_A, \mathbb{P}_A).
\end{equation*}

\subsection{Riesz products} \label{SS4} \ \  We define the Riesz product\\
\begin{equation*}
\mathfrak{R}_F(\textbf{x}) \sim \prod_{\alpha \in F} \big(r_0 + \textbf{x}(\alpha) r_{\alpha} \big), \ \ F \subset A, \ \ \textbf{x} \in \mathbb{C}^F,
\end{equation*}
to be the Walsh series
\begin{equation} \label{e35}
\mathfrak{R}_F(\textbf{x}) \sim \sum_{k=0}^{\infty} \left(\sum_{w \in W_{F,k}, \ w = r_{\alpha_1} \cdots r_{\alpha_k}} \textbf{x}(\alpha_1) \cdots  \textbf{x}(\alpha_k) r_{\alpha_1} \cdots r_{\alpha_k}\right).
\end{equation}\\
  We detail below two classical scenarios (cf. \cite{Riesz:1918}, ~\cite{Zygmund:1947}), which play key roles in this work.\\

\noindent \textbf{i.} \  \textit{If \ $\emph{\bf{x}} \in l^{\infty}_{\mathbb{R}}(F)$ ( $= \mathbb{R}^F$ with the supremum norm) and $\|\emph{\bf{x}}\|_{\infty} \leq 1$ (i.e., ${\bf{x}} \in B_{l^{\infty}_{\mathbb{R}}(F)}$), then $\mathfrak{R}_F(\emph{\bf{x}})$ represents a probability measure on $(\Omega_A, \mathscr{B}_A).$} \\[3mm]
\noindent
To verify this, let $f$ be a Walsh polynomial with $\text{spect}(f) \subset W_E,$  where $E \subset A$ is finite.  Then (by Parseval's formula),
\begin{equation} \label{e36}
\sum_{w \in W_A} \hat{f}(w) \widehat{\mathfrak{R}_F(\textbf{x})}(w) = \int_{\Omega_A} f(\omega) \ \mathfrak{R}_{E \cap F}(\textbf{x})(\omega) \ \mathbb{P}_A(d \omega).
\end{equation}
Note that $\mathfrak{R}_{E \cap F}(\emph{\bf{x}})$ is a positive Walsh polynomial, and therefore
\begin{equation*}
\|\mathfrak{R}_{E  \cap F}\|_{L^1} = \int_{\Omega_A} \mathfrak{R}_{E \cap F} (\emph{\bf{x}}) \ d \mathbb{P}_A  = 1.
\end{equation*}
Then, from \eqref{e36},
\begin{equation*}
\big |\sum_{w \in W_A} \hat{f}(w) \widehat{\mathfrak{R}_F(\textbf{x})}(w) \big | \leq \|f\|_{L^{\infty}}.
\end{equation*}\\
Therefore, by the Riesz representation theorem and the density of Walsh polynomials in $C(\Omega_A)$,  \eqref{e35} is the Walsh series of a probability measure $\mathfrak{R}_F(\textbf{x})$  on $(\Omega_A, \mathscr{B}_A).$ \\

Let 
\begin{equation} \label{e39a}
\begin{split}
P_F({\bf{x}}) &= \mathfrak{R}_F\big(\frac{{\bf{x}}}{2}\big) - \mathfrak{R}_F\big(\frac{-{\bf{x}}}{2}\big) \\\\
&\sim   \sum_{k=0}^{\infty}\frac{1}{4^k} \bigg(\sum_{{w \in W_{F,2k+1}} \atop{w = r_{\alpha_1} \cdots r_{\alpha_{2k+1}}}}{\bf{x}}(\alpha_1) \cdots  {\bf{x}}(\alpha_{2k+1}) r_{\alpha_1} \cdots r_{\alpha_{2k+1}}\bigg).
\end{split}
\end{equation}\\
\begin{lemma} \label{L2a}
Let ${\bf{x}} \in B_{l^{\infty}_{\mathbb{R}}(F)}$.   Then,
\begin{equation} \label{e39e}
 \emph{spect} \big( P_F({\bf{x}}) \big)    \subset \bigcup_{k=0}^{\infty} W_{F, 2k+1} \ ;
 \end{equation}\\ 
\begin{equation} \label{e40aa}
\|P_F({\bf{x}})\|_{M(\Omega_A)}   \leq 2; 
\end{equation}\\
\begin{equation} \label{e41aa}
 \widehat{P_F({\bf{x}})}(r_{\alpha})   =   {\bf{x}}(\alpha), \ \ \alpha \in F, \ 
 \end{equation}\\
and
 \begin{equation} \label{e43aa}
 \|\widehat{P_F({\bf{x}})}|_{R_F^c} \|_{\infty} \leq \frac{1}{4},
 \end{equation} \\
where $\widehat{P_F({\bf{x}})}|_{R_F^c}$ is the restriction of $\widehat{P_F({\bf{x}})}$ to  $R_F^c := W_F \setminus R_F$.\\\\
\end{lemma}
 
 \noindent \textbf{ii.} \  \textit{If \ $\emph{\bf{x}} \in l^2_{\mathbb{R}}(F)$ (i.e., ${\bf{x}} \in \mathbb{R}^F$ and $\|\emph{\bf{x}}\|_2 := \big( \sum_{\alpha \in F}|\emph{\bf{x}}(\alpha)|^2 \big)^{\frac{1}{2}} < \infty$) then $\mathfrak{R}_F(i \emph{\bf{x}}) \in L^{\infty}(\Omega_A, \mathbb{P}_A)$,  and
\begin{equation} \label{e36a}
\|\mathfrak{R}_F(i\emph{\bf{x}})\|_{L^{\infty}} \leq e^{\frac{\|\emph{\bf{x}}\|_2^2}{2}},
\end{equation}
where  $i = \sqrt{-1}$.}\\

\noindent
To prove this, observe first that if $F \subset A$ is finite, then
\begin{equation} \label{e37}
\|\mathfrak{R}_F(i\emph{\bf{x}})\|_{L^{\infty}} = e^{\frac{1}{2} \sum_{\alpha \in F} \log(1 + |\emph{\bf{x}}(\alpha)|^2)} \leq e^{\frac{\|\emph{\bf{x}}\|_2^2}{2}}. 
\end{equation}\\
If  $F \subset A$ is arbitrary and (without loss of generality) countably infinite, then take  $F_n \subset F_{n+1},$ $n = 1, \dots,$ to be an increasing sequence of finite sets, such that $\bigcup_{n=1}^{\infty} F_n = F$.   Then, 
\begin{equation} \label{e38}
\lim_{n \rightarrow \infty} \widehat{\mathfrak{R}_{F_n}(i\textbf{x})}(w) = \widehat{\mathfrak{R}_F(i\textbf{x})}(w), \ \ w \in W_A.
\end{equation}\\
Therefore, by \eqref{e37} and \eqref{e38},  the sequence $$\mathfrak{R}_{F_n}(i\textbf{x}), \ \ n=1, \dots,$$ converges in the weak* topology of $L^{\infty}(\Omega_A, \mathbb{P}_A)$ to an element in $L^{\infty}(\Omega_A, \mathbb{P}_A),$ whose Walsh series is
\begin{equation*}
 \mathfrak{R}_F(i\textbf{x}) \sim  \sum_{k=0}^{\infty}i^k \left(\sum_{w \in W_{F,k}  \ w = r_{\alpha_1} \cdots r_{\alpha_k}}\textbf{x}(\alpha_1) \cdots  \textbf{x}(\alpha_k) r_{\alpha_1} \cdots r_{\alpha_k}\right),
 \end{equation*}
 and whose $L^{\infty}$-norm is bounded by  $e^{\frac{\|\emph{\bf{x}}\|_2^2}{2}}$.\\ 
 \begin{remark}
   \emph{The random variables $$\mathfrak{R}_{F_n}(i\textbf{x}) = \prod_{\alpha \in F_n}\big(r_0 + i \textbf{x}(\alpha) r_{\alpha} \big), \ \ n = 1, \ldots, $$ form an $L^{\infty}$-bounded martingale sequence, and therefore, by the martingale convergence theorem, the numerical sequence $$\big(\mathfrak{R}_{F_n}(i\textbf{x})\big)(\omega), \ \ n=1, \dots,$$ converges to  $\big(\mathfrak{R}_F(i\textbf{x})\big)(\omega)$ for almost all $\omega \in (\Omega_A, \mathbb{P}_A)$. (E.g., see ~\cite{williams1991}.)} 
   
  \emph{Otherwise, the sequence $\mathfrak{R}_{F_n}(i\textbf{x}), \ n = 1, \ldots,$ converges in the $L^{\infty}$-norm if and only if ${\bf{x}} \in l^1(F)$, i.e., $\sum_{\alpha \in F} |{\bf{x}}(\alpha)| < \infty$.  (See Proposition \ref{P3}.)} \\\
   \end{remark}

 For our purposes here, we take the imaginary part of $\mathfrak{R}_F(i\textbf{x}),$
 \begin{equation} \label{e45}
 Q_F(\textbf{x}) := \text{Im}{\mathfrak{R}_F(i\textbf{x})} \sim   \sum_{k=0}^{\infty}(-1)^k \bigg(\sum_{{w \in W_{F,2k+1}} \atop{w = r_{\alpha_1} \cdots r_{\alpha_{2k+1}}}}\textbf{x}(\alpha_1) \cdots  \textbf{x}(\alpha_{2k+1}) r_{\alpha_1} \cdots r_{\alpha_{2k+1}}\bigg),
  \end{equation}\\
 and estimate the $l^2$-norm of the restriction of $\widehat{Q_F(\textbf{x})}$ to  $R_F^c$  (complement of $R_F$ in $W_F$):  for each $k \geq 2$, 
 \begin{equation} \label{e44}
 \sum_{w \in W_{F,k}, \ w = r_{\alpha_1} \cdots r_{\alpha_k} } |\textbf{x}(\alpha_1) \cdots \textbf{x}(\alpha_k)|^2 \  \leq \ \frac{\|x\|_2^{2k}}{k!}, 
 \end{equation}
 and therefore
 \begin{equation} \label{e44a}
 \begin{split}
 \|\widehat{Q_F(\textbf{x})}|_{R_F^c} \|_2 \ &\leq  \left(\sum_{k=1}^{\infty}\frac{\|\textbf{x}\|_2^{2(2k+1)}}{(2k+1)!}\right)^{\frac{1}{2}}\\\\
 & =  \sqrt{\sinh(\|{\bf{x}}\|_2^2) - \|{\bf{x}}\|_2^2 }.\\
 \end{split}
 \end{equation}\\
 We summarize, and record for future use:
 
 \begin{lemma} \label{L2}
 Let $F \subset A,$ and ${\bf{x}} \in l^2_{\mathbb{R}}(F)$.  Let $ Q_F(\rm{\bf{x}})$ be the imaginary part of the Riesz product $\mathfrak{R}_F(i\rm{\bf{x}}).$  Then
 \begin{equation} \label{e39}
 \emph{spect} \big( Q_F({\bf{x}}) \big)    \subset \bigcup_{k=0}^{\infty} W_{F, 2k+1} \ ;
 \end{equation}\\\
\begin{equation} \label{e40}
\|Q_F({\bf{x}})\|_{L^{\infty}}   \leq e^{\frac{\|{\bf{x}}\|_2^2}{2}} \ ;
\end{equation}\\\
\begin{equation} \label{e41}
 \widehat{Q_F(\rm{\bf{x}})}(r_{\alpha})   =   \rm{\bf{x}}(\alpha), \ \ \alpha \in F \ ;
 \end{equation}\\\
 \begin{equation} \label{e42}
  \|\widehat{Q_F(\rm{\bf{x}})}|_{R_F^c} \|_2   \leq  \sqrt{\sinh \|{\bf{x}}\|_2^2 - \|{\bf{x}}\|_2^2}.
 \end{equation}\\\
 \end{lemma}

 \subsection{Continuity} \    Viewed as a function from $l^2_{\mathbb{R}}(A)$ into $L^{\infty}(\Omega_A, \mathbb{P}_A),$ $Q_A$ is ($l^2 \rightarrow L^2$)-continuous; i.e., $Q_A$ is continuous with respect to the $l^2(A)$-norm on its domain and the $L^2(\Omega_A, \mathbb{P}_A)$-norm on its range.  (It is \emph{not} continuous with respect to the $L^{\infty}(\Omega_A, \mathbb{P}_A)$-norm; see \S \ref{contQ}.)  This is   a consequence of the following. 
\begin{lemma} \label{L3}
For  $\textbf{\emph{x}} \in l_{\mathbb{R}}^2(A), \ \textbf{\emph{y}} \in l_{\mathbb{R}}^2(A),$\\
\begin{equation}\label{g10}
\|Q_A({\bf{x}}) - Q_A({\bf{y}}) \|_{L^2(\Omega_A,\mathbb{P}_A)} \ \leq \  \sqrt{2 \cosh(2\rho^2)} \ \|{\bf{x}} - {\bf{y}}\|_2,
\end{equation}\\
where $\rho = \max\{\|{\bf{x}}\|_2, \|{\bf{y}}\|_2\}$.
\end{lemma}
\vskip0.3cm
\begin{proof}
For ${\bf{y}} = {\bf{0}}$,  by the estimate in \eqref{e44a},
\begin{equation}
 \|Q_A({\bf{x}})\|_{L^2} \leq \sqrt{\sinh(\|{\bf{x}}\|_2^2)},
 \end{equation} 
which implies \eqref{g10} in this instance.  Now suppose  ${\bf{x}} \neq {\bf{0}},$ and ${\bf{y}} \neq {\bf{0}},$  and  write $\|\textbf{x} - \textbf{y} \|_2 = \epsilon$.  By Plancherel's theorem and  spectral analysis of Riesz products,
\begin{equation} \label{g11}
\|Q_A(\textbf{x}) - Q_A(\textbf{y}) \|_{L^2}^2 = \sum_{k=0}^{\infty}   \bigg (\sum_{\substack{w \in W_{A,2k+1} \\[1mm] w = r_{\alpha_1} \cdots \ r_{\alpha_{2k+1}}}}|\textbf{x}(\alpha_1) \cdots \textbf{x}(\alpha_{2k+1}) - \textbf{y}(\alpha_1) \cdots \textbf{y}(\alpha_{2k+1})|^2 \bigg ).
\end{equation}\\\\
For $k \geq 1$ (cf. \eqref{e44}),\\\
\begin{equation*}
\sum_{\substack{w \in W_{A,2k+1} \\[1mm] w = r_{\alpha_1} \cdots \ r_{\alpha_{2k+1}}}}|\textbf{x}(\alpha_1) \cdots \textbf{x}(\alpha_{2k+1}) - \textbf{y}(\alpha_1) \cdots \textbf{y}(\alpha_{2k+1})|^2 \
\end{equation*}
\begin{equation*}
\leq \ \frac{1}{(2k+1)!} \sum_{\alpha_1, \ldots, \alpha_{2k+1}} |\textbf{x}(\alpha_1) \cdots \textbf{x}(\alpha_{2k+1}) - \textbf{y}(\alpha_1) \cdots \textbf{y}(\alpha_{2k+1})|^2 \ 
\end{equation*}\
\begin{equation*}
\leq \ \frac{2}{(2k+1)!} \sum_{\alpha_1, \ldots, \alpha_{2k+1}}  |\textbf{x}(\alpha_1) \textbf{x}(\alpha_2) \cdots \textbf{x}(\alpha_{2k+1}) - \textbf{y}(\alpha_1) \textbf{x}(\alpha_2) \cdots \textbf{x}(\alpha_{2k+1})|^2  \   
\end{equation*}
\begin{equation*}
+ \  \frac{2}{(2k+1)!} \sum_{\alpha_1, \ldots, \alpha_{2k+1}}  | \textbf{y}(\alpha_1) \textbf{x}(\alpha_2) \cdots \textbf{x}(\alpha_{2k+1}) -  \textbf{y}(\alpha_1) \cdots \textbf{y}(\alpha_{2k+1})|^2 \ 
\end{equation*}
\begin{equation} \label{g12}
 =  \ \frac{2 \epsilon^2}{(2k+1)!} \ \|\textbf{x}\|_2^{4k} \ + \  \frac{2}{(2k+1)!} \ \|\textbf{y}\|_2^2  \sum_{\alpha_2, \ldots, \alpha_{2k+1}} |\textbf{x}(\alpha_2) \cdots \textbf{x}(\alpha_{2k+1}) - \textbf{y}(\alpha_2) \cdots \textbf{y}(\alpha_{2k+1})|^2.
\end{equation}\\\\
By a recursive application of these estimates ($2k$ times) to the sum on the right side of \eqref{g12}, we obtain that the left side of \eqref{g12} is bounded by
\begin{equation} \label{g15}
 \frac{\epsilon^2 2^{2k+1}}{(2k+1)!}  \sum_{\l = 0}^{2k} \|\textbf{x}\|_2^{2(2k- \l)} \|\textbf{y}\|_2^{2\l}  \ <  \ \epsilon^2 \left(  \frac{2^{2k+1}}{(2k)!} \right) (\max\{\|\textbf{x}\|^2_2, \|\textbf{y}\|^2_2\})^{2k}.
\end{equation}\\\
By applying \eqref{g15} to each of the summands on the right side of \eqref{g11}, we deduce \eqref{g10}.\\\
\end{proof} 

When restricted to bounded sets in $l^2(A),$  the vector-valued function $Q_A$ is also weakly continuous.  

\begin{lemma} \label{L3W} \ $Q_A : B_{l_{\mathbb{R}}^2(A)} \rightarrow L^{\infty}(\Omega_A,\mathbb{P}_A)$ is continuous with respect to the weak topology on $l^2(A)$, and the weak* topology on $L^{\infty}(\Omega_A,\mathbb{P}_A)$ and (therefore) the weak topology on $L^2(\Omega_A,\mathbb{P}_A)$.
\end{lemma}

 \begin{proof}  If ${\bf{x}}_j \rightarrow {\bf{x}}$ weakly in 
$l^2(A)$ and \ ${\bf{x}}_j \in B_{l^2_{\mathbb{R}}(A)}$, then
\begin{equation}
\widehat{Q_A({\bf{x}}_j)}(w) \ \underset{j \to \infty}{\longrightarrow} \ \widehat{Q_A({\bf{x}})}(w), \ \ \ w \in W_A,
\end{equation}
 and $$\sup_j \|Q_A({\bf{x}}_j\|_{L^{\infty}} \leq e^{\frac{1}{2}}.$$  Therefore, 
\begin{equation}
\int_{\Omega_A} Q_A({\bf{x}}_j) f d \mathbb{P}_A \ \underset{j \to \infty}{\longrightarrow} \ \int_{\Omega_A} Q_A({\bf{x}}) f d \mathbb{P}_A, \ \ \ \text{Walsh polylnomials} \  f,
\end{equation}\\\
verifying that $Q_A({\bf{x}}_j)$ converge to $Q_A({\bf{x}})$ in the weak* topology of  $L^{\infty}(\Omega_A,\mathbb{P}_A)$.\\
\end{proof}

 \vskip 0.5cm
 \section{\bf{Proof of Theorem \ref{T2}}}\label{s4}
  
\subsection{The construction of  $\Phi$} \  
Let   $\{A_j: j \in \mathbb{N} \}$  be a partition of $A$, such that   $A_j$ and $A$ have the same cardinality for each $j \in \mathbb{N}$.  We fix a bijection     
\begin{equation} \label{g19}
\tau_{0}: \  A \rightarrow A_1.
\end{equation}
 For $j \in \mathbb{N},$ denote
\begin{equation*}
 C_j := \bigcup_{k=1}^{\infty}W_{A_j,2k+1},
\end{equation*}
and then fix a bijection 
\begin{equation*}
\tau_j: \ C_j \rightarrow A_{j+1}.
\end{equation*}
(If $A$ is infinite, then $R_A$ and its complement $R_A^c$ in $W_A$ have the same cardinality.)

Let $\textbf{x} \in B_{l^2_{\mathbb{R}}(A)}$,  and define $\textbf{x}^{(1)} \in l^2_{\mathbb{R}}(A_1)$  by
\begin{equation} \label{g1}
\textbf{x}^{(1)}(\tau_0\alpha) = \textbf{x}(\alpha), \ \ \ \alpha \in A.
\end{equation}\
Let   
\begin{equation}\label{e42e}
\begin{split}
\delta &:= \  \sqrt{\sinh 1 - 1} \ \approx \ 0.419 \\\\
& \geq \   \bigg \|\bigg (Q_{A_1}\big({\bf{x}}^{(1)}\big) \big |_{R_{A_1}^c}\bigg )^{\wedge}  \bigg \|_2 \ , \ \ \ \ \ \ \ \text{by \eqref{e42}.}
\end{split}
\end{equation}\\\\
We continue recursively:  for  $ j = 2, \dots,$  define $\textbf{x}^{(j)} \in l^2_{\mathbb{R}}(A_j)$   by \\   
\begin{equation} \label{g2}
\textbf{x}^{(j)}(\tau_{j-1}w) =   \delta^{j-2}   \bigg(Q_{A_{j-1}}\big( \frac{\textbf{x}^{(j-1)}}{\delta^{j-2}}\big) \bigg)^{\wedge}(w), \ \ \ w \in C_{j-1}.
\end{equation} \\
 By Lemma \ref{L2} and induction on $j,$ we obtain
\begin{equation} \label{g4}
\|\textbf{x}^{(j)}\|_2 \leq \delta^{j-1}, \ \ \ j \geq 1.
\end{equation}
Then, by \eqref{e39}, \eqref{e40}, and \eqref{e41} (in Lemma \ref{L2}), 
\begin{align}
\text{spect}\bigg(Q_{A_j}\big(\frac{ \textbf{x}^{(j)}}{\delta^{j-1}}\big)\bigg)  & \subset W_{A_j}, \label{g6}\\\
\big \|Q_{A_j}\big(\frac{ \textbf{x}^{(j)}}{\delta^{j-1}}\big) \big \|_{L^{\infty}} & \leq \sqrt{e},  \label{g3}
\end{align}
and
\begin{equation} \label{g7}
\delta^{j-1}   \bigg(Q_{A_j}\big(\frac{ \textbf{x}^{(j)}}{\delta^{j-1}}\big)\bigg)^{\wedge} (r_{\alpha})  = \textbf{x}^{(j)}(\alpha), \ \ \ \  \alpha \in A_j, \ \ j = 1, \ldots \ .
\end{equation}\\\
Define
\begin{equation} \label{g5}
\Phi(\textbf{x}) = \sum_{j=1}^{\infty} (i\delta)^{j-1} \ Q_{A_j}\big(\frac{ \textbf{x}^{(j)}}{\delta^{j-1}}\big),
\end{equation}\\
where $i := \sqrt{-1}$.   The series in \eqref{g5} is absolutely convergent in the $L^{\infty}$-norm, and
\begin{equation} \label{approx}
\begin{split}
\|\Phi(\textbf{x})\|_{L^{\infty}} \ & \leq \ \frac{\sqrt{e}}{1 - \delta} \ = \ \frac{\sqrt{2e}}{\sqrt{2} - \sqrt{e - e^{-1} - 2}} \\\\
& \approx \ 2.836 \ .
\end{split}
\end{equation}\\

For arbitrary $\textbf{x} \in B_{l^2(A)},$ write $\textbf{x} = \textbf{u} + i \textbf{v},$ \  $\textbf{u} \in B_{l^2_{\mathbb{R}}(A)}, \   \textbf{v} \in B_{l^2_{\mathbb{R}}(A)},$ and define\\
\begin{equation} \label{com2}
\begin{split}
\Phi(\textbf{x}) &:= \Phi(\textbf{u}) + i \Phi(\textbf{v})\\\
&=  \sum_{j=1}^{\infty}(i\delta)^{j-1} \  \bigg( Q_{A_j}\big (\frac{ \textbf{u}^{(j)}}{\delta^{j-1}} \big) +  iQ_{A_j}\big (\frac{ \textbf{v}^{(j)}}{\delta^{j-1}}\big) \bigg)\\\
&:= \sum_{j=1}^{\infty}  \Theta_j(\textbf{x}).
\end{split}
\end{equation}\\
Then,
\begin{equation}\label{a45}
\|\Theta_j(\textbf{x})\|_{L^{\infty}} \leq \ 2 \delta^{j-1} \sqrt{e},
\end{equation}
which verifies  \eqref{e31} with 
\begin{equation} \label{const}
K = \frac{2\sqrt{e}}{1 - \delta} \ .
\end{equation}

  If ${\bf{x}}_1 \in B_{l^2}, \ {\bf{x}}_2 \in B_{l^2}$, and ${\bf{x}}_1 \neq {\bf{x}}_2$, then $\widehat{\Phi({\bf{x}}_1)} \neq \widehat{\Phi({\bf{x}}_2)},$  which verifies that $\Phi$ is one-one. 
 
\subsection{$\Phi$  is odd} \ \ To verify \eqref{e43}, note that for $ \textbf{x} \in l^2_{\mathbb{R}}(A),$ the Walsh series of $Q_A(\textbf{x})$  contains only characters of odd order, and that
\begin{align*}
\widehat{Q_A(-\textbf{x})}(r_{\alpha_1} \cdots r_{\alpha_{2k+1}}) & = (-1)^{2k+1} \prod_{\l=1}^{2k+1}\textbf{x}(\alpha_{\l}) \\
&  = -\widehat{Q_A(\textbf{x})}(r_{\alpha_1} \cdots r_{\alpha_{2k+1}}).
\end{align*}\\
Therefore, by induction based on \eqref{g2}, 
\begin{equation*}
(-\textbf{x})^{(j)} = -\textbf{x}^{(j)}, \ \ j \geq 1, \ \ {\bf{x}} \in B_{l^2_{\mathbb{R}}(A)},
\end{equation*}\\
which, by the definition of $\Phi$ in \eqref{g5}, implies
\begin{equation*}
\Phi(-\textbf{x}) = -\Phi(\textbf{x}).
\end{equation*}\\

\subsection{An integral representation of the dot product} \label{SS5} \ \    It suffices to verify  \eqref{e20} for $\textbf{x} \in B_{l_{\mathbb{R}}^2(A)}$ and  $\textbf{y} \in B_{l_{\mathbb{R}}^2(A)}$. For every integer $N \geq 1$,\\
\begin{equation} \label{g8}
\int_{\Omega_A} \left (\sum_{j=1}^N  (i\delta)^{j-1}  \ Q_{A_j}\big(\frac{ \textbf{x}^{(j)}}{ \delta^{j-1}}\big)\right ) \left(\sum_{j=1}^N   (i\delta)^{j-1}  \ Q_{A_j}\big(\frac{ \textbf{y}^{(j)}}{ \delta^{j-1}}\big)\right ) d \mathbb{P}_A  
\end{equation}\\
\begin{equation*}
= \ \sum_{j=1}^N  (-\delta^2)^{j-1} \sum_{w \in W_{A_j}} \bigg(Q_{A_j}\big(\frac{ \textbf{x}^{(j)}}{ \delta^{j-1}}\big)\bigg)^{\wedge}(w) \  \bigg(Q_{A_j}\big(\frac{ \textbf{y}^{(j)}}{ \delta^{j-1}}\big)\bigg)^{\wedge}(w) 
\end{equation*}
(by Parseval's formula, by \eqref{g6}, and because $W_{A_j}$ are pairwise disjoint)\\
\begin{equation*}
=  \sum_{j=1}^N  (-1)^{j-1}  \left ( \sum_{\alpha \in A_j}\textbf{x}^{(j)}(\alpha)   \textbf{y}^{(j)}(\alpha)   +    \sum_{\alpha \in C_j} \delta^{2(j-1)} \bigg(Q_{A_j}\big(\frac{ \textbf{x}^{(j)}}{ \delta^{j-1}}\big)\bigg)^{\wedge}(w)   \bigg(Q_{A_j}\big(\frac{ \textbf{y}^{(j)}}{ \delta^{j-1}}\big)\bigg)^{\wedge}(w)  \right )
\end{equation*}
(by \eqref{g7})
\begin{equation*}
= \ \sum_{j=1}^N  (-1)^{j-1}  \left ( \sum_{\alpha \in A_j}\textbf{x}^{(j)}(\alpha)  \textbf{y}^{(j)}(\alpha)  \ + \   \sum_{\alpha \in A_{j+1}}\textbf{x}^{(j+1)}(\alpha)  \textbf{y}^{(j+1)}(\alpha) \right )
\end{equation*}
(by \eqref{g2})
\begin{equation*} 
= \ \sum_{\alpha \in A}\textbf{x}(\alpha)  \textbf{y}(\alpha)  \  + \ (-\delta^2)^{N-1}  \sum_{\alpha \in C_N}\bigg(Q_{A_N}\big(\frac{ \textbf{x}^{(N)}}{ \delta^{N-1}}\big)\bigg)^{\wedge}(w)   \bigg(Q_{A_N}\big(\frac{ \textbf{y}^{(N)}}{ \delta^{N-1}}\big)\bigg)^{\wedge}(w)\\\\\  
\end{equation*}
(by "telescoping" and by \eqref{g1}).\\\\
Letting $N \rightarrow \infty$ in \eqref{g8}, we deduce
\begin{equation*}
 \int_{\Omega_{A}} \Phi(\textbf{x})\Phi(\textbf{y})d\mathbb{P}_{A} \ = \ \sum_{\alpha \in A}\textbf{x}(\alpha)  \textbf{y}(\alpha).
\end{equation*}\\

\subsection{$\Phi$ \ is weakly continuous}  \ Suppose $${\bf{x}}_k \underset{k \to \infty}{\longrightarrow} {\bf{x}} \ \ \ \ \text{weakly in} \  B_{l^2_{\mathbb{R}}(A)}.$$ 
  By Lemma \ref{L3W},
 \begin{equation}
 \bigg(Q_{A_1}\big( \textbf{x}^{(1)}_k \big)\bigg)^{\wedge}(w) \ \underset{k \to \infty}{\longrightarrow} \ \bigg(Q_{A_1}\big( \textbf{x}^{(1)} \big)\bigg)^{\wedge}(w), \ \ \ \ w \in W_{A_1}, 
 \end{equation}
 and then by induction based on \eqref{g2},
 \begin{equation}
 \bigg(Q_{A_j}\big(\frac{ \textbf{x}^{(j)}_k}{\delta^{j-1}}\big)\bigg)^{\wedge}(w) \ \underset{k \to \infty}{\longrightarrow} \ \bigg(Q_{A_j}\big(\frac{ \textbf{x}^{(j)}}{\delta^{j-1}}\big)\bigg)^{\wedge}(w), \ \ \ \ w \in W_{A_j}, \ \ j = 2, \ldots \ .
 \end{equation}\\
 Therefore,
 \begin{equation} \label{g14w}
 \big(\Phi({\bf{x}}_k)\big)^{\wedge}(w)  \ \underset{k \to \infty}{\longrightarrow} \ \big(\Phi({\bf{x}})\big)^{\wedge}(w), \ \ \ \ w \in W_A,
 \end{equation}\\
 and the assertion in Theorem \ref{T2}(iv) now follows from\\  
 \begin{equation}
 \sup_k \big \|\Phi({\bf{x}}_k) \big\|_{L^{\infty}} \leq K.
 \end{equation}\\
 
 \subsection{$\Phi$ \ is $(l^2 \rightarrow L^2)$-continuous} \  Suppose  ${\bf{x}}_k \in B_{l^2_{\mathbb{R}}(A)}, \ {\bf{x}} \in B_{l^2_{\mathbb{R}}(A)}$, and  
 \begin{equation*}
 \|{\bf{x}}_k - {\bf{x}}\|_2  \ \underset{k \to \infty}{\longrightarrow} \ 0 .
 \end{equation*}
 By Lemma \ref{L3},  we have 
 \begin{equation}
 \big \| Q_A\big({\bf{x}}^{(1)}_k\big) -  Q_A\big({\bf{x}}^{(1)}\big) \big \|_{L^2} \ \underset{k \to \infty}{\longrightarrow} \ 0.
 \end{equation}
 Then, by induction and Lemma \ref{L3}, 
 \begin{equation}
 \big \| Q_A\big({\bf{x}}^{(j)}_k\big) -  Q_A\big({\bf{x}}^{(j)}\big) \big \|_{L^2} \ \underset{k \to \infty}{\longrightarrow} \ 0 , \ \ \ \ \ j = 2, \ldots \ ,
 \end{equation} 
which implies
\begin{equation}
\big \| \Phi({\bf{x}}_k) -  \Phi({\bf{x}}) \big \|_{L^2} \ \underset{k \to \infty}{\longrightarrow} \ 0 .
\end{equation}
\\ 

\section{\bf{Variations on a theme}} \label{sss5}  
To obtain multidimensional extensions of the Grothendieck inequality, and specifically multilinear analogs of Theorem \ref{T2}, we will need a  $\Phi$-like map defined on the \emph{entire}  Hilbert space $l^2(A)$.  To this end, we could simply take the map defined in \eqref{e31e}.   However,  an "on site" modification of the algorithm that produced $\Phi$ in \eqref{g5} seems more natural, and also yields sharper bounds. 
\subsection{The map \ $\Phi_2$}\label{se} \ As before, we start with a partition  $\{A_j: j \in \mathbb{N} \}$ of $A$, with the property that  $A_j$ and $A$ have the same cardinality for each $j \in \mathbb{N}$, and bijections     
\begin{equation*} 
\tau_{0}: \  A \rightarrow A_1,
\end{equation*} 
\begin{equation*}
\tau_j:  \ C_j \rightarrow A_{j+1}, \ \ j \geq 1,
\end{equation*}
where
\begin{equation*}
 C_j := \bigcup_{k=1}^{\infty}W_{A_j,2k+1}.
\end{equation*}
 For $t > 0$, and $\textbf{x} \in l^2(A),$ define 
\begin{equation} \label{nor1}
 \sigma_t \textbf{x} = \left \{
\begin{array}{lcc}
\textbf{x}/{t \|\textbf{x}\|_2} & \quad  \text{if} \quad  \textbf{x} \not= \textbf{0}  \\\\
\textbf{0} & \quad   \text{ if} \quad  \textbf{x} = \textbf{0}.
\end{array} \right.
\end{equation}
For  $t = 1$,  we write $\sigma$ for  $\sigma_1$, as per \eqref{nor}.  \  Fix $t > c > 0$, where 
\begin{equation} \label{endpt}
2 - c^2 \sinh(1/c^2) = 0
\end{equation}
($c \approx 0.6777$), and let $\textbf{x} \in l^2_{\mathbb{R}}(A)$ be arbitrary.  Define $\textbf{x}^{(1)} \in l^2_{\mathbb{R}}(A_1)$  by
\begin{equation} \label{gg1}
\textbf{x}^{(1)}(\tau_0\alpha) = \textbf{x}(\alpha), \ \ \alpha \in A,
\end{equation}\
and then produce $\textbf{x}^{(j)} \in l^2_{\mathbb{R}}(A_j)$,  $ j = 2, \dots,$   by   
\begin{equation} \label{gg2}
\textbf{x}^{(j)}(\tau_{j-1}w) = t \|\textbf{x}^{(j-1)}\|_2 \  \big({Q_{A_{j-1}}(\sigma_t \textbf{x}^{(j-1)})}\big)^{\wedge}(w), \ \ w \in C_{j-1}.
\end{equation}\
By \eqref{e39}, \eqref{e40}, and \eqref{e41} in Lemma \ref{L2}, we have
\begin{align}
\text{spect}(Q_{A_j}(\sigma_t \textbf{x}^{(j)}))  & \subset W_{A_j} , \label{gg6}\\\
\|Q_{A_j}(\sigma_t \textbf{x}^{(j)}) \|_{L^{\infty}} & \leq e^{1/2t^2}, \label{gg3}\\\
t \| \textbf{x}^{(j)}\|_2  \  \big(Q_{A_j}(\sigma_t \textbf{x}^{(j)})\big)^{\wedge} (r_{\alpha}) & = \textbf{x}^{(j)}(\alpha), \ \ \alpha \in A_j, \ \  j \geq 1. \label{gg7}
\end{align} \ \\
By \eqref{e42}, 
\begin{equation} \label{gg4}
  \|\big(Q_{A_j}(\sigma_t\rm{\bf{x}})\big)^{\wedge}|_{C_j} \|_2   \leq  \sqrt{\sinh (1/t^2) -  1/t^2} \ := \ \delta_t, \ \ \ j \geq 1.
 \end{equation}\\\
 ($\delta_1$ is the $\delta$ in  \eqref{e42e}.)  Then, by induction on $j$,  \eqref{gg1}, \eqref{gg2}, and \eqref{gg4}, we obtain
\begin{equation} \label{ggg4}
\|\textbf{x}^{(j)}\|_2 \leq (t \delta_t)^{j-1}\|\textbf{x}\|_2, \ \ j \geq 1.
\end{equation}
Note that \eqref{endpt} is exactly the requirement that $t \delta_t < 1$.  Now define
\begin{equation} \label{gg5}
\Phi_2(\textbf{x};t) = \sum_{j=1}^{\infty}i^{j-1}t \|\textbf{x}^{(j)}\|_2 \  Q_{A_j}(\sigma_t \textbf{x}^{(j)}).
\end{equation}
 (Cf. \eqref{g5}.)  Then, from \eqref{gg3} and \eqref{ggg4}, we obtain\\
\begin{equation} \label{ggg5}
\|\Phi_2(\textbf{x};t) \|_{L^{\infty}} \leq \ t e^{1/2t^2} \sum _{j=1}^{\infty} \|\textbf{x}^{(j)}\|_2 \ \leq \ K(t) \|\textbf{x}\|_2,
\end{equation}\\
where
 \begin{equation} \label{obj1}
K(t) = \frac{t e^{1/2t^2}}{1-t \sqrt{\sinh(1/t^2) - 1/t^2}} \ , \  \ \ \ \ t > c,
\end{equation}\\
and $c > 0$ is the solution to \eqref{endpt}.  For  $\textbf{x} \in l^2(A),$  write $\textbf{x} = \textbf{u} + i \textbf{v},$ \  $\textbf{u} \in l^2_{\mathbb{R}}(A), \   \textbf{v} \in l^2_{\mathbb{R}}(A),$ and define\\
\begin{equation} \label{com2}
\Phi_2(\textbf{x};t) := \Phi_2(\textbf{u};t) + i \Phi_2(\textbf{v};t).
\end{equation} \\

\begin{corollary} \label{full}
For $t >  c$,  let $$\Phi_2(\cdot \ ;t): \ l^2(A) \ \rightarrow \ L^{\infty}(\Omega_A, \mathbb{P}_A)$$ be the map defined by \eqref{gg5} and \eqref{com2}.  Then, 
\begin{equation}
\|\Phi_2({\bf{x}};t)\|_{L^{\infty}} \ \leq 2K(t)\|{\bf{x}}\|_2, \ \ \ {\bf{x}} \in l^2(A),
\end{equation}
where $K(t)$ is given by \eqref{obj1}, \  $\Phi_2(\cdot \ ;t)$ is $\big(l^2 \rightarrow L^2\big)$-continuous, and
\begin{equation} \label{gg8}
\begin{split}
&\sum_{\alpha \in A}{\bf{x}}(\alpha)  {\bf{y}}(\alpha) =  \int_{\Omega_{A}} \Phi_2({\bf{x}};t)\Phi_2({\bf{y}};t)d\mathbb{P}_{A} \\\
&= \ \sum_{j=1}^{\infty}  (-1)^{j-1} \sum_{w \in W_{A_j}} t^2 \|{\bf{x}}^{(j)} \|_2 \  \|{\bf{y}}^{(j)} \|_2 \ \big(Q_{A_j}(\sigma_t {\bf{x}}^{(j)})\big)^{\wedge}(w) \  \big(Q_{A_j}(\sigma_t {\bf{y}}^{(j)})\big)^{\wedge}(w), \\\
& \  \ \ \ \ \ \ \ \ \ \ \ \ \ \ \ \ \ \ \ \ \ \ \ \ \ \ \ \ \ \ \ \ \ \ \ \ \ \ \ \ \ \ \ \ \ \ \ \ \ \ \ \ \ \ \ \ \ \ \ \ \ \ \ \ {\bf{x}} \in l^2(A), \ \ {\bf{y}}  \in l^2(A).
\end{split}
\end{equation}
Moreover, the map 
\begin{equation} \label{gg9}
\Phi_2({\bf{x}}) := \Phi_2({\bf{x}};1), \ \ \ {\bf{x}} \in l^2(A), 
\end{equation}
is homogeneous on $l^2_{\mathbb{R}}(A)$, i.e.,
\begin{equation} \label{homo}
\Phi_2(\xi{\bf{x}}) = \xi \Phi_2({\bf{x}}), \ \ \xi \in \mathbb{C}, \ \  {\bf{x}} \in l_{\mathbb{R}}^2(A).
\end{equation}
\end{corollary}
\begin{proof}[Sketch of proof]
Arguments are similar to those in the proof of Theorem \ref{T2}. 

 To verify $\big(l^2 \rightarrow L^2\big)$-continuity,  we use induction based on Lemma \ref{L3} together with  \eqref{gg2} to show that if $\textbf{x}_k \rightarrow \textbf{x}$ in $l_{\mathbb{R}}^2(A)$, then
\begin{equation*}
\textbf{x}^{(j)}_k \ \underset{k \to \infty}{\longrightarrow} \ \textbf{x}^{(j)} \ \ \text{in} \ l_{\mathbb{R}}^2(A_j), \ \ \ \ j = 1, \ldots \ ,
\end{equation*}
and therefore, again by Lemma \ref{L3},
\begin{equation*}
\|\textbf{x}_k^{(j)}\|_2 \  Q_{A_j}(\sigma_t \textbf{x}_k^{(j)}) \ \underset{k \to \infty}{\longrightarrow} \ \|\textbf{x}^{(j)}\|_2 \  Q_{A_j}(\sigma_t \textbf{x}_k^{(j)}) \ \ \text{in} \ L^2(\Omega_A,\mathbb{P}_A), \ \ \ \ j = 1, \ldots \ .
\end{equation*}

To verify the Parseval-like formula in \eqref{gg8}, it suffices to take $\textbf{x} \in l_{\mathbb{R}}^2(A),$ $\textbf{y} \in l_{\mathbb{R}}^2(A)$. For every integer $N \geq 1$, we have (by the usual Parseval formula)
\begin{equation} \label{ggg8}
\int_{\Omega_A} \left (\sum_{j=1}^N  i^{j-1} t \|\textbf{x}^{(j)} \|_2 \ Q_{A_j}(\sigma_t \textbf{x}^{(j)})\right ) \left(\sum_{j=1}^N  i^{j-1} t \|\textbf{y}^{(j)} \|_2 \ Q_{A_j}(\sigma_t \textbf{y}^{(j)})\right ) d \mathbb{P}_A  
\end{equation}
\begin{equation*}
= \ \sum_{j=1}^N  (-1)^{j-1} \sum_{w \in W_{A_j}} t^2 \|\textbf{x}^{(j)} \|_2 \  \|\textbf{y}^{(j)} \|_2 \ \big(Q_{A_j}(\sigma_t \textbf{x}^{(j)})\big)^{\wedge}(w) \  \big(Q_{A_j}(\sigma_t \textbf{y}^{(j)})\big)^{\wedge}(w) 
\end{equation*}

\begin{equation*}
= \ \sum_{j=1}^N  (-1)^{j-1}  \left ( \sum_{\alpha \in A_j}\textbf{x}^{(j)}(\alpha) \  \textbf{y}^{(j)}(\alpha)  \ + \   \sum_{\alpha \in C_j} t^2 \|\textbf{x}^{(j)} \|_2 \ \|\textbf{y}^{(j)} \|_2 \ \big(Q_{A_j}(\sigma_t \textbf{x}^{(j)})\big)^{\wedge}(w) \  \big(Q_{A_j}(\sigma_t \textbf{y}^{(j)})\big)^{\wedge}(w)  \right )
\end{equation*}

\begin{equation*}
= \ \sum_{j=1}^N  (-1)^{j-1}  \left ( \sum_{\alpha \in A_j}\textbf{x}^{(j)}(\alpha) \  \textbf{y}^{(j)}(\alpha)  \ + \   \sum_{\alpha \in A_{j+1}}\textbf{x}^{(j+1)}(\alpha) \  \textbf{y}^{(j+1)}(\alpha) \right )
\end{equation*}

\begin{equation} \label{ggg9}
= \ \sum_{\alpha \in A}\textbf{x}(\alpha) \  \textbf{y}(\alpha)  \  + \ (-1)^{N-1}t^2 \|\textbf{x}^{(N)}\|_2 \  \|\textbf{y}^{(N)}\|_2 \  \sum_{\alpha \in C_N} \big(Q_{A_N}(\sigma_t \textbf{x}^{(N)})\big)^{\wedge}(w) \  \big(Q_{A_N}(\sigma_t \textbf{y}^{(N)})\big)^{\wedge}(w). 
\end{equation}\\\\
We deduce \eqref{gg8} by letting $N \rightarrow \infty$ in \eqref{ggg8} and \eqref{ggg9}.

To verify homogeneity in the instance $t = 1$, note that for $\textbf{x} \in l^2_{\mathbb{R}}(A),$ by the recursive definition of $\textbf{x}^{(j)}$ in \eqref{gg2}, 
\begin{equation} \label{ggg10}
\textbf{x}^{(j)} = \|\textbf{x}\|_2 \ (\sigma \textbf{x})^{(j)}, \ \ j \geq 1,
\end{equation}
and therefore,  
\begin{equation} \label{gg10}
\Phi_2(\textbf{x}) = \|\textbf{x}\|_2 \  \Phi_2(\sigma \textbf{x}).
\end{equation}
Also, by the same argument that verified \eqref{e43},  
\begin{equation} \label{gg11}
\Phi_2(-\textbf{x}) = -\Phi_2(\textbf{x}).
\end{equation}
Now deduce \eqref{homo} by combining \eqref{gg10} and \eqref{gg11}.
\end{proof} \ \

\begin{remark} \label{notw}
\emph{The map $\Phi_2$ (unlike $\Phi$ of Theorem \ref{T2}) is not continuous with respect to the weak topologies on $l^2(A)$ and  $L^2(\Omega_A, \mathbb{P}_A)$, and \emph{a fortiori} the weak*-topology on $L^{\infty}(\Omega_A, \mathbb{P}_A)$.    For, if \   $\textbf{x}_j \rightarrow \textbf{x}$ weakly in $l^2(A)$, \ $\textbf{x} \neq \textbf{0}$, \ and  $$\liminf_j \|\textbf{x}_j\|_2  > \|\textbf{x}\|_2,$$  then  $$\widehat{Q_A(\textbf{x}_j)}   \rightarrow \widehat{Q_A(\textbf{x})} \  \ \text{on} \  \  W_A  \setminus R_A,$$ and therefore
  $$\|\textbf{x}_j\|_2\widehat{Q_A(\textbf{x}_j)}   \not\rightarrow \|\textbf{x}\|_2\widehat{Q_A(\textbf{x})} \  \ \text{on} \  \  W_A  \setminus R_A.$$\\}
\end{remark}

 \subsection{Sharper bounds} \label{sf} \ The bound  $K(1)  \approx \ 2.836$ \ is the same as the bound obtained from the construction  in the previous section; see \eqref{approx}.  The minimum of  $K(t)$ over $(c, \infty)$ is at $t_{\text{min}} \approx  1.429,$ and $K(t_{\text{min}}) \approx 2.285.$ 

  In Corollary \ref{C1},  the bounds on $K_H$  and  $K_{\mathcal{W}H}$ were obtained through a norm estimate of $\Phi$.   Sharper bounds can be obtained through norm estimates of the summands in  \eqref{gg5} and  \eqref{g5}.  Specifically, for ${\bf{x}}$ and ${\bf{y}}$ in $l^2_{\mathbb{R}}(A)$,
 \begin{equation} \label{gg12}
 \sum_{\alpha \in A} \textbf{x}(\alpha) \  \textbf{y}(\alpha)   = \sum_{j=1}^{\infty}  (-1)^{j-1}  \int_{\Omega_A}   t \|\textbf{x}^{(j)} \|_2 \ Q_{A_j}(\sigma_t \textbf{x}^{(j)}) \ t \|\textbf{y}^{(j)} \|_2 \ Q_{A_j}(\sigma_t \textbf{y}^{(j)}) d \mathbb{P}_A,  
\end{equation}
where ${\bf{x}}^{(j)}$ and ${\bf{y}}^{(j)}$ are given by \eqref{gg2}.   Therefore, by \eqref{gg3} and \eqref{ggg4},
\begin{equation}\label{imp}
K_H \leq 2t^2 e^{1/t^2} \sum_{j=1}^{\infty} (t \delta_t)^{2(j-1)} = \frac{2t^2 e^{1/t^2}}{2 - t^2 \sinh(1/t^2)} := 2L(t), \ \ \ t \ > c > 0,
\end{equation}
where $c > 0$ satisfies \eqref{endpt}

Similarly, we can achieve the same bounds on $K_{\mathcal{W}H}$ via a "$t$-perturbation" of the definition of $\Phi$ in \eqref{g5}.  (The summands in \eqref{gg5} are \emph{not} $\big(\text{weak} \ l^2 \rightarrow \text{weak} \ L^2\big)$-continuous on $B_{l^2_{\mathbb{R}}(A)},$ and hence cannot be used to estimate $K_{\mathcal{W}H}$.)  Given $t > c$, and ${\bf{x}} \in B_{l^2_{\mathbb{R}}(A)},$ we first modify the definition of ${\bf{x}}^{(j)}$ in \eqref{g2},
\begin{equation} \label{ggg2}
\textbf{x}^{(j)}(\tau_{j-1}w) =   t(t \delta_t)^{j-2}   \bigg(Q_{A_{j-1}}\big( \frac{\textbf{x}^{(j-1)}}{t(t \delta_t)^{j-2}}\big) \bigg)^{\wedge}(w), \ \ \ w \in C_{j-1}, \ \ \  \ j = 2, \ldots \ .
\end{equation} \\
Then, by induction,
\begin{equation*}
\|\textbf{x}^{(j)}\|_2 \leq (t \delta_t)^{j-1}, \ \ \ j = 1, \ldots \  .
\end{equation*}
Now define
\begin{equation} \label{modif}
\Phi({\bf{x}}; t) = \sum_{j=1}^{\infty} \ i ^{j-1}t(t \delta_t)^{j-1}  Q_{A_j}\big(\frac{ \textbf{x}^{(j)}}{t (t \delta_t)^{j-1}}\big), \ \ \ {\bf{x}} \in B_{l^2_{\mathbb{R}}(A)}.
\end{equation}
Each of the summands in \eqref{modif}  is $\big(\text{weak} \ l^2 \rightarrow \text{weak} \ L^2\big)$-continuous on $B_{l^2_{\mathbb{R}}(A)},$   and therefore, through an integral representation similar to  \eqref{gg12},  
\begin{equation}
K_{\mathcal{W}H} \leq 2L(t), \ \ \ t \ > c > 0,
\end{equation} 
where $c > 0$ satisfies \eqref{endpt}.
The minimum of $L(t)$ over $(c, \infty)$ is at  $t_{\text{min}} \approx  1.141 $, and $L(t_{\text{min}}) \approx  3.123 $ \ .

\begin{remark} \label{constants} \ \emph{Improved bounds come at a price.  Norm estimates involving  $\Phi_2(\cdot \ ; t)$ improve when   $t$ \ is "dialed" away from  $t = 1$, but then homogeneity is lost. (The would be analog of \eqref{ggg10}, with $\sigma_t$ replacing $\sigma$, is false.)  Sharper bounds not withstanding, we will keep to the case $t = 1$, mainly because homogeneity, as stated in  \eqref{homo}, will be needed later in \S \ref{s8}; see Remark \ref{R7}.ii.}
\end{remark}

  \section{\bf{More about $\Phi$}}\label{s5}
  
  \subsection{Spectrum} \ \  For ${\bf{x}} \in B_{l^2_{\mathbb{R}}(A)}$, the construction of $\Phi(\textbf{x})$  is based on  the spectral analysis of $Q_{A_j}$ and the recursive definition of $\textbf{x}^{(j)} \in B_{l^2_{\mathbb{R}}(A_j)}$.  The construction starts with the observation that for infinite $A$, $$\{S \in 2^A:  |S| < \infty\}$$
and $A$  have the same cardinality, and therefore that also 
\begin{equation*}
\text{spect} \Phi := \bigcup _{j, k=1}^{\infty}  W_{A_j,2k-1}
\end{equation*}
and $A$ have the same cardinality.  If ${\bf{x}}$ has infinite support, then $\text{spect} \Phi({\bf{x}})$ is countably infinite.
For finitely supported ${\bf{x}}$, \emph{except} for two-dimensional vectors  $\textbf{x}$,  the cardinalities of the supports of $\textbf{x}^{(j)}$  grow iterative-exponentially (very fast!) to infinity.  To make this precise, we denote  
  \begin{equation*}
 \nu(\textbf{x}) := |\{\alpha: \textbf{x}(\alpha) \neq 0\}|, 
   \end{equation*}\\
 and  let
 \begin{equation*}
 f(k) = 2^k - k - 1, \ \ \  k \in \mathbb{N}.
 \end{equation*}
  If $\nu(\textbf{x}) < \infty$,  then\\
  \begin{equation*}
  \nu(\textbf{x}^{(j)}) = f^{j-1}\big(\nu(\textbf{x})\big), \ \ j \geq 2,
  \end{equation*}
  where $f^{j-1}$ denotes the $(j-1)^{\text{st}}$ iterate of $f$. Therefore, $\textbf{x}^{(j)} = \textbf{0}$ for all $j \geq 2$ only if $\nu(\textbf{x}) \leq 2$.  If $ 2 < \nu(\textbf{x}) < \infty $,
  then $$\nu(\textbf{x}^{(j)}) = f^{j-1}\big(\nu(\textbf{x})\big) \underset{j \to \infty}{\longrightarrow} \infty.$$  That is,  $$\nu(\textbf{x}) \leq 2 \  \Rightarrow \  |\text{spect}\big(\Phi(\textbf{x})\big)| \leq 2,$$  but otherwise, $$\nu(\textbf{x}) >  2 \  \Rightarrow \  |\text{spect}\big(\Phi(\textbf{x})\big)|  = \infty.$$ \
   
  \subsection{($l^2$ $ \rightarrow$ $L^p$)-continuity} \ \  Lemma \ref{L3} has the following extension.
  \begin{lemma} \label{L4}
For all $p > 2$, there exist $K_p >0$ such that\\
\begin{equation} \label{g16}
\|Q_A(\textbf{\emph{x}}) - Q_A(\textbf{\emph{y}}) \|_{L^p(\Omega_A,\mathbb{P}_A)} \ \leq \ K_p \|\textbf{\emph{x}} - \textbf{\emph{y}} \|_2, \ \ \ \ \ \textbf{\emph{x}} \in B_{l_{\mathbb{R}}^2(A)},  \ \ \textbf{\emph{y}} \in B_{l_{\mathbb{R}}^2(A)}.
\end{equation}
\end{lemma}
\vskip0.3cm
\begin{proof}  Consider the spectral decomposition
\begin{equation} \label{g18}
Q_A(\textbf{x}) - Q_A(\textbf{y}) = \sum_{k=0}^{\infty} (Q_A(\textbf{x}) - Q_A(\textbf{y}))_k,
\end{equation}
where
\begin{equation*}
\text{spect}\big((Q_A(\textbf{x}) - Q_A(\textbf{y}))_k\big) \subset W_{A,2k+1}, \ \ k = 0, 1, \ldots  \ .
\end{equation*}\\
By the estimate in \eqref{g15}, and by an application of the $L^2-L^p$ inequalities in ~\cite{bonami1970etude} to $(Q_A(\textbf{x}) - Q_A(\textbf{y}))_k$, we obtain
\begin{equation} \label{g17}
\|(Q_A(\textbf{x}) - Q_A(\textbf{y}))_k \|_{L^p(\Omega_A,\mathbb{P}_A)} \leq  \|\textbf{x} - \textbf{y} \|_2 \left(  \frac{2^{2k+1}}{(2k)!} \right)^{\frac{1}{2}} p^{\frac{2k+1}{2}}.
\end{equation}
The estimate in   \eqref{g16}  now follows from an application of the estimate in \eqref{g17} to each of the summands in \eqref{g18}.\\\
\end{proof}

\begin{corollary} The maps $\Theta_j$ defined in \eqref{com2}, and (therefore)  $\Phi$,  are  $(l^2 \rightarrow L^p)$-continuous for every $p \geq 2.$\\\
\end{corollary}

  \subsection{($l^2$ $ \rightarrow$ $L^{\infty}$)-continuity?} \label{contQ}\ \ 
Although bounded in the $L^{\infty}(\Omega_{A}, \mathbb{P}_{A})$-norm,  $\Phi$ is \emph{not} continuous with respect to it. To verify this, we use the following dichotomy.
 
 \begin{proposition} \label{P3} \ 
 For $\textbf{\emph{x}} \in B_{l^2_{\mathbb{R}}(A)}$,
 \begin{equation*}
 Q_A(\textbf{\emph{x}}) \in C(\Omega_A)  \ \Longleftrightarrow  \ \textbf{\emph{x}} \in l^1_{\mathbb{R}}(A).
  \end{equation*}
 \end{proposition}
 \begin{proof}
 If $\textbf{x} \in l^1_{\mathbb{R}}(A)$, then
 \begin{equation*}
 Q_A(\textbf{x}) \sim  \sum_{k=0}^{\infty}(-1)^k \bigg(\sum_{\substack{w \in W_{A,2k+1} \\[1mm] w = r_{\alpha_1} \cdots \ r_{\alpha_{2k+1}}}}\textbf{x}(\alpha_1) \cdots  \textbf{x}(\alpha_{2k+1}) r_{\alpha_1} \cdots r_{\alpha_{2k+1}}\bigg),
 \end{equation*}
 where
 \begin{equation*}
 \sum_{\substack{w \in W_{A,2k+1} \\[1mm] w = r_{\alpha_1} \cdots \ r_{\alpha_{2k+1}}}} 
 |\textbf{x}(\alpha_1) \cdots \textbf{x}(\alpha_{2k+1)}| \  \leq \ \frac{\|\textbf{x}\|_1^{2k+1}}{(2k+1)!}, \ \  \ k \geq 0
 \end{equation*}\\
 (cf. \eqref{e45} and \eqref{e44}),  which implies
 \begin{equation} \label{a1}
 \sum_{w \in W_A} |\widehat{Q_A(\textbf{x})}(w)| \ \leq \ \sinh(\|\textbf{x}\|_1),
 \end{equation} 
 and thus $ Q_A(\textbf{x}) \in C(\Omega_A)$.
 
 Now suppose $ Q_A(\textbf{x}) \in C(\Omega_A)$, and without loss of generality assume $|\textbf{x}(\alpha)| < 1$ for all $\alpha \in A$.   Let $E_n, \ n = 1, \dots, $ be a sequence of finite subsets of $A$, monotonically increasing to $E := \text{support} (\textbf{x})$. Then, $Q_A(\textbf{x}) = Q_E(\textbf{x})$, and (e.g., by~\cite[Corollary VII.9]{blei:2001}) $Q_{E_n}(\textbf{x}) \underset{n \to \infty}{\longrightarrow}Q_E(\textbf{x})$ uniformly on $\Omega_A$.  Therefore,
 \begin{equation*}
 \sum_{\alpha \in E_n}\log \big (1+\textbf{x}(\alpha)r_{\alpha} \big) \ \underset{n \to \infty}{\longrightarrow}  \ \sum_{\alpha \in E}\log \big (1+\textbf{x}(\alpha)r_{\alpha} \big)
 \end{equation*}
 uniformly on $\Omega_A$, which implies
  \begin{equation*}
 \sum_{\alpha \in E_n} \textbf{x}(\alpha)r_{\alpha} \ \underset{n \to \infty}{\longrightarrow}  \ \sum_{\alpha \in E}\textbf{x}(\alpha)r_{\alpha}
 \end{equation*}
 uniformly on $\Omega_A$, and hence $\sum_{\alpha \in A} |\textbf{x}(\alpha)| < \infty$.\\
 \end{proof}
 
 \begin{corollary} \label{C3} \ 
 For $\textbf{\emph{x}} \in B_{l^2(A)}$,
  \begin{equation*}
 \Phi(\textbf{\emph{x}}) \in C(\Omega_{A}) \Longleftrightarrow \textbf{\emph{x}} \in l^1(A).
  \end{equation*}
 \end{corollary}
 \begin{proof}
 If $ \textbf{x} \in l^1_{\mathbb{R}}(A)$, then by the recursive definition of $\textbf{x}^{(j)}$ and \eqref{a1}, $$\bigg(Q_{A_j}\big(\frac{\textbf{x}^{(j)}}{\delta^{j-1}}\big)\bigg)^{\wedge} \in l^1(W_{A_j}),$$ and therefore $$Q_{A_j}\big(\frac{\textbf{x}^{(j)}}{\delta^{j-1}}\big) \in C(\Omega_A),  \ \ j \geq 1.$$ Therefore, $\Phi(\textbf{x}) \in C(\Omega_{A})$.
 
 Conversely, if $\Phi(\textbf{x}) \in C(\Omega_{A})$, then $Q_{A_1}(\textbf{x}^{(1)}) \in C(\Omega_{A_1})$ (because the $W_{A_j}$ are independent), and therefore $\textbf{x} \in l^1(A)$ (by Proposition \ref{P3}).\\
   \end{proof}
   
   \begin{corollary} \label{C4}
    $\Phi$ is not $(l^2 \to L^{\infty})$-continuous.   
\end{corollary}
 \begin{proof}
Every $\textbf{x} \in l^2(A)$ is an $l^2$-norm limit of finitely supported members of $l^2(A)$.  But if $\textbf{x} \in B_{l^2(A)}$ has finite support, then  $\Phi(\textbf{x}) \in C(\Omega_{A})$, and   therefore, $(l^2 \to L^{\infty})$-continuity of $\Phi$  would contradict Corollary \ref{C3}.
 
 \end{proof}

\subsection{Linearization} \label{Lin} \ \ Consider the map  $$\Phi_2: l^2(A) \rightarrow L^{\infty}(\Omega_{A}, \mathbb{P}_{A})$$  defined in \eqref{e31e}.  This map is non-linear, but its image   $$\Phi_2[l^2] := \{\Phi_2(\textbf{x}): \textbf{x} \in l^2(A)\}$$ is norm-closed in $L^2(\Omega_{A} \mathbb{P}_{A})$ (via $(l^2 \to L^2)$-continuity), and (therefore) also norm-closed in  $L^{\infty}(\Omega_{A}, \mathbb{P}_{A})$.  We take the linear span of $\Phi_2[l^2]$, with the following equivalence relation in it:   for $f$ and  $g$  in  $\text{span}\big( \Phi_2[l^2]\big)$, 
 \begin{equation} \label{a8}
 f \equiv g \ \Longleftrightarrow  \  \  \int_{\Omega_{A}}f \ \Phi_2(\textbf{z}) d \mathbb{P}_A =  \int_{\Omega_{A}}g \ \Phi_2(\textbf{z}) d \mathbb{P}_A,  \ \ \  \forall \textbf{z}  \in l^2(A).
 \end{equation} 
 By \eqref{e20}, 
 \begin{equation*}
 f \equiv g \ \Longleftrightarrow \ \hat{f}|_{R_{A_1}} = \hat{g}|_{R_{A_1}}.
 \end{equation*}
 In particular, if $f  \in \text{span}\big(\Phi_2[l^2]\big)$, then there is a unique $\textbf{x} \in l^2(A)$ such that $f \equiv \Phi_2(\textbf{x}).$  We take the quotient space  
\begin{align*}
   H & :=\text{span}( \Phi_2[l^2])/ \equiv \\
   & = \text{span}( \Phi_2[l^2])/\{f \in  \text{span}( \Phi_2[l^2]): \hat{f}|_{R_{A_1}} = 0 \},
   \end{align*}   
  and consider the quotient map $\Tilde{\Phi}$,  
  \begin{equation*}
 \Tilde{\Phi}: l^2(A) \rightarrow H,
  \end{equation*}
where  $\Tilde{\Phi}(\textbf{x})$ is the equivalence class in $H$ whose class representative is $\Phi_2(\textbf{x})$.  The dot product $ \langle \cdot, \cdot \rangle_{l^2(A)}$ in $l^2(A)$ becomes the inner product $\gamma$ in $H$,  
 \begin{align} \label{a9}
 \gamma(f,g) & := \int_{\Omega_A}\Tilde{\Phi}(\textbf{x}_f)\Tilde{\Phi}(\bar{\textbf{x}}_g) d\mathbb{P}_A, \ \ f \in H, \ \ g \in H,\\\
 & =  \langle \textbf{x}_f, \textbf{x}_g \rangle_{l^2(A)},
 \end{align}
 where $\textbf{x}_f \in l^2(A)$ is given by
 \begin{equation*}
 \textbf{x}_f(\alpha) = \hat{f}(r_{\tau_0 \alpha}), \ \ \alpha \in A,
 \end{equation*}
 and  $\tau_0$ is defined in \eqref{g19}.   Equipped with $\gamma$, the resulting Hilbert space $H$ is unitarily equivalent (via $\Tilde{\Phi}$) to $l^2(A)$: \ for 
 $\textbf{x}$ and $\textbf{y}$ in $l^2(A)$, 
 \begin{equation}\label{a6}
 \begin{split}
 \langle \textbf{x},\textbf{y} \rangle  &\ =  \int_{\Omega_A}\Tilde{\Phi}(\textbf{x})\Tilde{\Phi}(\bar{\textbf{y}}) d\mathbb{P}_A \\\
 & \ \ \ \ \ \ \ \ \ \ \ \ \ \ \ \ \ \ \ \  (\text{by \eqref{e20}}) \\\\
 & \ =  \ \gamma \big(\Tilde{\Phi}(\textbf{x}),\Tilde{\Phi}(\textbf{y})\big)\\\
 & \ \ \ \ \ \ \ \ \ \  \ \ \ \ \ \ \ \ \ \ \ (\text{by \eqref{a9}}).
 \end{split}
 \end{equation}
 Moreover,  the Hilbert space norm induced by $\gamma$ on $H$ is equivalent (also via $\Tilde{\Phi}$) to the quotient $L^{\infty}$-norm on $H$, which is defined by
  \begin{equation}\label{a10}
 \|f\|_{\Tilde{L}^{\infty}} :=  \inf\{\|g\|_{L^{\infty}}: g \in \text{span}(\Tilde{\Phi}[l^2]), \ g \equiv f\}, \ \  f \in H.
  \end{equation}
  
  In summary,  the quotient map $\Tilde{\Phi}: l^2(A) \rightarrow H$ is linear, and satisfies 
  \begin{equation}\label{a7}
  \langle \textbf{x}, \textbf{y} \rangle \ =  \ \int_{\Omega_A}\Tilde{\Phi}(\textbf{x})\Tilde{\Phi}(\bar{\textbf{y}}) d\mathbb{P}_A, \ \ \ \ \ {\bf{x}} \in l^2(A), \ {\bf{y}} \in l^2(A), 
  \end{equation}
  where integrands are equivalence class representatives, and the integral is well defined by   \eqref{a8}.  Like $\Phi_2$,  the quotient map $\Tilde{\Phi}$ does not commute with complex conjugation.  In particular, the norm induced on $H$ by $\gamma$ is \emph{not} the same as the norm induced by the usual inner product in $L^2(\Omega_A,\mathbb{P}_A)$.  (See Remark \ref{R4e}.ii.)  Following \eqref{e31}, we also have 
  
  \begin{equation}\label{a11}
 \| \Tilde{\Phi}(\textbf{x})\|_{\Tilde{L}^{\infty}} \leq K \|\textbf{x}\|_2, \ \ \ \textbf{x} \in l^2(A).
 \end{equation}\\
 Finally,  the quotient map $\Tilde{\Phi}$ is continuous with respect to the norm topologies in the Hilbert spaces $l^2(A)$ and $H$, via the $(l^2 \rightarrow L^2)$-continuity of  $\Phi_2$ defined in \eqref{e31e}, and  continuous with respect to their weak topologies, via the unitary equivalence in \eqref{a6}.\\\ 

\section{\bf{Integrability}}\label{sI}    Next we consider vector-valued integrability of $\Phi$ and $\Phi \otimes \Phi$  on  $B_{l^2}$ and $B_{l^2} \times B_{l^2}$,  where $\Phi$ (of Theorem \ref{T2}) is viewed as an  $L^{\infty}(\Omega_A, \mathbb{P}_A)$-valued function on the unit ball $B_{l^2} := B_{l^2(A)}$ with the weak topology.  
 
 \subsection{Integrability of $\Phi$} \ We let  $M(B_{l^2})$ be the space of complex measures on $\mathscr{B}_{B_{l^2}}$, where $\mathscr{B}_{B_{l^2}}$ is the Borel field in  $B_{l^2}$ generated by the weak topology on $l^2(A)$.  We proceed to define      
 \begin{equation}\label{b3}
 \int_{B_{l^2}} \Phi({\bf{x}})\mu(d{\bf{x}}), \ \ \ \mu \in M(B_{l^2}),
 \end{equation}
 as $L^{\infty}$-valued integrals, specifically as weak* limits in $L^{\infty}(\Omega_A, \mathbb{P}_A)$.  
 
 For certain $\mu \in M(B_{l^2})$, e.g., for discrete measures, or measures supported on $l^1(A)$, the integrals in \eqref{b3} can be defined numerically, point-wise almost surely on $(\Omega_A, \mathbb{P}_A)$.  But for arbitrary $\mu \in M(B_{l^2})$, because the integrands $\Phi(\cdot)(\omega)$ in \eqref{b3} cannot be shown to be $\mathscr{B}_{B_{l^2}}$-measurable, say for almost all $\omega \in (\Omega_A,\mathbb{P}_A)$, these integrals in general can be obtained only "weakly." 

\begin{proposition} \label{P6}
For $\mu \in M(B_{l^2})$, the Walsh series
\begin{equation}\label{a12}
\begin{split}
\int_{B_{l^2}} \Phi({\bf{x}}) \mu (d{\bf{x}}) &:=  \sum_{w \in W_A} \bigg ( \int_{B_{l^2}} \widehat{\Phi({\bf{x}})}(w) \mu(d{\bf{x}}) \bigg)w\\\
& =  \sum_{j = 1}^{\infty}  \sum_{w \in W_{A_j}} \bigg ( \int_{B_{l^2}} \widehat{\Theta_j({\bf{x}}})(w) \mu(d{\bf{x}}) \bigg )w
\end{split} 
\end{equation}
represents an element of $L^{\infty}(\Omega_A, \mathbb{P}_A)$, such that
\begin{equation} \label{a13}
 \big \| \int_{B_{l^2}} \Phi({\bf{x}}) \mu (d{\bf{x}})   \big \|_{L^{\infty}} \leq \big ( \frac{2\sqrt{e}}{1 - \delta} \big) \| \mu \|_M,
\end{equation}
where $\Theta_j$ is given in \eqref{com2}, $\|\cdot\|_M$ is the usual total variation norm, and $\delta$ is given in \eqref{e42e}.\\\

\end{proposition} 
\noindent
Proposition \ref{P6} is a consequence of the following.\\\

\begin{lemma} \label{L8}
For $\mu \in M(B_{l^2})$ and  $j \in \mathbb{N}$, the Walsh series\\\
\begin{equation}\label{a17}
 \int_{B_{l^2}} \Theta_j({\bf{x}}) \mu (d{\bf{x}}) :=  \sum_{w \in W_{A_j}} \left ( \int_{B_{l^2}} \widehat{\Theta_j({\bf{x}}})(w) \mu(d{\bf{x}}) \right )w
\end{equation}\\\
represents an element of $L^{\infty}(\Omega_A, \mathbb{P}_A)$, and 
\begin{equation} 
 \big \| \int_{B_{l^2}} \Theta_j({\bf{x}}) \mu (d{\bf{x}})    \big \|_{L^{\infty}} \ \leq \  2 \delta^{j-1} \sqrt{e} \  \| \mu \|_M .
\end{equation}
\end{lemma}
\begin{proof} \ By the definition of ${\bf{x}}^{(j)}$ in \eqref{g2}, by Lemma \ref{L3W}, and by induction on $j$, the $l^2$-valued function 
 
 \begin{equation*}
 {\bf{x}} \mapsto {\bf{x}}^{(j)}, \ \ \ {\bf{x}} \in B_{l^2},
 \end{equation*}\\\
 is weakly continuous on $B_{l^2}$.  Therefore, by the definition of $\Theta_j$ in \eqref{com2}, the (scalar-valued) integrands on the right side of \eqref{a17} are continuous on $B_{l^2}$ (with respect to the weak topology), which implies that the (numerical) integrals on the right side of \eqref{a17} are well defined.

For ${\bf{x}} \in B_{l^2},$ and $n = 1, \ldots,$   let
\begin{equation} \label{a24}
\rho_n{\bf{x}} = \mathbf{1}_{\{|{\bf{x}}| \geq \frac{1}{n}\}}{\bf{x}}, 
\end{equation}
whose supports are (obviously!) finite.   Then, we have the Walsh polynomial
\begin{equation} \label{a14}
 \int_{B_{l^2}} \Theta_j(\rho_n{\bf{x}}) \mu (dx)  =    \sum_{w \in W_{A_j}} \left ( \int_{B_{l^2}} \big(\Theta_j(\rho_n{\bf{x}}) \big )^{\wedge}(w) \mu(d{\bf{x}}) \right )w,\\\\
\end{equation}
and by \eqref{a45},\\\
\begin{equation} \label{a15}
\big \| \int_{B_{l^2}} \Theta_j(\rho_n{\bf{x}}) \mu (dx) \big \|_{L^{\infty}} \leq 2\delta^{j-1}  \sqrt{e} \| \mu \|_M, \ \ \  \ n = 1, \ldots \ .\\\\
\end{equation}\\\
Letting $n \rightarrow \infty$, we have $\rho_n{\bf{x}} \rightarrow {\bf{x}}$  in the $l^2$-norm, and therefore (e.g., by Lemma \ref{L3W}) for every $w \in W_{A_j}$,\\\
\begin{equation}\label{a16}
\int_{B_{l^2}} \big(\Theta_j(\rho_n{\bf{x}}) \big )^{\wedge}(w) \mu(d{\bf{x}}) \ \underset{n \to \infty}{\longrightarrow} \  \int_{B_{l^2}} \widehat{\Theta_j({\bf{x}})}(w) \mu(d{\bf{x}}).\\\\\
\end{equation}\\\
Finally, we deduce from  \eqref{a15} and \eqref{a16} that the polynomials in \eqref{a14} converge in the weak* topology of $L^{\infty}(\Omega_A,\mathbb{P}_A)$ to an element in  $L^{\infty}(\Omega_A,\mathbb{P}_A)$,  whose Walsh series is the right side of \eqref{a17}. 
\end{proof}

\subsection{Integrability of $\Phi \otimes \Phi$} \label {sbm} \ Next we check vector-valued integrability of $\Phi \otimes \Phi$, where
\begin{equation*}
(\Phi \otimes \Phi)({\bf{x}},{\bf{y}}) = \Phi({\bf{x}})\Phi({\bf{y}}), \ \ \ ({\bf{x}},{\bf{y}}) \in B_{l^2} \times B_{l^2}.
\end{equation*}
   Mimicking integration with respect to measures (Proposition \ref{P6} above), we integrate analogously over $B_{l^2} \times B_{l^2}$ with respect to $\mathcal{F}_2$-\emph{measures} (known also as \emph{bimeasures}).  These are scalar-valued set-functions $\mu$ defined on  "rectangles" $(E, F) \in \mathscr{B}_{l^2}\times \mathscr{B}_{l^2}$, such that for every $E \in \mathscr{B}_{l^2}$ and $ F \in \mathscr{B}_{l^2}$,  
\begin{equation*}
\mu(E, \cdot) \ \ \ \text{and} \ \ \ \mu(\cdot, F)
\end{equation*}
are, respectively, complex measures on $\mathscr{B}_{l^2}$.  The space of such set-functions is a generalization of  $\mathcal{F}_2(X \times Y)$ defined in \eqref{b1}, and is denoted by  $\mathcal{F}_2(\mathscr{B}_{l^2}, \mathscr{B}_{l^2})$.  Given bounded scalar-valued measurable functions $f$ and $g$ on $B_{l^2}$, and $\mu \in \mathcal{F}_2(\mathscr{B}_{l^2}, \mathscr{B}_{l^2})$,  we have a Lebesgue-type  bilinear integral 
\begin{equation} \label{a20}
\int_{B_{l^2} \times B_{l^2}} f({\bf{x}}) g({\bf{y}}) \mu(d{\bf{x}},d{\bf{y}}),
\end{equation}
which is computed iteratively:  first with respect to ${\bf{x}}$ and then with respect to ${\bf{y}}$, or vise versa.  The $\mathcal{F}_2$-norm of $\mu$ is 
\begin{equation} \label{a18}
\|\mu\|_{\mathcal{F}_2} := \sup \bigg \{\big|\int_{B_{l^2} \times B_{l^2}} f({\bf{x}}) g({\bf{y}}) \mu(d{\bf{x}},d{\bf{y}}) \big|: \|f\|_{\mathcal{L}^{\infty}} \leq 1, \   \|g\|_{\mathcal{L}^{\infty}} \leq 1\bigg \},\\\
\end{equation}\\\
where $f$ and $g$  in \eqref{a18} are measurable functions on $B_{l^2}$, and $\|\cdot\|_{\mathcal{L}^{\infty}}$ denotes  supremum over $B_{l^2}$.   ($\mathcal{F}$ is for  Maurice Fr\'echet, first to study bilinear analogs of measures, and in particular, first to construct integrals of the type in \eqref{a20}; see ~\cite{Frechet:1915},  ~\cite[Ch. VI]{blei:2001}, and also the discussion leading to Lemma \ref{L1}.)  

To verify the two-variable analog of Proposition \ref{P6}, we use

\begin{theorem}[The Grothendieck factorization theorem] \label{T5}
\ If  $\mu \in \mathcal{F}_2(\mathscr{B}_{l^2}, \mathscr{B}_{l^2})$, then there exist probability measures $\nu_1$ and $\nu_2$ on $\mathscr{B}_{l^2}$ such that for all bounded measurable functions $f$ and $g$ on $B_{l^2}$,
\begin{equation} \label{a19} 
\big|\int_{B_{l^2} \times B_{l^2}} f({\bf{x}}) g({\bf{y}}) \mu(d{\bf{x}},d{\bf{y}}) \big| \leq \mathcal{K}_G \|\mu\|_{\mathcal{F}_2} \|f\|_{L^2(\nu_1)}  \|g\|_{L^2(\nu_2)},
\end{equation}
where $\mathcal{K}_G$ is the constant in \eqref{e16}. 
\end{theorem}
\noindent
Theorem \ref{T5} is equivalent to the Grothendieck inequality: either statement  is derivable from the other, with the same constant $\mathcal{K}_G$ in both; e.g., see ~\cite[Ch. V]{blei:2001}.

Fixing an arbitrary $\mu \in \mathcal{F}_2(\mathscr{B}_{l^2}, \mathscr{B}_{l^2})$,  we write (formally) 
\begin{equation}\label{a21}
\begin{split}
\int_{B_{l^2} \times B_{l^2}}&\Phi({\bf{x}})\Phi({\bf{y}}) \mu(d{\bf{x}},d{\bf{y}})  \\\
&:=  \sum_{\substack{j,k = 1 \\[1mm] j \neq k}}^{\infty}  \int_{B_{l^2} \times B_{l^2}} \Theta_j({\bf{x}})  \Theta_k({\bf{y}})  \mu(d{\bf{x}},d{\bf{y}}) \\\
& \ \ \ \ \ \ \ \ \ \ \ + \sum_{k=1}^{\infty}  \int_{B_{l^2} \times B_{l^2}} \Theta_k({\bf{x}})  \Theta_k({\bf{y}})  \mu(d{\bf{x}},d{\bf{y}}),
\end{split}
\end{equation}\\
and proceed to verify that the integrals in the two sums on the right side of \eqref{a21} are elements in $L^{\infty}(\Omega_A,\mathbb{P}_A)$, and that the sums converge in the $L^{\infty}$-norm.

The first sum on the right side of \eqref{a21} is handled by the following.

\begin{lemma} \label{L9}
For $\mu \in \mathcal{F}_2(\mathscr{B}_{l^2}, \mathscr{B}_{l^2})$ and positive integers  $j \neq k$,  the Walsh series\\\
\begin{equation}\label{a22}
\begin{split}
&\int_{B_{l^2} \times B_{l^2}} \Theta_j({\bf{x}})  \Theta_k({\bf{y}})  \mu(d{\bf{x}},d{\bf{y}}) \\\
 &:= \sum_{w_1 \in W_{A_j}, \ w_2 \in W_{A_k}} \left ( \int_{B_{l^2} \times B_{l^2}}   \widehat{\Theta_j({\bf{x}})}(w_1)\widehat{\Theta_k({\bf{y}})}(w_2)\mu(d{\bf{x}},d{\bf{y}}) \right )w_1w_2
\end{split}
\end{equation}\\\
represents an element of $L^{\infty}(\Omega_A, \mathbb{P}_A)$, and\\\\
\begin{equation} \label{a23c}
 \bigg \| \int_{B_{l^2} \times B_{l^2}} \Theta_j({\bf{x}})  \Theta_k({\bf{y}})  \mu(d{\bf{x}},d{\bf{y}})   \bigg \|_{L^{\infty}} \leq  4e \delta^{j + k -2} \| \mu \|_{\mathcal{F}_2} .
\end{equation}\\\
\end{lemma}
\begin{proof}[Sketch of proof] \ 
For $w_1 \in W_{A_j}$,  $w_1^{\prime} \in W_{A_j}$, and $w_2 \in W_{A_k}$,  $w_2^{\prime} \in W_{A_k}$, \\
\begin{equation*}
w_1 w_2 = w_1^{\prime} w_2^{\prime} \ \ \  \Longrightarrow \ \ \ w_1 = w_1^{\prime}, \ \  w_2 = w_2^{\prime},
\end{equation*}\\\
because $W_{A_j}$ and $W_{A_k}$ are mutually independent.  Therefore, the series on the right side of \eqref{a22} is a \emph{bona fide} Walsh series.  

The proof of Lemma \ref{L9} is nearly identical to the proof Lemma \ref{L8}:    $\mu \in M(B_{l^2})$ and the integrals with respect to it in Lemma \ref{L8}  are replaced by  $\mu \in \mathcal{F}_2(\mathscr{B}_{l^2}, \mathscr{B}_{l^2})$  and the bilinear integrals with respect to it  (as per \eqref{a20} and  \eqref{a18}).  Details are omitted.

\end{proof}

The second sum on the right side of \eqref{a21} requires an intervention of Theorem \ref{T5}.

\begin{lemma} \label{L10}
For $\mu \in \mathcal{F}_2(\mathscr{B}_{l^2}, \mathscr{B}_{l^2})$ and $k \in \mathbb{N}$,  the Walsh series\\\
\begin{equation}\label{a29}
\begin{split}
&\int_{B_{l^2} \times B_{l^2}}  \Theta_k({\bf{x}})  \Theta_k({\bf{y}})  \mu(d{\bf{x}},d{\bf{y}}) \\\
 &:= \sum_{w \in W_{A_k}} \bigg ( \sum_{\substack{w_1 \in W_{A_k}, \ w_2 \in W_{A_k} \\[1mm] w_1 w_2 = w}}   \int_{B_{l^2} \times B_{l^2}}   \widehat{\Theta_k({\bf{x}})}(w_1)\widehat{\Theta_k({\bf{y}})}(w_2)\mu(d{\bf{x}},d{\bf{y}})    \bigg )w
\end{split}
\end{equation}\\\
represents an element of $L^{\infty}(\Omega_A, \mathbb{P}_A)$, and\\
\begin{equation} \label{a23}
 \bigg \| \int_{B_{l^2} \times B_{l^2}}\Theta_k({\bf{x}})  \Theta_k({\bf{y}})  \mu(d{\bf{x}},d{\bf{y}})    \bigg \|_{L^{\infty}} \leq  4e \delta^{2k-2} \| \mu \|_{\mathcal{F}_2} .
\end{equation}
\end{lemma}
\begin{proof}
 We first verify that for fixed $w \in W_{A_k}$,  the sum on the right side of \eqref{a22}
\begin{equation}\label{a26}
\begin{split}
&\sum_{\substack{w_1 \in W_{A_k}, \ w_2 \in W_{A_k} \\[1mm] w_1 w_2 = w}}  \bigg ( \int_{B_{l^2} \times B_{l^2}}   \widehat{\Theta_k({\bf{x}})}(w_1)\widehat{\Theta_k({\bf{y}})}(w_2)\mu(d{\bf{x}},d{\bf{y}}) \bigg )w_1w_2 \\\\
&= \sum_{w_1 \in W_{A_k}} \bigg ( \int_{B_{l^2} \times B_{l^2}}   \widehat{\Theta_k({\bf{x}})}(w_1)\widehat{\Theta_k({\bf{y}})}(w_1w)\mu(d{\bf{x}},d{\bf{y}}) \bigg )w_1\\
&
\end{split}
\end{equation} 
 converges absolutely.  To this end, let $\nu_1$ and $\nu_2$ be probability measures on $\mathscr{B}_{l^2}$ associated with $\mu$, as per Theorem \ref{T5}.  Then by \eqref{a19},  for every $w_1 \in W_{A_k}$ we estimate 
\begin{equation}\label{a25}
\begin{split}
 &  \big | \int_{B_{l^2} \times B_{l^2}}   \widehat{\Theta_k({\bf{x}})}(w_1)\widehat{\Theta_k({\bf{y}})}(w_1w)\mu(d{\bf{x}},d{\bf{y}}) \big | \\\\
& \leq  \   \mathcal{K}_G \|\mu\|_{\mathcal{F}_2} \bigg ( \int_{B_{l^2}}   \big | \widehat{\Theta_k({\bf{x}})}(w_1) \big|^2 \nu_1(d{\bf{x}}) \bigg )^{\frac{1}{2}} \bigg ( \int_{B_{l^2}} \big |\widehat{\Theta_k({\bf{y}})}(w_1w) \big|^2 \nu_2(d{\bf{y}}) \bigg )^{\frac{1}{2}}.\\  
\end{split}
\end{equation} \\\
Applying the estimate in  \eqref{a25} to the summands on the right side of \eqref{a26}, and then applying Cauchy-Schwarz, Plancherel, and \eqref{a45},  we obtain\\
\begin{equation}
\begin{split}
 &\sum_{w_1 \in W_{A_k}} \big | \int_{B_{l^2} \times B_{l^2}}   \widehat{\Theta_k({\bf{x}})}(w_1)\widehat{\Theta_k({\bf{y}})}(w_1w)\mu(d{\bf{x}},d{\bf{y}}) \big | \\\\
 & \leq \   \mathcal{K}_G \|\mu\|_{\mathcal{F}_2} \bigg (\sum_{w_1 \in W_{A_k}} \int_{B_{l^2}}   \big | \widehat{\Theta_k({\bf{x}})}(w_1) \big|^2 \nu_1(d{\bf{x}}) \bigg )^{\frac{1}{2}} \bigg (\sum_{w_1 \in W_{A_k}} \int_{B_{l^2}} \big |\widehat{\Theta_k({\bf{y}})}(w_1w) \big|^2 \nu_2(d{\bf{y}}) \bigg )^{\frac{1}{2}}\\\\
 &= \   \mathcal{K}_G \|\mu\|_{\mathcal{F}_2} \left ( \int_{B_{l^2}} \bigg ( \sum_{w_1 \in W_{A_k}} \big | \widehat{\Theta_k({\bf{x}})}(w_1) \big|^2 \bigg ) \nu_1(d{\bf{x}}) \right )^{\frac{1}{2}} \left (  \int_{B_{l^2}} \bigg ( \sum_{w_1 \in W_{A_k}} \big |\widehat{\Theta_k({\bf{y}})}(w_1w) \big|^2 \bigg ) \nu_1(d{\bf{y}}) \right )^{\frac{1}{2}}\\\\
 & \leq \  \mathcal{K}_G \|\mu\|_{\mathcal{F}_2} \left ( \int_{B_{l^2}} \big \| \Theta_k({\bf{x}})\big \|_{L^{\infty}(\Omega_A, \mathbb{P}_A)}^2 \ \nu_1(d{\bf{x}}) \right )^{\frac{1}{2}} \left ( \int_{B_{l^2}} \big \|\Theta_k({\bf{y}})\big \|_{L^{\infty}(\Omega_A, \mathbb{P}_A)}^2 \ \nu_1(d{\bf{y}}) \right )^{\frac{1}{2}}.\\\\
 &\leq \  4 e \delta^{2(k-1)} \mathcal{K}_G \|\mu\|_{\mathcal{F}_2},
\end{split}
\end{equation}\\\
which verifies that \eqref{a26} converges absolutely.

We now proceed as in the proofs of Lemmas \ref{L8} and  \ref{L9}.  For $w \in W_{A_k} \ w_1 \in W_{A_k}$,\\ 
\begin{equation}
\begin{split}
\int_{B_{l^2} \times B_{l^2}} \big(\Theta_k(\rho_n{\bf{x}}) &\big )^{\wedge}(w_1) \big(\Theta_k(\rho_n{\bf{y}}) \big )^{\wedge}(w_1w)\mu(d{\bf{x}},d{\bf{y}})\\\
& \underset{n \to \infty}{\longrightarrow} \  \ \int_{B_{l^2} \times B_{l^2}} \widehat{\Theta_k({\bf{x}})}(w_1) \widehat{\Theta_k({\bf{y}})}(w_1w)\mu(d{\bf{x}},d{\bf{y}}),
\end{split}
\end{equation}\\\
where $\rho_n$ is defined in \eqref{a24}; cf. \eqref{a16}.  Therefore, via the absolute convergence of  \eqref{a26}, we obtain
\begin{equation}\label{a27}
\begin{split}
\bigg (\int_{B_{l^2} \times B_{l^2}} &\Theta_k(\rho_n{\bf{x}}) \Theta_k(\rho_n{\bf{y}})  \mu(d{\bf{x}},d{\bf{y}}) \bigg )^{\wedge}(w) \\\
&\underset{n \to \infty}{\longrightarrow} \  \  \sum_{\substack{w_1 \in W_{A_k}, \ w_2 \in W_{A_k} \\[1mm] w_1 w_2 = w}} \int_{B_{l^2} \times B_{l^2}} \widehat{\Theta_k({\bf{x}})}(w_1)  \widehat{\Theta_k({\bf{y}})}(w_2)\mu({d\bf{x}},d{\bf{y}})  \\\
  & \ \ \ \ \ \ \ \ \ \ \ \ \ \ \ \ \ \ \ \ \ \ \ \ \ \ \ \ = \ \ \bigg (\int_{B_{l^2} \times B_{l^2}}  \Theta_k({\bf{x}})  \Theta_k({\bf{y}})  \mu(d{\bf{x}},d{\bf{y}})  \bigg )^{\wedge}(w).
\end{split}
\end{equation}\\\
Finally, by combining \eqref{a27} with the estimate\\
\begin{equation}\label{a30}
\bigg \| \int_{B_{l^2} \times B_{l^2}} \Theta_k(\rho_n{\bf{x}}) \Theta_k(\rho_n{\bf{y}})  \mu(d{\bf{x}},d{\bf{y}}) \bigg \|_{L^{\infty}} \leq 4e \delta^{2(k-1)} \|\mu\|_{\mathcal{F}_2}
\end{equation}\\\
 (cf. \eqref{a15} and \eqref{a16}), we deduce that the Walsh polynomials
 
\begin{equation*}  
\int_{B_{l^2} \times B_{l^2}} \Theta_k(\rho_n{\bf{x}}) \Theta_k(\rho_n{\bf{y}})  \mu(d{\bf{x}},d{\bf{y}}), \ \ \ n = 1, \ldots,
\end{equation*}\\\
converge in the weak* topology of $L^{\infty}(\Omega_A,\mathbb{P}_A)$ to an element in $L^{\infty}(\Omega_A,\mathbb{P}_A)$,  whose Walsh series is the right side of \eqref{a29}.  The $L^{\infty}$-bound in \eqref{a23} is a consequence of the  estimate in \eqref{a30}.

\end{proof}

 Applying Lemmas \ref{L9} and  \ref{L10} to \eqref{a21}, we obtain\\
 
 \begin{proposition} \label{P7} \ For $\mu \in \mathcal{F}_2(\mathscr{B}_{l^2},\mathscr{B}_{l^2})$,
\begin{equation}\label{a28}
\int_{B_{l^2} \times B_{l^2}}  \Phi({\bf{x}})\Phi({\bf{y}}) \mu(d{\bf{x}},d{\bf{y}})  
:=  \sum_{j,k=1}^{\infty}  \int_{B_{l^2} \times B_{l^2}}\Theta_j({\bf{x}})  \Theta_k({\bf{y}})  \mu(d{\bf{x}},d{\bf{y}}) \\
\end{equation}\\
is an element of $L^{\infty}(\Omega_A, \mathbb{P}_A)$, where the integrals on the right side of \eqref{a28} are given by Lemmas \ref{L9} and \ref{L10}.  Moreover,
\begin{equation} \label{a31}
 \bigg \| \int_{B_{l^2} \times B_{l^2}}  \Phi({\bf{x}})\Phi({\bf{y}}) \mu(d{\bf{x}},d{\bf{y}})   \bigg \|_{L^{\infty}} \leq \bigg ( \frac{2\sqrt{e}}{1 - \delta} \bigg)^2 \| \mu \|_{\mathcal{F}_2} \ .
\end{equation}\\
 \end{proposition}
 
  \subsection{Integrability of the inner product} \label{sint}
The classical Grothendieck inequality is in effect the statement that the inner product $\langle \cdot, \cdot \rangle$ in a Hilbert space $H$ is canonically integrable with respect to all discrete $\mathcal{F}_2$-measures on $B_H \times B_H$.  Namely,   
by the Grothendieck inequality and basic harmonic analysis, if $a = (a_{jk}) \in \mathcal{F}_2(\mathbb{N},\mathbb{N})$, i.e.,
\begin{equation} \label{a32}
\sup \bigg \{\big |\sum_{j,k=1}^N a_{jk}s_jt_k \big |:  (s_j,t_k) \in [-1,1]^2, \  (j,k) \in [N]^2, \  N \in \mathbb{N} \bigg \} := \|a\|_{\mathcal{F}_2} < \infty,
\end{equation}\\\
then for all  ${\bf{x}}_j \in B_H$ and $ {\bf{y}}_k \in B_H,  \ j \in \mathbb{N},\  k \in \mathbb{N}$, 
\begin{equation} \label{a33}
\begin{split}
\lim_{N \rightarrow \infty} \sum_{j,k=1}^N a_{jk} \langle {\bf{x}}_j, {\bf{y}}_k \rangle & = \sum_{j=1}^{\infty} \bigg (\sum_{k=1}^{\infty}  a_{jk} \langle {\bf{x}}_j, {\bf{y}}_k \rangle \bigg ) \\\
& =  \sum_{k=1}^{\infty} \bigg ( \sum_{j=1}^{\infty}  a_{jk} \langle {\bf{x}}_j, {\bf{y}}_k \rangle \bigg ) \\\
& := \sum_{j,k = 1}^{\infty}a_{jk} \langle {\bf{x}}_j, {\bf{y}}_k \rangle,
\end{split}
\end{equation}
 and
\begin{equation} \label{a34}
\big |\sum_{j,k = 1}^{\infty}a_{jk} \langle {\bf{x}}_j, {\bf{y}}_k \rangle \big | \leq \mathcal{K}_G \|a\|_{\mathcal{F}_2}.
\end{equation}\\ 
The statements in \eqref{a33} and \eqref{a34}  can be equivalently rephrased as follows.  Let $\mathscr{B}_{B_H}$ be the Borel field generated by the weak topology on $B_H$.  Given $a = (a_{jk}) \in \mathcal{F}_2(\mathbb{N},  \mathbb{N})$, \  $\{{\bf{x}}_j\} \subset B_H$, and $\{{\bf{y}}_k\} \subset B_H$, we define $\mu_a \in \mathcal{F}_2(\mathscr{B}_{B_H} , \mathscr{B}_{B_H})$ by
\begin{equation} \label{a35}
\mu_a = \sum_{j,k}a_{jk} \delta_{{\bf{x_j}}} \otimes \delta_{{\bf{y_k}}},
\end{equation}
where $\delta_{{\bf{x}}}$  denotes the usual point mass measure at ${\bf{x}}  \in B_H$; \ i.e., 
\begin{equation*}
\mu_a(E,F) \  = \sum_{\{{\bf{x_j}} \in E, \ {\bf{y_k}} \in F\}} a_{jk}, \ \ \  (E,F) \in \mathscr{B}_{B_H} \times \mathscr{B}_{B_H},
\end{equation*}
whence $$\|\mu_a\|_{\mathcal{F}_2(\mathscr{B}_{B_H} , \mathscr{B}_{B_H})} = \|a\|_{\mathcal{F}_2(\mathbb{N},  \mathbb{N})}.$$\\  Then, the function $\langle \cdot, \cdot \rangle$ on $B_H \times B_H$ is integrable with respect to $\mu_a$,
\begin{equation} \label{a36}
\int_{B_H \times B_H} \langle {\bf{x}}, {\bf{y}} \rangle  \ \mu_a(d {\bf{x}}, d {\bf{y}}) := \sum_{j,k = 1}^{\infty}a_{jk} \langle {\bf{x}}_j, {\bf{y}}_k \rangle,
\end{equation}
\and
\begin{equation} \label{a37}
\big |\int_{B_H \times B_H} \langle {\bf{x}}, {\bf{y}} \rangle  \ \mu_a(d {\bf{x}}, d {\bf{y}}) \big | \leq \ \mathcal{K}_G \|\mu_a\|_{\mathcal{F}_2}.
\end{equation}\\\

For arbitrary $\mu \in \mathcal{F}_2(\mathscr{B}_{B_H} , \mathscr{B}_{B_H})$, if  $H$ is separable and  $A = \{{\bf{e}}_j\}_{j \in \mathbb{N}}$ is an orthonormal basis for it,  then Theorem \ref{T5} implies that the limit
\begin{equation} \label{b21}
\int_{B_H \times B_H} \langle {\bf{x}}, {\bf{y}} \rangle  \ \mu(d {\bf{x}}, d {\bf{y}}) := \lim_{n \rightarrow \infty} \sum_{j=1}^n \int_{B_H \times B_H}\langle {\bf{x}}, {\bf{e}}_j\rangle  \ \langle {\bf{y}}, {\bf{e}}_j\rangle\ \mu(d {\bf{x}}, d {\bf{y}})
\end{equation}
exists, and is independent of the basis $A$.  (E.g., see \cite{bowers2010measure}.)
 
Given a non-separable $H$, we  fix an orthonormal basis $A$ for it, and take $\Phi_A := \Phi$ to be the map supplied by Theorem \ref{T2}.  Then, by mimicking the argument in \eqref{i5} with $\mu \in \mathcal{F}_2(\mathscr{B}_{B_H}, \mathscr{B}_{B_H})$,  we  formally have 
\begin{equation} \label{a46}
\begin{split}
\int_{B_H \times B_H} \langle {\bf{x}}, {\bf{y}} \rangle  \ \mu(d {\bf{x}}, d {\bf{y}}) & = \int_{B_H \times B_H} \bigg (\int_{\Omega_A} \Phi_A({\bf{x}}) \Phi_A(\bar{{\bf{y}}}) d \mathbb{P}_A \bigg) \mu(d {\bf{x}}, d {\bf{y}}) \\\\
& := \int_{\Omega_A} \bigg (\int_{B_{l^2} \times B_{l^2}} \Phi_A({\bf{x}}) \Phi_A(\bar{{\bf{y}}})  \mu(d {\bf{x}}, d {\bf{y}}) \bigg ) d \mathbb{P}_A,
\end{split}
\end{equation}\\
and\\\
\begin{equation} \label{a41}
\begin{split}
\bigg| \int_{B_H \times B_H} \langle {\bf{x}},{\bf{y}} \rangle \ \mu({\bf{x}},d{\bf{y}}) \bigg | &\leq \ \bigg \| \int_{B_{l^2} \times B_{l^2}} \Phi_A({\bf{x}}) \Phi_A(\bar{{\bf{y}}})  \mu(d {\bf{x}}, d {\bf{y}}) \bigg \|_{L^{\infty}} \\\\
& \leq \bigg ( \frac{2\sqrt{e}}{1 - \delta} \bigg)^2 \| \mu \|_{\mathcal{F}_2} .
\end{split}
\end{equation} \\\\

\begin{problem} If $A$ and $A^{\prime}$ are orthonormal bases of $H$, and  $\Phi_A$ and $\Phi_{A^{\prime}}$ are the respective maps supplied  by Theorem \ref{T2}, then does it follow that for every $\mu \in \mathcal{F}_2(\mathscr{B}_{B_H}, \mathscr{B}_{B_H})$, \\\
\begin{equation} \label{a39}
\begin{split}
\int_{\Omega_A} \bigg (\int_{B_{l^2(A)} \times B_{l^2(A)}}  \Phi_A({\bf{x}})\Phi_A({\bf{y}})& \mu(d{\bf{x}},d{\bf{y}}) \bigg) d\mathbb{P}_A\\ & = \int_{\Omega_{A^{\prime}}} \bigg (\int_{B_{l^2(A^{\prime})} \times B_{l^2(A^{\prime})}}  \Phi_{A^{\prime}}({\bf{x}})\Phi_{A^{\prime}}({\bf{y}}) \mu(d{\bf{x}},d{\bf{y}}) \bigg) d\mathbb{P}_{A^{\prime}} \ ?
\end{split}
\end{equation} \\
\end{problem}
\noindent
Indeed, the question whether the inner product in a non-separable  $H$ is integrable with respect to arbitrary $\mu \in \mathcal{F}_2(\mathscr{B}_{B_H},\mathscr{B}_{B_H})$ -- say, via \eqref{a46} -- appears to be open.\\\

\section{\bf{A Parseval-like formula for $\langle {\bf{x}}, {\bf{y}}\rangle$, \  ${\bf{x}} \in l^p$, \ ${\bf{y}} \in l^q$}} \label{Snonh}

 Next we consider a variant of Theorem \ref{T2} concerning an integral representation of the dual action between $l^p$ and $l^q$ for  \  $1 \leq p \leq  2 \leq q \leq \infty$, \ $\frac{1}{p} + \frac{1}{q} = 1$. 
  
  The Grothendieck inequality, an assertion about the dual action between $l^2$ and $l^2$, cannot be extended verbatim to the general $l^p - l^q$ setting.  To wit,  if we take $(a_{jk})$ in  \eqref{i1}  to be the matrix in \eqref{e17}, and  ${\bf{x}}_j$  to be the $j$-th basic unit vector  ${\bf{e}}_j$  in $l^p(\mathbb{N})$,  then, maximizing over vectors ${\bf{y}}_k$ in the unit ball of $l^q(\mathbb{N})$, we obtain that the left side of \eqref{i1} majorizes
  \begin{equation}
  \sum_{k=1}^N \big \| \sum_{j=1}^N a_{jk}{\bf{e}}_j \big \|_{l^p} = N^{\frac{1}{p} - \frac{1}{2}},
    \end{equation}
 and, therefore, that \eqref{i1} fails for $1 \leq p < 2$.  (Cf. Remark \ref{R2}.i, and also ~\cite{Lindenstrauss:1968}, Corollary 1 on p. 289.)  In particular, there can be no integral representation of  $\langle {\bf{x}}, {\bf{y}}\rangle$, ${\bf{x}} \in l^p$, \  ${\bf{y}} \in l^q$, with uniformly bounded integrands  for  \  $1 \leq p < 2 < q \leq \infty$, \ $\frac{1}{p} + \frac{1}{q} = 1$.
 
To derive a Parseval-like formula in the $l^p-l^q$ setting, we modify the scheme used to prove Theorem \ref{T2}.  We let  $A$ be an infinite set, and define
\begin{equation}
L^{\infty}_{(p)}(\Omega_A, \mathbb{P}_A) = \big \{f \in L^{\infty}(\Omega_A, \mathbb{P}_A): \hat{f} \in l^p(W_A) \big \}, \ \ 1 \leq p \leq 2,
\end{equation}
and 
\begin{equation}
M_{(p)}(\Omega_A) = \big \{\mu \in M(\Omega_A): \hat{\mu} \in l^p(W_A) \big \}, \ \ 2 < p \leq \infty, 
\end{equation}\\
with their respective norms
\begin{equation}
\|f\|_{L^{\infty}_{(p)}} = \max\{\|f\|_{L^{\infty}}, \|\hat{f}\|_p \}, \ \ f \in L^{\infty}_{(p)}(\Omega_A, \mathbb{P}_A), \ \ 1 \leq p \leq 2, 
\end{equation}\\
and
\begin{equation}
\|\mu\|_{M_{(p)}} = \max\{\|\mu\|_M, \|\hat{\mu}\|_p \}, \ \ \mu \in M_{(p)}(\Omega_A), \ \ 2 < p  \leq \infty.
\end{equation}\\
The extremal instances are:\\
\begin{equation}
\begin{split}
L^{\infty}_{(1)}(\Omega_A, \mathbb{P}_A) &= \mathbb{A}(\Omega_A) := \big \{\rm{absolutely \  convergent \ Walsh \ series \ on} \ \Omega_A \big \};\\\\\
L^{\infty}_{(2)}(\Omega_A, \mathbb{P}_A) &= L^{\infty}(\Omega_A, \mathbb{P}_A); \\\\
 M_{(\infty)}(\Omega_A) &= M(\Omega_A).\\\
\end{split}
\end{equation}
Note that if $f \in L^{\infty}_{(p)}(\Omega_A, \mathbb{P}_A)$, \  $\mu \in M_{(q)}(\Omega_A)$, \ $1 \leq p \leq 2 \leq q  \leq \infty, \ \frac{1}{p} + \frac{1}{q} = 1$,  then we have 
\begin{equation}
f \convolution \mu \in  \mathbb{A}(\Omega_A),
\end{equation}\\
with the  usual Parseval formula 
\begin{equation} \label{Parpq}
  \big(f \convolution \mu \big )(\boldsymbol{\omega}_0) =  \int_{\Omega_{A}} fd\mu = \sum_{w \in W_A} \hat{f}(w)\hat{\mu}(w),
  \end{equation}\\
where  $\boldsymbol{\omega}_0$ denotes the identity element in $\Omega_A$. 

We proceed to construct norm-bounded maps
\begin{equation}
\Phi^{(p)}: B_{l^p(A)} \  \rightarrow \ L^{\infty}_{(p)}(\Omega_A, \mathbb{P}_A), \ \ \ 1 \leq p \leq 2,
\end{equation}
and
\begin{equation}
\Phi^{(p)}: B_{l^p(A)} \ \rightarrow \ M_{(p)}(\Omega_A), \ \ \ 2 < p \leq \infty,
\end{equation}
such that
\begin{equation} \label{rep2}
\begin{split}
\sum_{\alpha \in A} {\bf{x}}(\alpha) \overline{{\bf{y}}(\alpha)} &= \int_{\Omega_A} \Phi^{(p)}({\bf{x}})  d\Phi^{(q)}(\overline{{\bf{y}}})\\\ 
& = \sum_{w \in W_A} \widehat{\Phi^{(p)}({\bf{x}})}(w)  \widehat{\Phi^{(q)}(\overline{{\bf{y}}})}(w),\ \ \  \ {\bf{x}} \in B_{l^p(A)}, \ {\bf{y}} \in B_{l^q(A)},
\end{split}
\end{equation} 
for all $1 \leq p \leq 2 \leq q  \leq \infty, \ \frac{1}{p} + \frac{1}{q} = 1$.  For $p \in [1, \infty]$, the maps  $\Phi^{(p)}$ turn out to be weakly continuous (weak topology on the domain and weak* topology on the range), and   $\widehat{ \Phi^{(p)}}$ will be  $\big (l^p(A) \rightarrow l^p(W_A) \big)$-continuous (norm topologies on domain and range).  For \  $p = q = 2$,  we let  $d\Phi^{(2)} := \Phi^{(2)} d\mathbb{P}_A$, and recover Theorem \ref{T2}. 

We consider first  the case $1 \leq p \leq 2.$    Following the setup in \S5, given ${\bf{x}} \in B_{l^p_{\mathbb{R}}(A)}$, we define ${\bf{x}}^{(1)} \in l^p_{\mathbb{R}}(A_1)$ by
\begin{equation} \label{g1p}
\textbf{x}^{(1)}(\tau_0\alpha) = \textbf{x}(\alpha), \ \ \alpha \in A.
\end{equation}
For $p=1$, define
\begin{equation} 
\Phi^{(1)}(\textbf{x}) = \sum_{\alpha \in A_1} {\bf{x}}^{(1)}(\alpha)r_{\alpha}, \ \ \ {\bf{x}} \in B_{l^1_{\mathbb{R}}(A)}.
\end{equation}\\
Let $1 < p \leq 2$, and note that if ${\bf{x}} \in B_{l^p_{\mathbb{R}}(A)}$,   then 
\begin{equation} \label{e40p}
\|Q_A(\rm{\bf{x}})\|_{L^{\infty}}   \leq e^{\frac{1}{2}\|{\bf{x}}\|_2^2} \leq e^{\frac{1}{2}\|{\bf{x}}\|_p^2}  \leq \sqrt{e}, 
\end{equation}\\\
and
 \begin{equation} \label{e42p}
  \|\widehat{Q_A(\rm{\bf{x}})}|_{R_F^c} \|_p   \leq (\sinh(1) - 1)^{\frac{1}{p}} = \delta^{\frac{2}{p}} < 1.
 \end{equation}\\\
(Cf. \eqref{e40}, \eqref{e42}  and \eqref{e42e}.)  Following \eqref{g1p}, we  define $\textbf{x}^{(j)} \in B_{l^p_{\mathbb{R}}(A_j)}$,  $ j = 2, \dots,$   by   
\begin{equation} \label{g2p}
\textbf{x}^{(j)}(\tau_{j-1}w) \  =  \  \delta^{\frac{2(j-2)}{p}}\bigg({Q_{A_{j-1}}\big(\textbf{x}^{(j-1)}/\delta^{\frac{2(j-2)}{p}}\big)}\bigg)^{\wedge}(w), \ \ \ w \in C_{j-1}.
\end{equation}
From \eqref{g1p}, \eqref{g2p}, \eqref{e42p}, we obtain\\ 
\begin{equation} \label{g4p}
\|\textbf{x}^{(j)}\|_p \leq \delta^{\frac{2(j-1)}{p}}, \ \ \ j \geq 1.
\end{equation}\\
By \eqref{e39} and  \eqref{e41} in Lemma \ref{L2}, and \eqref{e40p} and \eqref{g4p} above, we have for $j \geq 1, $
\begin{align}
\text{spect}\bigg(Q_{A_j}\big( \textbf{x}^{(j)}/\delta^{\frac{2(j-1)}{p}}\big)\bigg)  & \subset W_{A_j} \ ; \label{g6p}\\\
\big \|Q_{A_j}\big(\textbf{x}^{(j)}/\delta^{\frac{2(j-1)}{p}}\big) \big \|_{L^{\infty}} & \leq \sqrt{e} \ ; \label{g3p}\\\
\delta^{\frac{2(j-1)}{p}}  \bigg(Q_{A_j}\big( \textbf{x}^{(j)}/\delta^{\frac{2(j-1)}{p}}\big)\bigg)^{\wedge} (r_{\alpha}) & = \textbf{x}^{(j)}(\alpha), \ \ \ \alpha \in A_j \  .
\end{align}
Now define
\begin{equation} \label{g5p}
\Phi^{(p)}(\textbf{x}) = \sum_{j=1}^{\infty}\big(i\delta^{\frac{2}{p}}\big)^{j-1} \ Q_{A_j}\big(\textbf{x}^{(j)}/\delta^{\frac{2(j-1)}{p}}\big), \ \ \ {\bf{x}} \in B_{l^p_{\mathbb{R}}(A)}.
\end{equation}\\
(Cf. \eqref{g5}.)  The series in \eqref{g5p} is absolutely convergent in the $L^{\infty}$-norm, and  by  \eqref{g3p},\\
\begin{equation} \label{comp1}
\|\Phi^{(p)}(\textbf{x})\|_{L^{\infty}} \leq \ \frac{\sqrt{e}}{1 - \delta^{\frac{2}{p}}} \ .
\end{equation}\\
By the spectral analysis in \eqref{e45}, and by adapting the estimates in \eqref{e44} and \eqref{e44a} to the case ${\bf{x}} \in B_{l_{\mathbb{R}}^p(A)}$, we deduce
\begin{equation}
\sum_{w \in W_{A_j}} \big| \big(Q_{A_j}({\bf{x}}^{(j)})/\delta^{\frac{2(j-1)}{p}}\big)^{\wedge}(w) \big|^p \leq 1 + \delta^2,
\end{equation}
and therefore,
\begin{equation} \label{comp2}
\| \widehat{\Phi^{(p)}(\textbf{x})} \|_p \leq \bigg(\frac{1 + \delta^2}{1 - \delta^2}\bigg)^{\frac{1}{p}} .
\end{equation}
From \eqref{comp1}, \eqref{comp2}, and the definition of $\delta$, we obtain
\begin{equation}
\|\Phi^{(p)}(\textbf{x})\|_{L^{\infty}_{(p)}} \  \leq \  \max \bigg\{\frac{\sqrt{e}}{1 - \delta^{\frac{2}{p}}},\bigg (\frac{1 + \delta^2}{1 - \delta^2}\bigg)^{\frac{1}{p}} \bigg \} \ = \  \frac{\sqrt{e}}{1 - \delta^{\frac{2}{p}}}.
\end{equation}\\

Next let  $2 < p  \leq \infty$, and  ${\bf{x}} \in B_{l_{\mathbb{R}}^p(A)}$.  Recalling Lemma \ref{L2a},  for $F \subset A$ we have \\ 
\begin{equation}
\begin{split}
P_F({\bf{x}}) &= \mathfrak{R}_F\big(\frac{{\bf{x}}}{2}\big) - \mathfrak{R}_F\big(\frac{-{\bf{x}}}{2}\big) \\\\
&\sim \   \sum_{k=0}^{\infty}\frac{1}{4^k} \bigg(\sum_{{w \in W_{F,2k+1}} \atop{w = r_{\alpha_1} \cdots r_{\alpha_{2k+1}}}}\textbf{x}(\alpha_1) \cdots  \textbf{x}(\alpha_{2k+1}) r_{\alpha_1} \cdots r_{\alpha_{2k+1}}\bigg),
\end{split}
\end{equation}\\
where $\mathfrak{R}_F({\bf{x}})$ is the Riesz product in \S \ref{SS4}.i.  Then, 
\begin{equation} \label{e40a}
\|P_F({\bf{x}})\|_{M(\Omega_A)}   \leq 2 \ ;
\end{equation}\
\begin{equation} \label{e41a}
 \widehat{P_F({\bf{x}})}(r_{\alpha})   =   {\bf{x}}(\alpha), \ \ \alpha \in F \ ;
 \end{equation}\
 \begin{equation} \label{e42a}
  \|\widehat{P_F({\bf{x}})}|_{R_F^c} \|_p   \leq \bigg(2\sinh\big(\frac{1}{2}\big) - 1\bigg)^{\frac{1}{p}} < \delta^{\frac{2}{p}} < 1, \  \ \ \ 2 < p < \infty;
 \end{equation}\
 \begin{equation} \label{e43a}
 \|\widehat{P_F({\bf{x}})}|_{R_F^c} \|_{\infty} \leq \frac{1}{4}.
 \end{equation} 
 We define ${\bf{x}}^{(1)} \in B_{l_{\mathbb{R}}^p(A_1)}$ as in \eqref{g1p}.  For  $ j \geq 2$,  and  $p \in (2, \infty)$,  let 
 \begin{equation} \label{g2q}
\textbf{x}^{(j)}(\tau_{j-1}w) =  \delta^{\frac{2(j-2)}{p}}\bigg({P_{A_{j-1}}\big( \textbf{x}^{(j-1)}}/\delta^{\frac{2(j-2)}{p}}\big )\bigg)^{\wedge}(w), \ \ \ w \in C_{j-1}, 
\end{equation}\
and for $p = \infty$, \  let
 \begin{equation} \label{g2qq}
\textbf{x}^{(j)}(\tau_{j-1}w) =  \frac{1}{4^{j-2}}\bigg({P_{A_{j-1}}\big(4^{j-2} \textbf{x}^{(j-1)}}\big )\bigg)^{\wedge}(w), \ \ \ w \in C_{j-1}.
\end{equation}
By iterating \eqref{e42a} and \eqref{e43a}, we obtain for $p \in (2, \infty)$,
\begin{equation} \label{g4pa}
\|\textbf{x}^{(j)}\|_p \leq    \bigg(2\sinh\big(\frac{1}{2}\big) - 1\bigg)^{\frac{j-1}{p}}  \leq  \delta^{\frac{2(j-1)}{p}}, \ \ \  \ \ j \geq 1,
\end{equation}\\
and for $p = \infty$,
\begin{equation} \label{g4paa}
\|\textbf{x}^{(j)}\|_{\infty} \leq \bigg(\frac{1}{4}\bigg)^{j-1}, \ \ \  \ j \geq 1. \\\\
\end{equation}  
We define 
 \begin{equation} \label{g5pa}
 \Phi^{(p)}(\textbf{x}) = \sum_{j=1}^{\infty}\big(i\delta^{\frac{2}{p}}\big)^{j-1}P_{A_j}\big( \textbf{x}^{(j)}/\delta^{\frac{2(j-1)}{p}}\big), \ \  \ 2 < p <  \infty,  \ \ {\bf{x}} \in B_{l^p_{\mathbb{R}}(A)},
 \end{equation}
and  
\begin{equation} \label{g5paa}
 \Phi^{(\infty)}(\textbf{x}) = \sum_{j=1}^{\infty}\big(1/4\big)^{j-1}P_{A_j}\big( 4^{j-1}\textbf{x}^{(j)}\big), \ \ \ {\bf{x}} \in B_{l^{\infty}_{\mathbb{R}}(A)}.
 \end{equation}
We thus obtain $\Phi^{(p)}(\textbf{x}) \in M_{(p)}(\Omega_A)$ with  
\begin{equation} \label{g6pa}
\|\Phi^{(p)}(\textbf{x})\|_{M_{(p)}} \leq  \ 2\bigg (\frac{1 + \delta^2}{1 - \delta^2}\bigg)^{\frac{1}{p}}, \ \ \ 2 < p < \infty,
\end{equation}
 and
 \begin{equation} \label{g6paa}
\|\Phi^{(\infty)}(\textbf{x})\|_M \leq  \ \frac{8}{3} \ .
\end{equation}
 
 For  ${\bf{x}} \in B_{l^p(A)}$,  $1 \leq p \leq \infty$,   write $${\bf{x}} = {\bf{u}} + i {\bf{v}}, \ \  {\bf{u}} \in B_{l^p_{\mathbb{R}}(A)}, \  {\bf{v}} \in B_{l^p_{\mathbb{R}}(A)},$$ and then define\\
\begin{equation} \label{comp}
\Phi^{(p)}(\textbf{x}) := \Phi^{(p)}(\textbf{u}) + i \Phi^{(p)}(\textbf{v}).
\end{equation}\\
 
 To prove  \eqref{rep2} in the case $p = 1$ and $q = \infty$, observe that if ${\bf{x}} \in B_{l^1(A)}$, then $\text{spect}\big(\Phi^{(1)}(\textbf{x})\big) \subset A_1$, 
 \begin{equation}
 \big(\Phi^{(1)}(\textbf{x})\big)^{\wedge}(r_{\alpha}) = {\bf{x}}^{(1)}(\alpha), \ \ \ \alpha \in A_1,
 \end{equation}
 and that for ${\bf{y}} \in B_{l^{\infty}(A)}$,   
 \begin{equation}
 \big(\Phi^{(\infty)}(\textbf{y})\big)^{\wedge}(r_{\alpha}) = {\bf{y}}^{(1)}(\alpha), \ \ \ \alpha \in A_1.
 \end{equation}
 Then,  \eqref{rep2} in this instance is a simple consequence of the usual Parseval formula.   
 
 The proof of \eqref{rep2} in the case $1 < p \leq q < \infty$, which is similar to that of  
\eqref{e20} in \S \ref{SS5}, uses an alternating series argument based on iterating the Parseval formula\\
 \begin{equation}
 \begin{split}
 \big (Q_F({\bf{x}}) \convolution& P_F({\bf{y}})\big )(\boldsymbol{\omega}_0) 
  =  \sum_{k=0}^{\infty}\frac{1}{4^k} \bigg(\sum_{{w \in W_{F,2k+1}} \atop{w = r_{\alpha_1} \cdots r_{\alpha_{2k+1}}}}  \textbf{x}(\alpha_1) \cdots  \textbf{y}(\alpha_{2k+1}) \textbf{y}(\alpha_1) \cdots  \textbf{x}(\alpha_{2k+1}) \bigg ),\\\  &F \subset A, \ \ {\bf{x}} \in  l^p(F), \ \ {\bf{y}} \in  l^q(F), \ \ \frac{1}{p} + \frac{1}{q} = 1, \ \ 1 < p \leq  2 \leq q < \infty.
 \end{split}
 \end{equation}
  
  The proof that\begin{equation} \label{e43pp}
   \Phi^{(p)}(-{\bf{x}}) = -  \Phi^{(p)}({\bf{x}}), \ \ \ \ p \in [1, \infty], \ \  {\bf{x}} \in B_{l^p(A)}, 
  \end{equation}
is essentially the same as the proof of \eqref{e43}.

For $1< p < \infty$, the weak continuity of  $\Phi^{(p)}$ and the $\big (l^p(A) \rightarrow l^p(W_A) \big)$-continuity of  $\widehat{\Phi^{(p)}}$  follow from extensions of Lemmas \ref{L3W} and \ref{L3}.   The weak continuity of $\Phi^{(1)}$ and the $\big (l^1(A) \rightarrow l^1(W_A) \big)$-continuity of $\widehat{\Phi^{(1)}}$ are obvious;  the weak continuity of $\Phi^{(\infty)}$ and the $\big (l^{\infty}(A) \rightarrow l^{\infty}(W_A) \big)$-continuity of $\widehat{\Phi^{(\infty)}}$ follow from the spectral analysis of $P_A$.

 We summarize:

\begin{theorem} \label{T2pq}
  Let $A$ be an infinite set.  There exist one-one maps
  \begin{equation} \label{norgp}
\Phi^{(p)}: \ B_{l^p(A)}  \rightarrow \left \{
\begin{array}{lcc}
 L^{\infty}_{(p)}(\Omega_{A}, \mathbb{P}_{A}), & \quad  1 \leq p \leq 2,  \\\\
 M_{(p)}(\Omega_{A}), & \quad   2  < p \leq \infty,
\end{array} \right.
\end{equation}
 such that   
 \begin{equation} \label{e31p}
  \begin{split}
  \|\Phi^{(p)}(\textbf{\emph{x}})\|_{L^{\infty}_{(p)}} &\leq \ \frac{2\sqrt{e}}{1 - \delta^{\frac{2}{p}}},  \ \ \ \textbf{\emph{x}} \in B_{l^p(A)}, \ \ 1 \leq p \leq 2, \\\\
  \|\Phi^{(p)}(\textbf{\emph{x}})\|_{M_{(p)}} &\leq \ 4\bigg (\frac{1 + \delta^2}{1 - \delta^2}\bigg)^{\frac{1}{p}}, \ \ \ \textbf{\emph{x}} \in B_{l^p(A)}, \ \ 2 < p \leq \infty,\\\\
  \end{split}
  \end{equation}\\
  \begin{equation} \label{e43p}
   \Phi^{(p)}(-{\bf{x}}) = -  \Phi^{(p)}({\bf{x}}), \ \ \  \ {\bf{x}} \in B_{l^p(A)}, \ \ 1 \leq p \leq \infty,
  \end{equation}\\
  and
\begin{equation} \label{e20pq}
  \begin{split}
  \sum_{\alpha \in A} {\bf{x}}(\alpha)\overline{{\bf{y}}(\alpha)} &=  \int_{\Omega_{A}} \Phi^{(p)}({\bf{x}})d\Phi^{(q)}(\overline{{\bf{y}}}),\\\
  & {\bf{x}} \in  B_{l^p(A)}, \ \ {\bf{y}} \in  B_{l^q(A)}, \ \ \frac{1}{p} + \frac{1}{q} = 1, \ \ 1 \leq p \leq  2 \leq q \leq \infty, \\\\
  \end{split}
  \end{equation}\\\ 
where $d\Phi^{(2)} := \Phi^{(2)} d\mathbb{P}_A.$  \  Moreover,  for every $p \in [1, \infty],$ \ $\Phi^{(p)}$ is weakly continuous and $ \widehat{\Phi^{(p)}}$ is $\big(l^p(W) \rightarrow l^p(W_A) \big)$-continuous.\\

    \end{theorem}


\section{\bf{Grothendieck-like theorems in dimensions $> 2$?}} \label{s6}
The general focus so far has been on scalar-valued functions of two variables and their representations by functions of one variable.  Next we consider functions of $n$ variables, $n \geq 2$, and examine, analogously, the feasibility of representing them by functions that depend on $k$ variables, $0 < k < n$.  

\subsection{A multidimensional version of Question \ref{Q2}} \ \ We begin with a natural extension of Question \ref{Q2} in the case  $n \geq 2$ and $k = 1$.  Let  $X_1, \ldots, X_n$ be sets, and denote
 \begin{equation*}
  \emph{\bf{X}}^{[n]} := X_1 \times \cdots \times X_n,
 \end{equation*}
 where $[n] = \{1,\ldots,n\}$.
\begin{question} \label{Q5}
 Let   $f \in l^{\infty}( \emph{\bf{X}}^{[n]})$.   Can we find a probability space $(\Omega,  \mu)$, and functions $g_{\omega}^{(i)}$  indexed by $\omega \in \Omega$ and defined on   $X_i$, $i \in [n]$,   such that for  all $\emph{\bf{x}}  = (x_1, \ldots, x_n) \in \emph{\bf{X}}^{[n]}$ and $i \in [n]$,
\begin{equation*}
\omega \mapsto  g_{\omega}^{(i)}(x_i) \ \ \  \text{and} \ \ \   \omega \mapsto \|g_{\omega}^{(i)}\|_{\infty} 
\end{equation*} \\\
are $\mu$-measurable on $\Omega,$ and 
\begin{equation} \label{g20}\\\
f(\emph{\bf{x}}) = \int_{\Omega} g_{\omega}^{(1)} (x_1)  \cdots  g_{\omega}^{(n)}(x_n) \mu(d \omega)
\end{equation}
under the constraint
\begin{equation} \label{g21}
\sup_{\emph{\bf{x}} \in \emph{\bf{X}}^{[n]}} \|g_{\omega}^{(1)}(x_1) \|_{L^{\infty}(\mu)} \cdots  \|g_{\omega}^{(n)}(x_n) \|_{L^{\infty}(\mu)}  < \infty?
\end{equation}
(As before, we refer to  $\Omega$ as an indexing space, to $\mu$ as an indexing measure, and to $g_{\omega}^{(i)}$  as representing functions.)
\end{question}
Matters are formalized as usual.    Given $f \in l^{\infty}( \emph{\bf{X}}^{[n]})$, let  $\| f\|_{\tilde{\mathcal{V}}_n}$ be the infimum of the left side of \eqref{g21} taken over all probability spaces $(\Omega,  \mu)$, and all families of functions on $X_i$ indexed by  $(\Omega,  \mu)$ that represent $f$ by \eqref{g20}.  We  obtain the Banach algebra
\begin{equation*}
\tilde{\mathcal{V}}_n( \emph{\bf{X}}^{[n]}) = \{f\in l^{\infty}( \emph{\bf{X}}^{[n]}):\| f\|_{\tilde{\mathcal{V}}_n} < \infty \},
\end{equation*}
wherein algebra multiplication is given by point-wise multiplication on ${\bf{X}}^{[n]}$.  It follows from the case $n=2$ that if at least two of the $X_i$ are infinite, then $\tilde{\mathcal{V}}_n( \emph{\bf{X}}^{[n]}) \subsetneq l^{\infty}( \emph{\bf{X}}^{[n]}$.

Proposition \ref{P1} has a straightforward extension.  Take the compact Abelian group
\begin{equation*}
{\bf{\Omega}}_{ \emph{\bf{X}}^{[n]}} := \Omega_{X_1}\times \cdots \times \Omega_{X_n},
\end{equation*}\\\  
and the spectral set ($n$-fold Cartesian product of Rademacher systems)\\\
\begin{equation*}
\begin{split}
{\emph{\bf{R}}}_{ \emph{\bf{X}}^{[n]}} & = R_{X_1} \times \cdots \times R_{X_n} \\\
& = \big \{r_{x_1}\otimes \cdots \otimes r_{x_n}: (x_1, \ldots, x_n) \in \emph{\bf{X}}^{[n]} \big \} \subset {\widehat{\bf{\Omega}}}_{ \emph{\bf{X}}^{[n]}}.\\\
\end{split}
\end{equation*}
 Then, by identifying\\\ 
\begin{equation}\label{dua1}
 \begin{split}
  B({\emph{\bf{R}}}_{ \emph{\bf{X}}^{[n]}})  &:= \widehat{M}({\bf{\Omega}}_{ \emph{\bf{X}}^{[n]}} )/\{\hat{\lambda} \in \widehat{M}({\bf{\Omega}}_{ \emph{\bf{X}}^{[n]}}): \hat{\lambda} = 0 \ \text{on} \ {\emph{\bf{R}}}_{ \emph{\bf{X}}^{[n]}}\}\\\\
& = \{\text{restrictions of Walsh transforms to} \ {\emph{\bf{R}}}_{ \emph{\bf{X}}^{[n]}} \}\\\ 
\end{split}
\end{equation} 
as the dual space of\\\  
\begin{equation}\label{dua2}
C_{{\emph{\bf{R}}}_{ \emph{\bf{X}}^{[n]}}}({\bf{\Omega}}_{ \emph{\bf{X}}^{[n]}} ) := \{\text{continuous functions on} \ {\bf{\Omega}} _{ \emph{\bf{X}}^{[n]}} \  \text{with spectrum in} \ {\emph{\bf{R}}}_{ \emph{\bf{X}}^{[n]}} \},
\end{equation}\\\
 we obtain (omitting the proof)\\\  
\begin{proposition} \label{P4}
  \begin{equation*}
    \tilde{\mathcal{V}}_n( \emph{\bf{X}}^{[n]}) = B({\emph{\bf{R}}}_{ \emph{\bf{X}}^{[n]}}).
   \end{equation*}\\     
    \end{proposition}

  \subsection{A multidimensional version of Question \ref{Q3}} \ \ Next we consider integral representations of  $f \in l^{\infty}( \emph{\bf{X}}^{[n]})$ with constraints weaker than \eqref{g21}, and ask whether the feasibility of such representations implies the feasibility of integral representations under the stronger constraint in \eqref{g21}.  That is, are there Grothendieck-like theorems in dimensions  $> 2$?  (Cf. discussion preceding Theorem \ref{T1}, and also Remark \ref{R2}.i.)  
  
  Attempting first the obvious, we extend Question \ref{Q3} by replacing in it the role of Cauchy-Schwarz with that of a multi-linear H\"older inequality.
  
 \begin{question} \label{Q6}
 Let  $f \in l^{\infty}( \emph{\bf{X}}^{[n]})$, and let $\emph{\bf{p}} = (p_1, \dots, p_n)$ be a conjugate vector,  i.e.,   $\emph{\bf{p}} = (p_1, \dots, p_n) \in [1,\infty]^n$ such that 
 \begin{equation*}
  \frac{1}{p_1} + \cdots + \frac{1}{p_n} = 1.
  \end{equation*}\\\
 Can we find a probability space $(\Omega,  \mu)$, and functions $g_{\omega}^{(i)}$  indexed by $\omega \in \Omega$ and defined on   $X_i$, $i \in [n]$,  such that for $\emph{\bf{x}}  = (x_1, \ldots, x_n) \in \emph{\bf{X}}^{[n]}$ and $i \in [n]$,
\begin{equation*}
\omega \mapsto  g_{\omega}^{(i)}(x_i)  
\end{equation*}\\\ 
determine elements in $L^{p_i}(\Omega, \mu)$, and\\\  
\begin{equation}\label{g22}\\\
f(\emph{\bf{x}}) = \int_{\Omega} g_{\omega}^{(1)}(x_i) \cdots  g_{\omega}^{(n)}(x_n) \mu(d \omega)
\end{equation}\\\
subject to\\\
\begin{equation} \label{g23}
\sup_{\emph{\bf{x}} \in \emph{\bf{X}}^{[n]}} \|g_{\omega}^{(1)}(x_1) \|_{L^{p_1}(\mu)} \cdots  \|g_{\omega}^{(n)}(x_n) \|_{L^{p_n}(\mu)}  < \infty?
\end{equation}\\\
\end{question}
\noindent
Given $f \in l^{\infty}( \emph{\bf{X}}^{[n]})$, define  $\|f\|_{\mathcal{G}_{\emph{\bf{p}}}}$ as the infimum of the left side of \eqref{g23} over all integral representations of $f$ by \eqref{g22}.  We  obtain the Banach algebra\\\
\begin{equation}
\mathcal{G}_{\emph{\bf{p}}}( \emph{\bf{X}}^{[n]}) = \{f\in l^{\infty}( \emph{\bf{X}}^{[n]}):\| f\|_{\mathcal{G}_{\emph{\bf{p}}}} < \infty \},
\end{equation}\\\
 and note the inclusion \\\
\begin{equation} \label{g24}
\tilde{\mathcal{V}}_n({\emph{\bf{X}}}^{[n]})\subset \mathcal{G}_{\emph{\bf{p}}}( \emph{\bf{X}}^{[n]}), \ \ \ n \geq 2.
\end{equation}\\\

If  $n = 2$ and $\emph{\bf{p}} = (2,2)$,  then also the reverse inclusion holds; this is the gist of Theorem \ref{T1}.   Otherwise, if $X_1$ and $X_2$ are infinite and  $\emph{\bf{p}} \neq (2,2)$, then \eqref{g24} is a proper inclusion;  see Remark \ref{R2}.i, and also the start of \S \ref{Snonh}. 

 For $n > 2$ and infinite $X_i, \ i \in [n]$,  the inclusion in  \eqref{g24} is \emph{always} proper.  We demonstrate this in the case $n = 3$, which suffices.   We use a trilinear analog of the Fourier matrix in \eqref{e17},
\begin{equation}\label{3dfourier}
e^{2\pi i \frac{(j+k)l}{N} }, \  \ \ (j,k,l) \in [N]^3, \ \  N \in \mathbb{N},
\end{equation}
and a corresponding lemma (phrased somewhat differently in \cite[Exer. IV.2.ii]{blei:2001}). 
\begin{lemma} \label{3dgauss}
For all integers $N > 1$, \ $f \in B_{l^{\infty}([N])}$, \ $g \in B_{l^{\infty}([N])}$, and \ $h \in B_{l^{\infty}([N])}$, 
\begin{equation} \label{g25}
\big | \frac{1}{N^2} \sum_{j, k, l =1}^N e^{2\pi i \frac{(j+k)l}{N} } f(j)g(k)h(l)\big | \leq 1.
\end{equation}\\\
\end{lemma}
\begin{proof}  We take $[N]$ to be the compact Abelian group of $N^{\text{th}}$-roots of unity,
\begin{equation}
j \leftrightarrow e^{\frac{2\pi ij}{N}}, \ \ \ j \in [N],
\end{equation}
with its usual algebraic structure, and with the uniform probability measure as its Haar measure.  For $f \in l^{\infty}([N])$ and  $g \in l^{\infty}([N])$,
\begin{align*}
\hat{f}(l) = &  \frac{1}{N} \sum_{j=1}^N e^{ \frac{2\pi ijl}{N} } f(j), \\\
\hat{g}(l) = &  \frac{1}{N} \sum_{k=1}^N e^{ \frac{2\pi ikl}{N} } g(k), \ \ \ \ \ l \in [N],
\end{align*}
and therefore, for $h \in l^{\infty}([N])$,
\begin{align*}
\bigg | \frac{1}{N^2} \sum_{j, k, l =1}^N e^{2\pi i \frac{(j+k)l}{N} } f(j)g(k)h(l)\bigg | &= \bigg | \sum_{l=1}^N \bigg ( \frac{1}{N^2} \sum_{j, k = 1}^N e^{2\pi i \frac{(j+k)l}{N} } f(j)g(k)\bigg) h(l) \bigg | \\\\
& = \big |\sum_{l=1}^N \hat{f}(l) \hat{g}(l) h(l) \big | \leq \|\hat{f}\|_2 \|\hat{g}\|_2 \|h\|_{\infty} \leq \|f\|_{\infty}\|g\|_{\infty}\|h\|_{\infty}.
\end{align*}
\end{proof} \ \\\

\noindent
Let $\emph{\bf{p}} = (p_1, p_2, p_3)  \in [1,\infty]^3$ be a conjugate vector, and suppose $p_3 < \infty$.  For an arbitrary integer $N > 1$, take  $X_i = [N]$,  $ i = 1, 2, 3$, and  $\Omega = [N]$ with the uniform probability measure $\mu$  on it.  Take representing functions\\\
\begin{equation*}
g_{\omega}^{(1)}(j) = e^{\frac{-2\pi  i j \omega}{N}}, \ \  \ g_{\omega}^{(2)}(k) = e^{\frac{-2\pi  i k \omega}{N}}, \ \ \ \omega \in \Omega,  \ j \in X_1, \  k \in  X_2,
\end{equation*}\\\
 and
\begin{equation*}
 g_{\omega}^{(3)}(l) =  \left \{
\begin{array}{cccc}
0 & \quad  \text{if} & \quad  \omega  \not= l  & \quad \\
N^{\frac{1}{p_3}} & \quad   \text{ if} & \quad  \omega = l, & \quad  \ \ \ \omega \in \Omega, \  l \in X_3.
\end{array} \right. 
\end{equation*}\\\ 
Then,\\\
\begin{equation*}\\\
f(j,k,l) := \int_{\Omega}g_{\omega}^{(1)}(j) g_{\omega}^{(2)}(k)g_{\omega}^{(3)}(l) \mu(d \omega) = N^{\frac{1}{p_3} - 1}e^{-2\pi i \frac{(j+k)l}{N} }, \ \ \ (j,k,l) \in \emph{\bf{X}}^{[3]},
\end{equation*}\\\
and\\\
\begin{equation*}
\sup_{(j,k,l) \in \emph{\bf{X}}^{[3]}} \|g_{\omega}^{(1)}(j) \|_{L^{p_1}(\mu)} \|g_{\omega}^{(2)}(k) \|_{L^{p_2}(\mu)}  \|g_{\omega}^{(3)}(l) \|_{L^{p_3}(\mu)} = 1.
\end{equation*}\\\
 By duality and \eqref{g25}, we obtain\\\ 
\begin{equation*}
\|f\|_{\tilde{\mathcal{V}}_3} \geq  \big |\frac{1}{N^2}\sum_{j,k,l = 1}^N  e^{2\pi i \frac{(j+k)l}{N} }f(j,k,l) \big | =  N^{ \frac{1}{p_3}} \ \underset{N \to \infty}{\longrightarrow} \infty,
\end{equation*}\\\
which implies 
\begin{proposition}
For all conjugate vectors \  $\emph{\bf{p}} = (p_1, p_2, p_3)  \in [1,\infty]^3$, and infinite sets $X_1$, $X_2$, and $X_3$,\\\ 
\begin{equation*}
\tilde{\mathcal{V}}_3( \emph{\bf{X}}^{[3]}) \subsetneq \mathcal{G}_{(p_1,p_2,p_3)}( \emph{\bf{X}}^{[3]}) \subsetneq l^{\infty}( \emph{\bf{X}}^{[3]}) \ \ \ (\text{proper inclusions}).
\end{equation*}\\\
\end{proposition}

A question remains:  what are multidimensional analogues of the "two-dimensional" inclusion  \\\
\begin{equation*}
\mathcal{G}_{(2,2)}(\emph{\bf{X}}^{[2]}) \subset \tilde{\mathcal{V}}_2(\emph{\bf{X}}^{[2]})?
\end{equation*}\\\

\section{\bf{Fractional Cartesian products and multilinear functionals on a Hilbert space}} \label{s7}
We return to a Hilbert space setting as staging grounds for Grothendieck-like inequalities.
\vskip0.4cm
\subsection{Projective boundedness and projective continuity} \ \  To start, we note that the (two-dimensional) Grothendieck theorem is an assertion about integral representations of bilinear functionals on a Hilbert space:  if $\eta$ is a bounded bilinear functional on a Hilbert space $H$, then there exist bounded mappings $\phi_1$ and $\phi_2$ from $B_H$ into $L^{\infty}(\Omega, \mu)$, for some probability space  $(\Omega, \mu)$, such that
\begin{equation*}
\eta(\emph{\bf{x}}, \emph{\bf{y}}) = \int_{\Omega} \phi_1(\emph{\bf{x}}) \phi_2(\emph{\bf{y}}) d\mu, \ \ \ (\emph{\bf{x}},\emph{\bf{y}}) \in (B_H)^2 .
\end{equation*}
(Cf. \eqref{i6} in Proposition \ref{P0}.)   Theorem \ref{T2} provides an upgrade, asserting that $\phi_1$ and $\phi_2$ can be designed to be continuous.  
\begin{definition} \label{Q7}
 Let $\eta$ be a bounded $n$-linear functional on a Hilbert space $H$, i.e., 
 \begin{equation*}
 \eta: \underbrace{H \times \cdots \times H}_n  \longrightarrow \mathbb{C}
 \end{equation*} 
 is linear in each coordinate and
\begin{equation*}
\|\eta\|_{\infty} := \sup \{ |\eta(\emph{\bf{x}}_1, \ldots,  \emph{\bf{x}}_n)|:  (\emph{\bf{x}}_1, \dots, \emph{\bf{x}}_n) \in (B_H)^n  \} < \infty.
\end{equation*}
\noindent
{\bf{i.}}   $\eta$ is \emph{projectively bounded} if there exist a probability space $(\Omega, \mu)$ and bounded mappings $\phi_1, \ldots, \phi_n$ from $B_H$ into $L^{\infty}(\Omega, \mu)$, such that
\begin{equation} \label{z15}
\eta(\emph{\bf{x}}_1, \ldots,  \emph{\bf{x}}_n) = \int_{\Omega} \phi_1(\emph{\bf{x}}_1) \cdots  \phi_n(\emph{\bf{x}}_n) d\mu, \ \ \ (\emph{\bf{x}}_1, \dots, \emph{\bf{x}}_n) \in (B_H)^n.
\end{equation}\\\
Equivalently, $\eta$ is projectively bounded if
\begin{equation} \label{z16}
\|\eta\|_{\tilde{\mathcal{V}}_n(B_H \times \cdots \times B_H)} < \infty.
\end{equation}\\\
{\bf{ii.}}  $\eta$ is \emph{projectively continuous} if there exist a probability space $(\Omega, \mu)$ and bounded mappings $\phi_1, \ldots, \phi_n$ from $B_H$ into $L^{\infty}(\Omega, \mu)$ that satisfy \eqref{z15}, and are $(H \rightarrow L^2(\Omega, \mu))$-continuous.  Equivalently,  $\eta$ is projectively continuous if 
\begin{equation} \label{z17}
\|\eta\|_{\tilde{V}_n(B_H \times \cdots \times B_H)} < \infty,
\end{equation}\\\
where the $\|\cdot\|_{\tilde{V}_n({ \emph{\bf{X}}^{[n]}})}$-norm is the $n-$dimensional extension of the $\|\cdot\|_{\tilde{V}_2(X_1 \times X_2)}$-norm in Definition \ref{D1}.
\end{definition}
\vskip0.4cm
\noindent
(The notion of projective boundedness is equivalent, through duality, to the notion in  \cite[Definition 1.1]{Blei:1979fk}.)  We  have \\\
\begin{equation} \label{g26}
\|\eta\|_{\tilde{\mathcal{V}}_n(B_H \times \cdots \times B_H)}  \leq \|\eta\|_{\tilde{V}_n(B_H \times \cdots \times B_H)}.
\end{equation}\\\
Projective continuity easily implies projective boundedness, but I do not know whether the converse holds. (See \S \ref{q10}.)

In the case $n = 1$, every bounded linear functional on a Hilbert space $H$ is projectively continuous; this is an obvious instance of the  "two-dimensional" Theorem \ref{T2}.  But standing alone, this "one-dimensional" instance is not obvious.  Indeed, the assertion that every bounded linear functional on a Hilbert space is projectively bounded is equivalent to the \emph{little Grothendieck Theorem} ~\cite[Theorem 5.1]{pisier2012grothendieck};  see Remark \ref{RLG} below.  The ostensibly stronger assertion, that every such functional is projectively continuous, is a direct consequence of Lemma \ref{L3}. 
\begin{proposition} \label {P8}
If \  ${\bf{y}} \in B_H$, and
\begin{equation}\label{fun1}
\eta_{{\bf{y}}} ({\bf{x}}) = \langle {\bf{x}}, {\bf{y}} \rangle, \ \ {\bf{x}} \in B_H,
\end{equation}
where $\langle {\cdot, \cdot} \rangle$ is the inner product in $H$, then
\begin{equation}
\|\eta_{{\bf{y}}}\|_{\tilde{V}_1(B_H)} \leq 2e^{\frac{1}{2}}.
\end{equation}
\end{proposition}
\begin{proof}
For simplicity (and with no loss of generality) assume $H = l^2_{\mathbb{R}}(A)$.  Define 
\begin{equation}
g_ {{\bf{y}}} = \sum_{\alpha \in A} {\bf{y}} (\alpha) r_{\alpha}. 
\end{equation}
Then, $g_ {{\bf{y}}} \in L^2_{R_A}(\Omega_A, \mathbb{P}_A)$, and 
\begin{equation}
\|{\bf{y}}\|_2 = \|g_ {{\bf{y}}}\|_{L^2} \leq 1.
\end{equation}
By Parseval's formula and the spectral analysis of Riesz products (as per \S \ref{s3}), we have
\begin{equation} \label{replg}
\eta_{{\bf{y}}} ({\bf{x}}) = \sum_{\alpha \in A} {\bf{x}}(\alpha) {\bf{y}}(\alpha) = \int_{\Omega_A} Q_A({\bf{x}})  g_ {{\bf{y}}}   d \mathbb{P}_A, \ \ \ {\bf{x}} \in B_{l^2},
\end{equation}
which, by \eqref{e40} and Lemma \ref{L3},  implies
\begin{equation}
\|\eta_{{\bf{y}}}\|_{\tilde{V}_1(B_H)}\leq e^{\frac{1}{2}}.
\end{equation}\\\
\end{proof}

In the case $n = 2$, Theorem \ref{T2} implies that every bounded \emph{bilinear} functional on a Hilbert space is projectively continuous.   

In the case $n = 3$, there exist bounded trilinear functionals that are \emph{not} projectively bounded \cite{varopoulos1974inequality}, and \emph{a fortiori} not projectively continuous.  (See Remark \ref{R8}.)  

\begin{problem} \label{QP7}
  For an infinite-dimensional Hilbert space $H$, and for $n > 2$, which are the projectively bounded, and which are the projectively continuous $n$-linear functionals on $H$?\\
\end{problem}
\noindent
We may as well take $H$ to be $l^2(A)$, where $A$ is an infinite indexing set.  Then, given a bounded $n$-linear functional \ $\eta$ \ on  $l^2(A)$, we write (first formally)\\
\begin{equation} \label{za17}
\begin{split}
\eta({\bf{x}}_1, \ldots, {\bf{x}}_n) = \sum_{(\alpha_1,\ldots, \alpha_n) \in A^n} \theta_{A,\eta}(\alpha_1,& \ldots, \alpha_n){\bf{x}}_1(\alpha_1) \cdots {\bf{x}}_n(\alpha_n),\\\
&{\bf{x}}_1 \in l^2(A), \  \ldots, \ {\bf{x}}_n \in l^2(A),
\end{split}
\end{equation}
where
\begin{equation}
\theta_{A,\eta}(\alpha_1, \ldots, \alpha_n) := \eta({\bf{e}}_{\alpha_1},  \ldots, {\bf{e}}_{\alpha_n}), \  \ \ \ (\alpha_1,\ldots, \alpha_n) \in A^n,\\\
\end{equation}\\\
and $\{{\bf{e}}_{\alpha}: \alpha \in A\}$ is the standard basis of $l^2(A)$;  cf. \eqref{ker}.  In general, the sum on the right side of \eqref{za17} converges \emph{conditionally}, and is computed iteratively, coordinate-by-coordinate.  That is,

\begin{equation} \label{za18}
\begin{split}
\eta({\bf{x}}_1, \ldots, {\bf{x}}_n) = \sum_{\alpha_1 \in A} \bigg ( \cdots \bigg ( \sum_{\alpha_n \in A}\theta_{A,\eta}(\alpha_1,& \ldots, \alpha_n) {\bf{x}}_1(\alpha_1) \cdots {\bf{x}}_n(\alpha_n) \bigg) \cdots \bigg )\\\
&{\bf{x}}_1 \in l^2(A), \  \ldots, \ {\bf{x}}_n \in l^2(A),
\end{split}
\end{equation}
where sums over the respective coordinates are performed iteratively, in any order.  
In this work, we consider bounded $n$-linear functionals for which the right side of \eqref{za17} is absolutely summable, and focus specifically on kernels $\theta_{A,\eta}$ that are supported by \emph{fractional Cartesian products}. 

\begin{remark} \label{RLG}
\emph{\ The (so-called) \emph{little Grothendieck Theorem} is equivalent to the assertion that there exists a constant $K > 1$ such that for every finite scalar array  $(a_{jk})$, 
\begin{equation} \label{lg}
\begin{split}
 \sup \bigg \{\big |\sum_{j,k}a_{jk}  \langle{\bf{x}}_j, {\bf{y}}_k \rangle \big |:&  \ {\bf{x}}_j  \in l^2, \ {\bf{y}}_k \in l^2, \ \|{\bf{x}}_j\|_2 \leq 1, \  \sum _k \|{\bf{y}}_k\|_2^2 \leq 1\bigg \}  \\\
 & \leq  K \sup \bigg \{\big |\sum_{j,k}a_{jk}s_jt_k \big |:  |s_j|  \leq 1, \ \sum_k |t_k|^2 \leq 1 \bigg \}.\
 \end{split}
 \end{equation}
 (Cf. ~\cite[Theorem 5.2]{pisier2012grothendieck}.)  The inequality in \eqref{lg} is of course a quick consequence of the inequality in \eqref{i1}.  But it is implied also, though not as quickly, by  the  "one-dimensional" precursor to the Grothendieck inequality, 
 \begin{equation} \label{lc}
 \sup\{\|\eta_{{\bf{y}}}\|_{\tilde{\mathcal{V}}_1(B_H)}: {\bf{y}} \in B_H\} := k_G < \infty,
 \end{equation} 
  where $\eta_{{\bf{y}}}$ is defined in \eqref{fun1};  e.g., \eqref{lg} is implied by the integral representation of $\eta_{{\bf{y}}}$ in \eqref{replg}. 
 }
 
 \emph{Conversely,  the assertion that \eqref{lg} holds for every finite scalar array implies that every bounded linear functional on a Hilbert space is projectively bounded.  Indeed, $k_G$ in \eqref{lc} is the "smallest" possible $K$ in \eqref{lg}.  In particular, 
 by identifying a separable Hilbert space with the $L^2$-closure of the linear span of independent standard normal random variables, we obtain (by re-writing the proof of Proposition \ref{P8}) that \ $k_G := k_G^{\mathbb{C}} = \sqrt{4/\pi}$ \  if scalars are complex numbers, and  \  $k_G := k_G^{\mathbb{R}} = \sqrt{\pi/2}$ \ in the case of  real scalars. (Cf. ~\cite[Lemma 5.3]{pisier2012grothendieck}.)}
 
 \emph{If we add the continuity requirement, that  $\eta_{{\bf{y}}}$ be represented by integrals with uniformly bounded \emph{norm-continuous} integrands, as in \eqref{replg}, then we have 
 \begin{equation} \label{lcc}
 k_G \ \leq \ \sup\{\|\eta_{{\bf{y}}}\|_{\tilde{V}_1(B_H)}: {\bf{y}} \in B_H, \  \text{Hilbert space} \  H\} \ := k_{CG} \ \leq \  2e^{\frac{1}{2}}.
 \end{equation}\\
 I do not know which of the inequalities above are strict; cf. \eqref{g26} and Problem \ref{num}. 
  }\\
\end{remark}

  \subsection{Fractional Cartesian products} \ \ 
    Let $m$ be a positive integer.  A \emph{covering sequence} of $[m]$ is a set-valued sequence   \ $\mathcal{U} = (S_1, \ldots, S_n)$, such that $\emptyset \neq S_i \subset [m]$,  and
\begin{equation*}
\bigcup_{i=1}^nS_i = [m].
\end{equation*}
 Given sets $X_1, \ldots, X_m$,  \ and $S \subset [m]$, we consider the projection
\begin{equation} \label{g35}
\pi_S: \bigtimes_{j=1}^m X_j \rightarrow  \bigtimes_{j \in S} X_j := \emph{\bf{X}}^S,
\end{equation}
 defined by
\begin{equation} \label{g37}
\pi_S(\emph{\bf{x}}) = (x_j: j \in S), \ \ \emph{\bf{x}} = (x_1, \ldots, x_m) \in \bigtimes_{j=1}^m X_j.
\end{equation}
Then, given a covering sequence \ $\mathcal{U}= (S_1, \ldots, S_n)$ of $[m]$, we define \\\ 
\begin{equation} \label{g27}
\begin{split}
\emph{\bf{X}}^{\mathcal{U}} &:= \big\{ \big(\pi_{S_1}(\emph{\bf{x}}), \ldots, \pi_{S_n}(\emph{\bf{x}}) \big): \emph{\bf{x}} \in  \emph{\bf{X}}^{[m]} \big\}\\\\
& \subset \emph{\bf{X}}^{S_1} \times \cdots \times \emph{\bf{X}}^{S_n} := \emph{\bf{X}}^{[\mathcal{U}]}.
\end{split}
\end{equation}\\\
 We refer to  $\emph{\bf{X}}^{\mathcal{U}}$   as a  \emph{fractional Cartesian product} (based on $\mathcal{U}$), and to $ \emph{\bf{X}}^{[\mathcal{U}]}$ as its \emph{ambient product}. 

 \begin{example}[\emph{$3/2$-product}] \label{E1}
\emph{\ If \ $\mathcal{U} = (S_1, \ldots, S_n)$ covers $[m]$, and the $S_i$ are pairwise disjoint, then $\emph{\bf{X}}^{\mathcal{U}} = \emph{\bf{X}}^{[\mathcal{U}]}$.  For our purposes, we will focus on sequences \ $\mathcal{U}$ with the property that every $i \in [m]$ appears in at least two elements of \ $\mathcal{U}$.  The simplest nontrivial  example is 
\begin{equation}\label{b19}
\mathcal{U} = \big( \{1,2\}, \{2,3\}, \{1,3  \} \big). 
\end{equation}
In this instance, given sets $X_1$, $X_2,$  $X_3,$ we have
\begin{equation} \label{b4}
\emph{\bf{X}}^{\mathcal{U}} =  \big \{\big ((x_1,x_2), (x_2,x_3),(x_1,x_3) \big ): (x_1,x_2,x_3) \in X_1 \times X_2 \times X_3 \big \},
\end{equation}
and 
\begin{equation} \label{b5}
\emph{\bf{X}}^{[\mathcal{U}]} = (X_1 \times X_2) \times (X_2 \times X_3) \times (X_1 \times X_3).
\end{equation} 
We view $\emph{\bf{X}}^{\mathcal{U}}$ as a "3/2-fold" Cartesian product, a subset of the ambient   $3$-fold product  $\emph{\bf{X}}^{[\mathcal{U}]}$, and will use it throughout as a running example to illustrate workings of general definitions and arguments.  (The appearance of  the fraction $3/2$ is explained in Remarks \ref{R5}.i and \ref{R6}.i.) }\\

\end{example} 
Fractional Cartesian products provide a natural framework for the study of representing functions of $m$ variables, $m > 2$,  by functions of $k$ variables, $1 < k < m$. 
\begin{question} [cf. Question \ref{Q6}] \label{Q8} \ 
 Let  \  $\mathcal{U} = (S_1, \ldots, S_n)$ be a a covering sequence of $[m]$.  Given   $f \in l^{\infty}( \emph{\bf{X}}^{[m]})$,  can we find a probability space $(\Omega,  \mu)$, and representing functions $g_{\omega}^{(i)}$, indexed by $\omega \in \Omega$ and defined on   $\emph{\bf{X}}^{S_i}$ for $i \in [n]$,  such that for  all $\emph{\bf{x}} \in  \emph{\bf{X}}^{[m]}$ 

\begin{equation*}
\omega \mapsto \big( g_{\omega}^{(i)} \circ \pi_{S_i} \big)(\emph{\bf{x}})  \ \ \  \text{and} \ \ \   \omega \mapsto \|g_{\omega}^{(i)}\|_{\infty} 
\end{equation*} \\\
determine $\mu$-measurable functions on $\Omega,$ and
\begin{equation}\label{g28}\\\
f(\emph{\bf{x}}) = \int_{\Omega} \big( g_{\omega}^{(1)} \circ \pi_{S_1} \big)(\emph{\bf{x}})  \cdots  \big( g_{\omega}^{(n)} \circ \pi_{S_n} \big)(\emph{\bf{x}}) \mu(d \omega)
\end{equation}\\\
under the constraint
\begin{equation}\label{g29}
\int_{\Omega}  \|g_{\omega}^{(1)}\|_{\infty} \cdots  \|g_{\omega}^{(n)}\|_{\infty} \mu(d\omega) < \infty?
\end{equation}\\\
\end{question}
\noindent
We define  $\| f\|_{\tilde{\mathcal{V}}_{\mathcal{U}}}$  to be the infimum of the left side of \eqref{g29} taken over all representations of $f$ by \eqref{g28}, and 
\begin{equation} \label{a2}
\tilde{\mathcal{V}}_{\mathcal{U}}( \emph{\bf{X}}^{[m]}) := \{f\in l^{\infty}( \emph{\bf{X}}^{[m]}):\| f\|_{\tilde{\mathcal{V}}_{\mathcal{U}}} < \infty \}.
\end{equation}
Extending Proposition \ref{P4}, we identify  $\tilde{\mathcal{V}}_{\mathcal{U}}( \emph{\bf{X}}^{[m]})$ with an algebra of restrictions of Walsh transforms.  Specifically, we take the compact Abelian group
\begin{equation} \label{b6}
{\bf{\Omega}}_{\emph{\bf{X}}^{[\mathcal{U}]}} := \Omega_{\emph{\bf{X}}^{S_1}}\times \cdots \times \Omega_{\emph{\bf{X}}^{S_n}},
\end{equation}  
and the spectral set 
\begin{equation} \label{b7}
(\emph{\bf{R}}_{\emph{\bf{X}}})^{\mathcal{U}} := \big \{r_{\pi_{S_1}(\emph{\bf{x}})}\otimes  \cdots \otimes r_{\pi_{S_1}(\emph{\bf{x}})}: \emph{\bf{x}} \in  \emph{\bf{X}}^{[m]} \big \} ,
\end{equation} 
which is a fractional Cartesian product inside the ambient product 
\begin{equation} \label{b8}
(\emph{\bf{R}}_{{\emph{\bf{X}}}})^{[\mathcal{U}]} = R_{{\bf{X}}^{S_1}} \times \cdots \times R_{{\bf{X}}^{S_n}}.
\end{equation}
(See \eqref{g27}.)  For example, in the case of the $3/2$-product  (Example \ref{E1}) we have \\
\begin{equation} \label{b9}
{\bf{\Omega}}_{\emph{\bf{X}}^{[\mathcal{U}]}} = \{-1,1\}^{X_1\times X_2} \times \{-1,1\}^{X_2\times X_3} \times \{-1,1\}^{X_1\times X_3},
\end{equation}\\\
and\\\
\begin{equation} \label{b10}
(\emph{\bf{R}}_{\emph{\bf{X}}})^{\mathcal{U}} = \big \{r_{(x_1,x_2)}\otimes r_{(x_2,x_3)}\otimes r_{(x_1,x_3)}: (x_1,x_2,x_3) \in X_1 \times X_2 \times X_3 \}
\end{equation}\\\ 
with its ambient product\\\
\begin{equation} \label{b11}
(\emph{\bf{R}}_{\emph{\bf{X}}})^{[\mathcal{U}]} = R_{X_1\times X_2} \times R_{X_2\times X_3} \times R_{X_1\times X_3}. 
\end{equation}\\\
Observe, via Riesz representation and Hahn-Banach, that the dual space of \ $C_{(\emph{\bf{R}}_{\emph{\bf{X}}})^{\mathcal{U}}}({\bf{\Omega}}_{\emph{\bf{X}}^{[\mathcal{U}]}})$ (= \{continuous functions on ${\bf{\Omega}}_{\emph{\bf{X}}^{[\mathcal{U}]}}$ with spectrum in $(\emph{\bf{R}}_{\emph{\bf{X}}})^{\mathcal{U}}$\})  is \\\
    \begin{equation*} 
  B\big((\emph{\bf{R}}_{\emph{\bf{X}}})^{\mathcal{U}}\big)  := \widehat{M}({\bf{\Omega}}_{\emph{\bf{X}}^{[\mathcal{U}]}})/\{\hat{\lambda} \in \widehat{M}({\bf{\Omega}}_{\emph{\bf{X}}^{[\mathcal{U}]}}): \hat{\lambda} = 0 \ \text{on} \ (\emph{\bf{R}}_{\emph{\bf{X}}})^{\mathcal{U}}\} 
    \end{equation*}\\\
( = \{restrictions of Walsh transforms of measures on $M({\bf{\Omega}}_{\emph{\bf{X}}^{[\mathcal{U}]}})$ to $(\emph{\bf{R}}_{\emph{\bf{X}}})^{\mathcal{U}}$\}), and obtain
  \begin{proposition} \label{P5}
 \begin{equation*}
    \tilde{\mathcal{V}}_{\mathcal{U}}( \emph{\bf{X}}^{[m]}) = B\big((\emph{\bf{R}}_{\emph{\bf{X}}})^{\mathcal{U}}\big).
   \end{equation*}    
    \end{proposition}
\noindent Proposition \ref{P4} is the instance \ $\mathcal{U} = \big( \{1\}, \dots, \{n\} \big)$.\\\

\begin{remark} \label{R5} \ \ \\\
\em{\textbf{i.} Given infinite sets $E_1, \ldots, E_d$,  \  and  $F \subset E_1 \times \cdots \times E_d,$ 
we let \\\
\begin{equation} \label{ba19}
\begin{split}
\Psi_F(s) := \max \big \{|F\cap(A_1 \times \cdots \times A_d)|: A_i \subset E_i, \ |A_i| = s, \ i & \in [d] \big \},\\\\
 \ \ \ & s = 1,2, \ldots,
 \end{split} 
\end{equation}\\\\
and define the \emph{combinatorial dimension} of $F$ (relative to $E_1 \times \cdots \times E_d$) to be 
 \begin{equation}
 \dim F = \limsup_{s \rightarrow \infty} \frac{\log \Psi_F(s)}{\log s}. 
 \end{equation}  \\\
 If $\dim F = \alpha$, and
\begin{equation} \label{ba20}
0 < \liminf_{s \rightarrow \infty} \frac{\Psi_F(s)}{s ^{\alpha}}  \leq \limsup_{s \rightarrow \infty} \frac{\Psi_F(s)}{s ^{\alpha}} < \infty,
\end{equation}\\\
then we say that $F$ is an $\alpha$-\emph{product}.
}

 Given a covering sequence $\mathcal{U}$ of $[m]$, consider the solution to the linear programming problem,\\\  
\begin{equation} \label{b18} 
 \max \big \{v_1 + \cdots + v_m: \sum_{j \in S} v_j \leq 1,  \  S \in \mathcal{U}, \  v_i \geq 0, \  i \in [m] \big\} :=  \alpha(\mathcal{U})  = \alpha.\\\
\end{equation}\\\     
The main result in  ~\cite{Blei:1994} asserts that if  $X_j$  ($j \in [m]$) are infinite,  then $\emph{\bf{X}}^{\mathcal{U}}$ is an $\alpha$-product, and in particular,
\begin{equation} \label{ba18}
\dim\emph{\bf{X}}^{\mathcal{U}} = \alpha(\mathcal{U}).
\end{equation}
 Specifically, we have
\begin{equation} \label{a3}
\tilde{\mathcal{V}}_{\mathcal{U}}( \emph{\bf{X}}^{[m]}) = l^{\infty}( \emph{\bf{X}}^{[m]})  \  \  \Longleftrightarrow  \  \ \alpha(\mathcal{U}) = 1,
\end{equation} 
which is a consequence of precise relationships between combinatorial measurements and  $p-$\emph{Sidonicity} in a context of harmonic analysis; e.g., see ~\cite[Remark ii, p.492]{blei:2001}.  Moreover, for all covering sequences \ $\mathcal{U}_1$ and \  $\mathcal{U}_2$ of $[m]$, 
\begin{equation} \label{a4}
\alpha(\mathcal{U}_1) \neq \alpha(\mathcal{U}_2) \ \  \Longrightarrow  \ \ \tilde{\mathcal{V}}_{\mathcal{U}_1}( \emph{\bf{X}}^{[m]}) \neq  \tilde{\mathcal{V}}_{\mathcal{U}_2}( \emph{\bf{X}}^{[m]}).
\end{equation}
For example, for the sequence  $\mathcal{U}$ in \eqref{b19},  $\alpha(\mathcal{U}) = 3/2$ (by inspection), and then by \eqref{a3},
\begin{equation}\label{b20}
\tilde{\mathcal{V}}_{\mathcal{U}}( \emph{\bf{X}}^{[3]}) \subsetneq l^{\infty}( \emph{\bf{X}}^{[3]}). 
\end{equation}

The proper inclusion in  \eqref{b20}  can be proved directly, by verifying that for every positive integer $N$  there exist $\theta_N \in l^{\infty}([N]^3)$ so that
\begin{equation*}
\|\theta_N\|_{\infty} = 1,
\end{equation*}
and
\begin{equation} \label{z23}
\|\theta_N\|_{\mathcal{V}_{\mathcal{U}}([N]^3)} \underset{N \to \infty}{\longrightarrow} \infty.
\end{equation}\\\
The existence of such $\theta_N$  is guaranteed by the Kahane-Salem-Zygmund probabilistic estimates (e.g.,  \cite[ pp. 68-9]{Kahane1994some}), which imply that for every  $N > 0$ there exist \\
\begin{equation}
\theta_N(i,j,k) := \epsilon_{ijk} \in \{-1,1\}, \ \ \ (i,j,k) \in [N]^3
\end{equation}
so that 
\begin{equation} \label{z22}
\big \| \sum_{(i,j,k) \in [N]^3}  \epsilon_{ijk}  \ r_{ij}\otimes r_{jk}\otimes r_{ik} \big \|_{L^{\infty}} \leq K N^2,
\end{equation}\\
wherein Rademacher functions  are defined on  $\Omega_{\mathbb{N}^2}$, and $K > 0$ is an absolute constant.  We then let 
\begin{equation}
f_N = \frac{1}{KN^2}\sum_{(i,j,k) \in [N]^3}  \epsilon_{ijk}  \ r_{ij}\otimes r_{jk}\otimes r_{ik},
\end{equation}
and deduce, by applying duality (Proposition \ref{P5}),
\begin{equation} \label{z24}
 \sum_{(i,j,k) \in [N]^3}  \hat{f}_N(r_{ij}\otimes r_{jk}\otimes r_{ik}) \ \theta_N(i,j,k)  = \frac{N}{K} \leq \| \theta_N\|_{\mathcal{V}_{\mathcal{U}}([N]^3)}.
\end{equation}
 
But more is true:  if ${\bf{X}}^{\mathcal{U}}$ is an $\alpha$-product,
then
\begin{equation} \label{z25}
l^{\frac{2\alpha}{\alpha-1}}({\bf{X}}^{[m]}) \subset \tilde{\mathcal{V}}_{\mathcal{U}}( \emph{\bf{X}}^{[m]}),
\end{equation}
and if each $X_j$  is infinite ($j \in [m]$), then  
\begin{equation} \label{z26}
l^p( \emph{\bf{X}}^{[m]}) \not \subset \tilde{\mathcal{V}}_{\mathcal{U}}( \emph{\bf{X}}^{[m]}), \ \ \ p > \frac{2\alpha}{\alpha - 1},
\end{equation}\\ 
which, in particular, implies \eqref{a4}.  (As usual, $l^p(A)$  denotes the space of all scalar-valued functions ${\bf{x}}$ on a domain $A$ such that $\|{\bf{x}}\|_p := \big(\sum_{\alpha \in A}|{\bf{x}}(\alpha)|^p \big)^{\frac{1}{p}}  < \infty.$)   In the specific case of the $3/2$-product,  \eqref{z26}  can be obtained by modifying the proof (above) of \eqref{b20}, which also suggests how \eqref{z26} can be proved in the general case. 

All this (and more) is detailed in \cite[Ch. XII, Ch. XIII]{blei:2001};  see also the survey article \cite{blei2011measurements}. \\\\
\textbf{ii.} Suppose \ $\mathcal{U} = (S_1, \ldots, S_n)$ covers $[m]$.   If $f \in \tilde{\mathcal{V}}_n(\emph{\bf{X}}^{[\mathcal{U}]})$, then the restriction of $f$ to $\emph{\bf{X}}^{\mathcal{U}}$ is \emph{a fortiori} in $ \tilde{\mathcal{V}}_{\mathcal{U}}( \emph{\bf{X}}^{[m]})$, and we have
\begin{equation} \label{a5}
\|f\|_{\tilde{\mathcal{V}}_{\mathcal{U}}( \emph{\bf{X}}^{[m]})} \leq \|f\|_{ \tilde{\mathcal{V}}_n(\emph{\bf{X}}^{[\mathcal{U}]})}.\\\\\
\end{equation}
For example, if  \  $\mathcal{U} = \big (\{(1,2\},\{2,3\},\{1,3\} \big )$, and $$f \in \tilde{\mathcal{V}}_3\big ((X_1 \times X_2) \times (X_2 \times X_3) \times (X_1 \times X_3) \big ),$$
then by Proposition \ref{P4},  there exists $\lambda \in M({\bf{\Omega}}_{\emph{\bf{X}}^{[\mathcal{U}]}})$, where ${\bf{\Omega}}_{\emph{\bf{X}}^{[\mathcal{U}]}}$ is given by \eqref{b9}, such that\\\
\begin{equation}
f\big( (x_1,x_2),(x_3,x_4),(x_5,x_6) \big) = \hat{\lambda}\big(r_{(x_1,x_2)}\otimes r_{(x_3,x_4)}\otimes r_{(x_5,x_6)}\big), \ \ \ \ \ \ \ \ \ \ \ \ \ \ \\\
\end{equation}
 $$ \ \ \ \ \ (x_1,x_2) \in X_1 \times X_2, \ \ \ (x_3,x_4) \in X_3 \times X_4, \ \ \ (x_5,x_6)  \in X_1 \times X_3.$$ \\
Then (obviously!), 
\begin{equation}
f\big( (x_1,x_2),(x_2,x_3),(x_1,x_3) \big) = \hat{\lambda}\big(r_{(x_1,x_2)}\otimes  r_{(x_2,x_3)} \otimes r_{(x_1,x_3)}\big), \ \ \ \ \ \ \ \ \ \ \ \ \ \ \ \ \ \ \ \ \ \ \ \\\\
\end{equation}
 $$\ \ \ \ \ \ \ \ \ \ \ \ \ \ \ \ \ \ \ \ \ \ \ \ \ \ \ \ \ \ \ \  (x_1,x_2,x_3 ) \in X_1 \times X_2 \times X_3,$$
\\
which, by Proposition \ref{P5}, implies $f \in \tilde{\mathcal{V}}_{\mathcal{U}}({\bf{X}}^{[3]})$ and
$\|f\|_{\tilde{\mathcal{V}}_{\mathcal{U}}( \emph{\bf{X}}^{[3]})} \leq \|f\|_{ \tilde{\mathcal{V}}_3(\emph{\bf{X}}^{[\mathcal{U}]})}.$\\\

\end{remark}

\subsection{A characterization of projectively continuous functionals} \ \ We return to multilinear functionals on a Hilbert space.  Let  $m$ be a positive integer, and let $$\mathcal{U} = (S_1, \ldots, S_n )$$ be a covering sequence of $[m]$.   Given a set $A$ and $\theta \in l^{\infty}(A^m)$, we define (formally) an $n$-linear functional $\eta_{\mathcal{U},\theta}$ on \\\  $$l^2(A^{S_1}) \times \cdots \times l^2(A^{S_n})$$ by
\begin{equation} \label{g30}
 \eta_{\mathcal{U},\theta}(\emph{\bf{x}}_1, \ldots,  \emph{\bf{x}}_n) := \sum_{\bm{\alpha} \in A^m}  \theta(\bm{\alpha}) \  \emph{\bf{x}}_1\big(\pi_{S_1}(\boldsymbol{\alpha}) \big) \cdots \emph{\bf{x}}_n\big(\pi_{S_n}(\boldsymbol{\alpha}) \big),
 \end{equation}\\\
 \begin{equation*}
\ \ \ \ ( \emph{\bf{x}}_1, \ldots,  \emph{\bf{x}}_n) \in l^2(A^{S_1}) \times \cdots \times l^2(A^{S_n}).
\end{equation*}\\\
We refer to the functionals defined by \eqref{g30} as multilinear functionals \emph{based} on $\mathcal{U}$.  Note that the definition of $\eta_{\mathcal{U},\theta}$ at this point is completely formal:  modes of summation and convergence in \eqref{g30} have not been specified.  Denote
\begin{equation} \label{g32}
k_j(\mathcal{U}) := |\{i: j \in S_i\}|, \ \ j \in [m],
\end{equation}
 (\emph{incidence} of $j$ in $\mathcal{U}$), and
\begin{equation} \label{g32a}
I_{\mathcal{U}} := \min\{k_j(\mathcal{U}): j \in [m]\}.
\end{equation}
 \begin{lemma}[cf. {~\cite[Lemma 1.2]{Blei:1979fk}}] \label{L5}
Let \ $\mathcal{U} = (S_1, \ldots, S_n )$  be a covering sequence of $[m]$ with $I_{\mathcal{U}} \geq 2$, and  let   $\theta \in l^{\infty}(A^m)$.  Then\\\
\begin{equation*}
\begin{split}
\sum_{\bm{\alpha} \in A^m} \big| \theta(\bm{\alpha}) \  \emph{\bf{x}}_1\big(\pi_{S_1}(\boldsymbol{\alpha}) & \big)\cdots \emph{\bf{x}}_n\big(\pi_{S_n}(\boldsymbol{\alpha}) \big) \big|  \leq \|\theta\|_{\infty},\\\
& (\emph{\bf{x}}_1,\ldots,  \emph{\bf{x}}_n) \in B_{l^2(A^{S_1})} \times \cdots \times B_{l^2(A^{S_n})}.
\end{split}
\end{equation*}\\\
 That is,  $\eta_{\mathcal{U},\theta}$ is a bounded $n$-linear functional on $l^2(A^{S_1}) \times \cdots \times l^2(A^{S_n})$, and    
 \begin{equation*}
 \|\eta_{\mathcal{U},\theta}\|_{\infty}  \leq \|\theta\|_{\infty}.
 \end{equation*}
 \end{lemma}
\begin{proof}[Proof (by induction on m).] It suffices to verify 
\begin{equation} \label{g31}
\sum_{\bm{\alpha} \in A^m} \big |\emph{\bf{x}}_1\big(\pi_{S_1}(\boldsymbol{\alpha}) \big) \cdots \emph{\bf{x}}_n\big(\pi_{S_n}(\boldsymbol{\alpha}) \big) \big| \leq \|\emph{\bf{x}}_1\|_2 \cdots \|\emph{\bf{x}}_n\|_2,
\end{equation}
$$ (\emph{\bf{x}}_1,\ldots,  \emph{\bf{x}}_n) \in l^2(A^{S_1}) \times \cdots \times  l^2(A^{S_n}).$$ 
Suppose  $m=1$,  and  $$\mathcal{U} = (\underbrace{\{1\},\ldots, \{1\}}_n),  \ \ n = I_{\mathcal{U}} \geq 2.$$  In this case, apply to \eqref{g31} the $n$-linear H\"older inequality with conjugate vector $(\frac{1}{n}, \ldots, \frac{1}{n})$, and  then apply $\|\cdot\|_{I_{\mathcal{U}}} \leq \|\cdot\|_2$ (convexity):
  \begin{align*}
 \sum_{\alpha \in A} |\emph{\bf{x}}_1(\alpha) \cdots \emph{\bf{x}}_n(\alpha)| & \leq \|\emph{\bf{x}}_1\|_n \cdots \|\emph{\bf{x}}_n\|_n \\
 & \leq \|\emph{\bf{x}}_1\|_2 \cdots \|\emph{\bf{x}}_n\|_2,  \  \ \ \ \ \ ( \emph{\bf{x}}_1, \ldots,  \emph{\bf{x}}_n) \in  l^2(A) \times \cdots \times  l^2(A).
 \end{align*}\\
 Now take $m > 1$, and let  $\mathcal{U} = (S_1, \ldots, S_n)$ be a covering sequence of  $[m]$.  Let
  \begin{equation*}
  S_i^{\prime} = S_i \setminus \{m\}, \ \ i = 1, \ldots n.
  \end{equation*}
  Then,  $\mathcal{U}^{\prime} = (S_1^{\prime}, \ldots, S_n^{\prime})$ covers $[m-1]$, and  $I_{\mathcal{U}^{\prime}} \geq 2$.   Let $\emph{\bf{x}}_1 \in l^2(A^{S_1}), \ldots, \emph{\bf{x}}_n \in l^2(A^{S_n})$.  Denote 
\begin{equation}\label{ba12}  
  T := \{i:  m \in S_i\}.
 \end{equation} 
 For $i \in T$ and $u \in A$, define $\emph{\bf{x}}_{u,i} \in l^2(A^{S_i^{\prime}})$ by 
\begin{equation*}
\emph{\bf{x}}_{u,i}(\bm{\alpha}) = \emph{\bf{x}}_i(\bm{\alpha},u),  \ \ \ \bm{\alpha} \in A^{S_i^{\prime}}.
\end{equation*}  
 By applying the induction hypothesis, and then \eqref{g31} in the case $m=1$ with $n = |T|$, we obtain\\\
  \begin{align*}
 & \sum_{\bm{\alpha} \in A^m} \big|\emph{\bf{x}}_1\big(\pi_{S_1}(\boldsymbol{\alpha}) \big) \cdots \emph{\bf{x}}_n\big(\pi_{S_n}(\boldsymbol{\alpha}) \big)\big| = & \\
   & \ \ \ \ \ \ \ \ \ \ =  \sum_{u \in A} \  \sum_{\bm{\alpha} \in A^{[m-1]}}  \ \prod_{i \in [n] \setminus T}\big|\emph{\bf{x}}_i\big(\pi_{S_i^{\prime}}(\bm{\alpha})\big)\big| \  \prod_{i \in T}|\emph{\bf{x}}_{u,i}\big(\pi_{S_i^{\prime}}(\bm{\alpha}) \big)\big| 
\end{align*}\\\
\begin{equation*}
   \ \ \ \ \ \ \ \ \ \ \ \ \ \ \ \ \ \  \leq \prod_{i \in [n] \setminus T}\|\emph{\bf{x}}_i \|_2 \sum_{u \in A} \ \prod_{i \in T}\|\emph{\bf{x}}_{u,i}\|_2 \  \leq \ \|\emph{\bf{x}}_1\|_2 \cdots \|\emph{\bf{x}}_n\|_2.\\\
  \end{equation*}\\
 \end{proof}
 \begin{remark} \label{R6} \  \\\
 {\bf{i.}} \ \emph{We illustrate the proof of Lemma \ref{L5}  in the case $$\mathcal{U} = \big( \{1,2\}, \{2,3\}, \{1,3\} \big ).$$  Taking $\theta = 1$, we have 
 \begin{equation} \label{b12}
 \begin{split}
 \eta_{\mathcal{U}}({\bf{x}}_1,{\bf{x}}_2,{\bf{x}}_3)  = \sum_{(\alpha_1,\alpha_2,\alpha_3) \in A^3}{\bf{x}}_1&(\alpha_1,\alpha_2){\bf{x}}_2(\alpha_2,\alpha_3){\bf{x}}_3(\alpha_1,\alpha_3),\\\
 &{\bf{x}}_j \in B_{l^2(A^2)}, \ \ j = 1, 2, 3.\\\
 \end{split}
 \end{equation}
 Then, by applying Cauchy-Schwarz three times in succession to the sums over $\alpha_1, \alpha_2,$ and $ \alpha_3$, we obtain
 \begin{equation} \label{b13}
 \begin{split}
 \sum_{\alpha_3 \in A} \  \sum_{\alpha_2 \in A} \ \sum_{\alpha_1 \in A}|{\bf{x}}_1(\alpha_1,\alpha_2){\bf{x}}_2(\alpha_2,\alpha_3){\bf{x}}_3&(\alpha_1,\alpha_3)| \leq  1,\\\
 &{\bf{x}}_j \in B_{l^2(A^2)}, \ \ j = 1, 2, 3,
 \end{split}
 \end{equation}\\\
 which implies Lemma \ref{L5} in this instance.}
 
 \emph{Given a positive integer $s$, and arbitrary $s$-sets $E_j \subset A^2, \ j =1, 2, 3$ (i.e., $|E_j| = s$),  we put in \eqref{b13}
 \begin{equation}
 {\bf{x}}_j = \mathds{1}_{E_j}/\sqrt{s}, \ \ \ \ j = 1, 2, 3,
 \end{equation}
  and deduce
 \begin{equation}\label{b14}
 |{\bf{A}}^{\mathcal{U}} \cap (E_1 \times E_2 \times E_3)| \leq  s^{\frac{3}{2}},
 \end{equation}
 where, as in \eqref{b4} and \eqref{b5}, 
 \begin{equation} \label{b15}
 \begin{split}
 {\bf{A}}^{\mathcal{U}} &= \big\{ \big((\alpha_1,\alpha_2), (\alpha_2,\alpha_3),(\alpha_1,\alpha_3) \big ): (\alpha_1, \alpha_2, \alpha_3) \in A^3 \big \}\\\\
 & \subset A^2 \times A^2 \times A^2 = {\bf{A}}^{[\mathcal{U}]}.
 \end{split}
 \end{equation}
 }

\emph{To show that the estimate in \eqref{b14} is asymptotically best possible, we take  $F \subset A$ to be a $k$-set, where $k$ is an arbitrary positive integer, and  $E_1 = E_2 = E_3 = F^2$.  We take $s := k^2 = |E_j|, \ j = 1,2,3,$ \ and then obtain
\begin{equation} \label{b16}
|{\bf{A}}^{\mathcal{U}} \cap (E_1 \times E_2 \times E_3)| = |F^3| = s^{\frac{3}{2}}.
\end{equation}
}

\emph{Recalling the combinatorial gauge $ \Psi$ in \eqref{ba19}, we have in this case \\
\begin{equation} \label{b16a}
\begin{split}
\Psi_{{\bf{A}}^{\mathcal{U}}}(s) = \max \big \{|{\bf{A}}^{\mathcal{U}} \cap (E_1 \times E_2 \times E_3)|: s\text{-sets} \ E_j \subset A^2, & \ j =1, 2, 3 \big \}\\\
& s = 1, 2, \ldots \ .
\end{split} 
\end{equation}\\
Combining \eqref{b14} and \eqref{b16}, we obtain
\begin{equation} \label{b17}
\lim_{s \rightarrow \infty} \frac{\Psi_{{\bf{A}}^{\mathcal{U}}}(s)}{s^{\frac{3}{2}}} = 1,
\end{equation}\\
which means that  ${\bf{A}}^{\mathcal{U}}$  is a  $3/2$-product;  cf. \eqref{ba20}. The statement in \eqref{b17} is a special case of the general result in \eqref{ba18}.\\\ 
}

\noindent
\emph{{\bf{ii.}} \ The functionals  $\eta_{\mathcal{U}, \theta}$ of Lemma \ref{L5} can be classified according to their underlying \  $\mathcal{U}$ put in standard form.  To illustrate, take
\begin{equation} \label{b26}
\mathcal{U} = \big ( \{1,2\}, \{1\}, \{2\} \big ),
\end{equation}
and let $\theta \in l^{\infty}(A^2)$.  The evaluations of the corresponding trilinear functional $$\eta_{\mathcal{U},\theta}({\bf{x}}, {\bf{y}}, {\bf{z}}), \ \ ({\bf{x}}, {\bf{y}}, {\bf{z}}) \in l^2(A^2) \times l^2(A) \times l^2(A),$$ can be  realized as evaluations of  $\eta_{\mathcal{U}^{\prime},\theta}$ on $l^2(A^2) \times l^2(A^2) \times l^2(A^2)$,  where 
\begin{equation}
\mathcal{U}^{\prime} = \big ( \{1,2\}, \{1,2\}, \{1,2\} \big ).
\end{equation}
Specifically, for ${\bf{x}} \in l^2(A)$, define
\begin{equation} \label{b22}
 \tilde{\bf{x}}(\alpha_1,\alpha_2) =  \left \{
\begin{array}{lll}
{\bf{x}}(\alpha) &  \quad  \alpha_1 = \alpha_2  & \quad \\\\
0 & \quad \alpha_1 \not = \alpha_2,  \ \ \ (\alpha_1,\alpha_2) \in A^2,
\end{array} \right. 
\end{equation}\\\
and then write
\begin{equation} \label{b23}
\begin{split}
\eta_{\mathcal{U},\theta}({\bf{x}}, {\bf{y}}, {\bf{z}}) = \eta_{\mathcal{U}^{\prime},\theta}({\bf{x}}, \tilde{{\bf{y}}}, \tilde{{\bf{z}}}), \\\
&({\bf{x}}, {\bf{y}}, {\bf{z}}) \in l^2(A^2) \times l^2(A) \times l^2(A).
\end{split} 
\end{equation}
Next we identify $\{1,2\}$ with $\{1\}$, and replace $\mathcal{U}^{\prime}$ by  
\begin{equation}
\mathcal{U}^{\prime\prime} = \big (\{1\}, \{1\},\{1\} \big ).
\end{equation}
The evaluations of $\eta_{\mathcal{U},\theta}$ can now be realized as evaluations of a trilinear functional based on $\mathcal{U}^{\prime\prime}$. Specifically, write $B = A^2$, and rewrite \eqref{b22} and \eqref{b23} accordingly:  For ${\bf{x}} \in l^2(A^2)$, let
\begin{equation}
\tilde{{\bf{x}}}(\beta) = {\bf{x}}(\alpha_1,\alpha_2), \ \ \beta = (\alpha_1,\alpha_2),
\end{equation}
and for  ${\bf{x}} \in l^2(A)$, let
\begin{equation}
\tilde{\bf{x}}(\beta) =  \left \{
\begin{array}{lll}
{\bf{x}}(\alpha)  \quad & \text{if}  \quad & \beta = (\alpha, \alpha)  \ \ \ \alpha \in A \\
0  \quad  & \text{if}  \quad & \beta = (\alpha_1, \alpha_2), \ \ \ \alpha_1 \not = \alpha_2, \ \ \  (\alpha_1,\alpha_2) \in A^2. 
\end{array} \right. 
\end{equation}\\\
Write
\begin{equation}
\tilde{\theta}(\beta) = \theta(\alpha_1,\alpha_2), \ \ \ \beta = (\alpha_1, \alpha_2).
\end{equation}
Then,
\begin{equation} 
\begin{split}
\eta_{\mathcal{U},\theta}({\bf{x}}, {\bf{y}}, {\bf{z}}) = \eta_{\mathcal{U}^{\prime\prime},\tilde{\theta}}(\tilde{{\bf{x}}}, \tilde{{\bf{y}}}, \tilde{{\bf{z}}}), \\\
&({\bf{x}}, {\bf{y}}, {\bf{z}}) \in l^2(A^2) \times l^2(A) \times l^2(A),
\end{split} 
\end{equation}
where $ \eta_{\mathcal{U}^{\prime\prime},\tilde{\theta}}$ is the trilinear functional defined on $l^2(B) \times l^2(B) \times l^2(B)$.  Similarly, any trilinear functional based on   
\begin{equation} \label{b25}
 \mathcal{V} = \big(\{1,2,3\}, \{2,3\},\{3\} \big)
 \end{equation}
is "subsumed" by a trilinear functional based also on $\big(\{1\}, \{1\},\{1\} \big)$;  that is, any trilinear functional based on $ \mathcal{V}$ can be realized as a trilinear functional based on $\mathcal{U}^{\prime\prime}$.   In this sense,  respective trilinear functionals based on \ $\mathcal{U}$ in \eqref{b26} and $\mathcal{V}$  in \eqref{b25} belong to the same class.
}

\emph{We say that a covering sequence $\mathcal{U} = (S_1, \ldots, S_n)$ of $[m]$  is in \emph{standard form} (or, $\mathcal{U}$ is \emph{standard})  if for all $j$ and  $k$ in $[n]$, 
\begin{equation}
S_j \neq S_k \  \ \Rightarrow \  \ S_j \triangle S_k  \not = \emptyset,
\end{equation}
($\triangle = $ symmetric difference), and if for all $j$ and $k$ in $[m]$ \  ($j \not = k$), there exists $S_i \in \mathcal{U}$ such that
\begin{equation}
j \in S_i \  \ \text{and} \ \  k \not \in S_i.
\end{equation}
Moreover, we consider two standard covering sequences of $[m]$  $$\mathcal{U} = (S_1, \ldots, S_n)  \ \ \text{and} \ \  \mathcal{V}= (T_1, \ldots, T_n)$$ to be equivalent if there exist permutations $\tau$ of $[m]$ and $\sigma$ of $[n]$ such that $\tau[S_i] = T_{\sigma(i)}$ for every $i \in [n]$.  For example,
\begin{equation}
\begin{split}
\mathcal{U} &= \big (\{1,2,3\}, \{2,3,4\},\{1,3,4\} \big)\\\  \text{and}\\   \mathcal{V} &= \big (\{1,2,3\}, \{1,2,4\}, \{1,3,4\} \big)
\end{split}
\end{equation}
are in standard form, and are equivalent via the cycles $\tau = (123)$ and $\sigma = (23)$.}

\emph{We say that a covering sequence $\mathcal{U}$ is \emph{subsumed} by a covering sequence $\mathcal{V}$, and write
$\mathcal{U} \prec \mathcal{V}$, if every multilinear functional based on $\mathcal{U}$ can be realized  as a multilinear functional based on $\mathcal{V}$. (We forego a precise definition.)
} 

\emph{ Every covering sequence  $\mathcal{U} = (S_1, \ldots, S_n)$ of $[m]$ is subsumed by a standard sequence $\mathcal{U}^{\prime \prime}$.  To obtain such $\mathcal{U}^{\prime \prime},$  first let $\mathcal{U}^{\prime}$ be a sequence derived from $\mathcal{U}$ by replacing $S_j$ with $S_k$ whenever $S_j \subsetneq S_k$ for $j$ and $k$ in $[n]$.  Next we consider $j$ and $k$  in $[m]$  to be equivalent (relative to $\mathcal{U}^{\prime}$) if  for all $S_i \in \mathcal{U}^{\prime}$,
\begin{equation} \label{b24}
j \in S_i \Leftrightarrow k \in S_i,  
\end{equation}\\
and then take  $\{j_1, \ldots, j_{\ell}\} \subset [m]$ to be a list of equivalence class representatives.  Now let
\begin{equation}
S_i^{\prime \prime} = \{k: j_k \in S_i \}, \ \ i \in [n],
\end{equation}
and
\begin{equation}
\mathcal{U}^{\prime \prime} = (S_1^{\prime \prime}, \ldots, S_n^{\prime \prime}).
\end{equation}
Then, $\mathcal{U}^{\prime \prime}$ is a standard covering sequence of $[\ell]$,   
 $$\alpha(\mathcal{U}^{\prime \prime }) =  \alpha(\mathcal{U}^{\prime }) = \alpha(\mathcal{U}), \ \ \  I_{\mathcal{U}^{\prime \prime}} = I_{\mathcal{U}^{\prime}}  \geq I_{\mathcal{U}}, \ \  \text{and} \ \  \mathcal{U} \prec\mathcal{U}^{\prime \prime}.$$  Observe that $\mathcal{U}^{\prime \prime}$ need not be unique (even up to equivalence).  
} 

 \emph{For every integer $n \geq 2$, there is a finite number $\mathcal{S}(n)$ of standard covering sequences $\mathcal{U}$ with $n$ terms and $I_{\mathcal{U}} \geq 2$.   For $n = 2$ we have only  $\mathcal{U} = \big (\{1\}, \{1\} \big)$.  For $n = 3$, we have (up to equivalence)
\begin{equation}\label{st3}
\begin{split}
\mathcal{U}_1 &= \big(\{1\}, \{1\},\{1\} \big ),\\\
\mathcal{U}_2 &= \big(\{1,2\}, \{2,3\},\{1,3\} \big ),\\\
\mathcal{U}_3 &= \big(\{1,2,4\}, \{2,3,4\},\{1,3,4\} \big ),\\\
\text{and}\\\
\mathcal{U}_4 &= \big(\{1,2,4\}, \{2,3,4\},\{1,3\} \big ),\\\
\end{split}
\end{equation}
and no more.  Note that $\alpha(\mathcal{U}_1) = 1$, whereas \\\
\begin{equation} \label{st4}
\alpha(\mathcal{U}_2) = \alpha(\mathcal{U}_3) = \alpha(\mathcal{U}_4) = \frac{3}{2}.\\\\
\end{equation}\\\
Note also that $\mathcal{U}_1 \prec \mathcal{U}_i$ for $i = 2, 3, 4,$ and $\mathcal{U}_2 \prec \mathcal{U}_i$ for $ i = 3, 4.$  Otherwise, $\mathcal{S}(n)$ grows rapidly to infinity as $n$ increases.}

 \end{remark}
\vskip0.40cm

The theorem below addresses the question:   when are the  functionals in Lemma \ref{L5} projectively continuous? 
\begin{theorem} \label{T3}
 Let  \ $\mathcal{U} = (S_1, \ldots, S_n )$  be a covering sequence of $[m]$ with $I_{\mathcal{U}} \geq 2$.  Let $A$ be an infinite set, and let $\theta \in l^{\infty}(A^m)$. 
  
  If $\theta \in \tilde{\mathcal{V}}_{\mathcal{U}}(A^m)$, then $\eta_{\mathcal{U},\theta}$ is projectively continuous. 
 
 If  $\eta_{\mathcal{U},\theta}$ is projectively bounded, then $\theta \in \tilde{\mathcal{V}}_{\mathcal{U}}(A^m)$, and (therefore) $\eta_{\mathcal{U},\theta}$ is projectively continuous.
 
 Moreover,
  \begin{equation} \label{g46}
  K^{-\beta_{\mathcal{U}}} \|\eta_{\mathcal{U},\theta}\|_{\tilde{V}_n(\emph{\bf{B}}^{[\mathcal{U}]})} \leq \|\theta\|_{\tilde{\mathcal{V}}_{\mathcal{U}}(A^m)} \leq \|\eta_{\mathcal{U},\theta}\|_{\tilde{\mathcal{V}}_n(\emph{\bf{B}}^{[\mathcal{U}]})},
  \end{equation} 
  where $$\emph{\bf{B}}^{[\mathcal{U}]} := B_{l^2(A^{S_1})} \times \cdots \times B_{l^2(A^{S_n})},$$
  \begin{equation*}
  \beta_{\mathcal{U}} := \sum_{j=1}^m k_j(\mathcal{U}) \  \ \big(= \sum_{i=1}^n|S_i| \  \big),
  \end{equation*}
  and $K >1$ is an absolute constant.
\end{theorem}
\vskip0.30cm 
\section{\bf{Proof of Theorem \ref{T3}}}\label{s8}
The proof has three parts. \

\subsection{A multilinear Parseval-like formula} \ \    Let  $m$ be a positive integer, and let  $\mathcal{U} = (S_1, \ldots, S_n)$ be a covering sequence of $[m]$ with $I_{\mathcal{U}} \geq 2$.  Taking $\theta = 1$, we denote 
\begin{equation} \label{g45}
 \eta_{\mathcal{U}}(\emph{\bf{x}}_1, \ldots,  \emph{\bf{x}}_n) := \sum_{\bm{\alpha} \in A^m}  \emph{\bf{x}}_1\big(\pi_{S_1}(\boldsymbol{\alpha}) \big) \cdots \emph{\bf{x}}_n\big(\pi_{S_n}(\boldsymbol{\alpha}) \big),
 \end{equation}
 \begin{equation*}
\ \ \ \ ( \emph{\bf{x}}_1, \ldots,  \emph{\bf{x}}_n) \in l^2(A^{S_1}) \times \cdots \times l^2(A^{S_n}).
\end{equation*}\\\
We proceed to show that $\eta_{\mathcal{U}}$ is projectively continuous.
\begin{lemma}[cf. Theorem \ref{T2}, Corollalry \ref{full}] \label{L6}
Let $A$ be an infinite set.  For every integer $k \geq 2$, there exists a one-one map
  \begin{equation*}
\Phi_k : l^2(A) \rightarrow L^{\infty}(\Omega_{A}, \mathbb{P}_{A}),
  \end{equation*}
 which is continuous with respect to the  $l^2(A)$-norm (on its domain) and the $L^2(\Omega_{A}, \mathbb{P}_{A})$-norm (on its range), and has the following properties:

  \begin{equation}\label{g36}
  \|\Phi_k(\textbf{\emph{x}})\|_{L^{\infty}} \leq K \|\textbf{\emph{x}}\|_2, \ \ \ \textbf{\emph{x}} \in l^2(A),
  \end{equation}
  where $K > 1$ is an absolute constant;

  \begin{equation} \label{z27}
   \Phi_k(\xi\textbf{\emph{x}}) = \xi \Phi_k(\textbf{\emph{x}}), \ \  \xi \in \mathbb{C}, \ \  {\bf{x}} \in l_{\mathbb{R}}^2(A);
  \end{equation}

  \begin{equation}\label{g34}
  \begin{split}
 & \sum_{\alpha \in A}  \textbf{\emph{x}}_1(\alpha) \cdots \textbf{\emph{x}}_k(\alpha)  =  \big(\Phi_k(\textbf{\emph{x}}_1)\convolution \cdots \convolution \Phi_k(\textbf{\emph{x}}_k)\big)(\textbf{\emph{e}})\\\\
  &= \int_{(\Omega_A)^{k-1}} \left( \big( \prod_{j=1}^{k-1}\Phi_k(\textbf{\emph{x}}_j)(\omega_j) \big)  \Phi_k(\textbf{\emph{x}}_k)(\omega_1 \cdots \omega_{k-1}) \right) d\omega_1 \cdots  d\omega_{k-1}\\\\
  &= \sum_{\gamma \in \widehat{\Omega}_A} \widehat{\big(\Phi_k(\textbf{\emph{x}}_1)}(\gamma) \cdots \widehat{\big(\Phi_k(\textbf{\emph{x}}_k)}(\gamma),
  \end{split}
  \end{equation}
  \begin{equation*}
  \ \ \ \textbf{\emph{x}}_1 \in  l^2(A),  \ \ldots, \ \textbf{\emph{x}}_k \in  l^2(A),
  \end{equation*}\\\
where $d\omega_1 \cdots  d\omega_{k-1}$ stands for $\mathbb{P}_A(d\omega_1) \times \cdots \times  \mathbb{P}_A(d\omega_{k-1})$;  $\convolution$ denotes convolution over $\Omega_A$, and $\textbf{\emph{e}}(\alpha) = 1$ for $ \alpha \in A$.
\end{lemma}
\begin{proof}[Sketch of proof]  For $k = 2$,  $\Phi_2$ is the map in Corollary \ref{full}.  To produce $\Phi_k$ for $k > 2$,  replace the imaginary $i$ in \eqref{gg5} with $e^{\frac{\pi i}{k}}$.  That is, let
\begin{equation} \label{g47}
\Phi_k(\textbf{x}) = \sum_{j=1}^{\infty}(e^{\frac{\pi i}{k}})^{(j-1)}\|\textbf{x}^{(j)}\|_2 \ Q_{A_j}(\sigma \textbf{x}^{(j)}), \ \ \   \textbf{x} \in  l_{\mathbb{R}}^2(A).
\end{equation}
 Note that the presence of $e^{\frac{\pi i}{k}}$  in \eqref{g47}  guarantees a representation of 
\begin{equation}
 \big(\Phi_k({\bf{x}}_1)\convolution \cdots \convolution \Phi_k({\bf{x}}_k)\big)(\boldsymbol{\omega}_0)\\\\ 
\end{equation}
by alternating series analogous to \eqref{gg8} with $t = 1$.  The proof now is similar to that of  Corollary \ref{full}. 
\end{proof}

\begin{remark} \label{Kn}
\emph{By an estimate nearly identical to that in \eqref{approx}, we obtain \eqref{g36} with 
\begin{equation} \label{approxn}
K \  = \  \frac{2\sqrt{2e}}{\sqrt{2} - \sqrt{e - e^{-1} - 2}} \ .
\end{equation}\\
If we "perturb" the definition of $\Phi_k$ by $t > c > 0$, as in  \S \ref{sf}  (and thereby lose the homogeneity in \eqref{z27}), then we  obtain
\begin{equation}
\begin{split}
K \  = \ 2K(t_{\text{min}}) &= \min\bigg \{ \frac{2t e^{1/2t^2}}{1-t \sqrt{\sinh(1/t^2) - 1/t^2}}: \  \  t > c \bigg \} \\\\
& <  \  \frac{2\sqrt{2e}}{\sqrt{2} - \sqrt{e - e^{-1} - 2}} \ . 
\end{split}
\end{equation} 
}
\end{remark}
\vskip0.3cm

Next, given a covering sequence $\mathcal{U} = (S_1,\ldots,S_n)$ of $[m]$, we derive an integral representation of $\eta_{\mathcal{U}}$ that extends the usual (bilinear) Parseval formula
\begin{equation}
\begin{split}
(f \convolution g)(\textbf{e}) &= \int_{\omega \in \Omega_A}f(\omega)g(\omega)\mathbb{P}_A(d\omega) \\\\\
& = \sum_{\gamma \in \Hat{\Omega}_A} \hat{f}(\gamma)\hat{g}(\gamma)\\\\\
&= \eta_{(\{1\},\{1\})}(\hat{f},\hat{g}),  \ \ \ \ \ (f,g) \in L^2(\Omega_A, \mathbb{P}_A) \times L^2(\Omega_A, \mathbb{P}_A).
\end{split}
\end{equation} \\\\
\noindent
To this end, for $j \in [m]$ let
\begin{equation*}
\kappa_{\mathcal{U}}(j) = \kappa(j) := \max\{i: j \in S_i\}
\end{equation*}
(the index of the "last" term in $\mathcal{U}$ that contains $j$), and 
\begin{equation*}
T_j = \{i: j \in S_i, \ i < \kappa(j) \}.
\end{equation*}\\\
Recall that $k_j = k_j(\mathcal{U})$ is the incidence of $j$  in $\mathcal{U}$ (defined in \eqref{g32}),   and note  that  $|T_j| = k_j -1.$  For $ j \in [m]$,  denote \\\ $$\bm{\omega}_j := \big (\omega_{\ell j}:  \ \ell \in T_j \big) \in \Omega_A^{T_j},$$\\\ and then for $i \in [n]$, let
\begin{equation*}
 \xi_{ij}(\bm{\omega}_1, \ldots, \bm{\omega}_m)  =   \left \{
\begin{array}{ccc}
\omega_{ij}& \quad  \text{if} \quad  i  \in T_j  \\\
\\\
 \prod_{\ell \in T_j} \omega_{\ell j} & \quad   \text{ if} \quad   i  = \kappa(j). 
\end{array} \right.
\end{equation*}\\\\
Given $$(f_1, \ldots, f_n) \in L^2\big (\Omega_A^{S_1},\mathbb{P}_A^{S_1}\big) \times  \cdots \times L^2 \big(\Omega_A^{S_n},\mathbb{P}_A^{S_n}\big),$$ define\\\
\begin{equation} \label{g33}
\begin{split}
&\convolution_{\mathcal{U}}(f_1, \ldots, f_n) :=\\\\
& \int_{ \bm{\omega}_1 \in \Omega_A^{T_1}} \bigg (  \cdots \bigg (\int_{ \bm{\omega}_m \in \Omega_A^{T_m}}  \big \{ \prod_{i=1}^n f_i \big (\xi_{ij}(\bm{\omega}_1, \dots, \bm{\omega}_m): j \in S_i \big )\big \}d\bm{\omega}_m \bigg) \cdots \bigg )d\bm{\omega}_1,
\end{split}
\end{equation}\\\
where  $$d \bm{\omega}_j :=  \prod_{i \in T_j} d \omega_{ij}$$ stands for $$ \mathbb{P}_A^{T_j}(d\bm{\omega}_j), \ \ \ j = 1, \ldots, m.$$\\\

For the covering sequence $\mathcal{U} = \big (\overbrace{ \{1\}, \ldots, \{1\} }^n\big )$ of $[1]$,  the right side of \eqref{g33} is the usual $n$-fold convolution of $f_1 \in L^2(\Omega_A, \mathbb{P}_A), \ldots, f_n \in L^2(\Omega_A, \mathbb{P}_A), $ evaluated at $\textbf{e}$  (as in \eqref{g34} with $n = k$).  In this instance, 
\begin{equation*}
\convolution_{\mathcal{U}}(f_1,\ldots,f_n) = \sum_{\gamma \in \widehat{\Omega}_A} \hat{f}_1(\gamma) \cdots \hat{f}_n(\gamma).
\end{equation*}
In the general case, the iterated integrals in \eqref{g33} form a succession of $k_j$-fold convolutions ($j \in [m]$).  The lemma below -- a "fractional-multilinear" extension of  Parseval's  formula -- can be proved by induction on $m \geq 1$.  (Proof is omitted.)  

\begin{lemma}\label{L7} Let  $m$ be a positive integer, and let  \  $\mathcal{U} = (S_1, \ldots, S_n)$ be a covering sequence of $[m]$ with $I_{\mathcal{U}} \geq 2$.  Then, for $f_i \in L^2(\Omega_A^{S_i},\mathbb{P}_A^{S_i}), \ \ i \in [n]$,\\
\begin{equation*}
\begin{split}
\convolution_{\mathcal{U}}(f_1, \ldots, f_n) &= \sum_{\bm{\chi} \in \widehat{\Omega}_A^m}\hat{f}_1\big(\pi_{S_1}(\bm{\chi}) \big) \cdots \hat{f}_n\big(\pi_{S_n}(\bm{\chi})\big)\\\\
&= \eta_{\mathcal{U}}(\hat{f_1}, \ldots, \hat{f_n}),\\\\
\end{split}
\end{equation*}
where $$\pi_{S_i}: (\widehat{\Omega}_A)^m \rightarrow (\widehat{\Omega}_A)^{S_i}$$ are the projections in \eqref{g35} defined by \eqref{g37} with $X_j = \widehat{\Omega}_A, \ \ j \in [m]$.\\\
\end{lemma}
\vskip0.4cm
The result below is an extension of Lemma \ref{L6}.  Its proof uses induction in a framework of fractional Cartesian products.  In Remark \ref{R7}.i, we illustrate the proof by running the arguments in the archetypal case of a $3/2$-product.

\begin{theorem} \label{T4}
Let $A$ be an infinite set.  For every integer ${\ell} > 0$, and  ${\ell}$-tuple
\begin{equation*}
\textbf{\emph{k}} = (k_1, \ldots, k_{\ell}) \in \mathbb{N}^{\ell}
\end{equation*}
with $k_i \geq 2$,  $i = 1, \ldots, {\ell}$,
there exists a one-one map 
 \begin{equation} \label{z4}
\Phi_{\textbf{\emph{k}}} : l^2(A^{\ell}) \rightarrow L^{\infty}(\Omega_A^{\ell}, \mathbb{P}_A^{\ell}),
  \end{equation}
 with the following properties:\\\\
 (1) \ \  $\Phi_{\textbf{\emph{k}}}$  is  $\big ( l^2(A^{\ell}) \rightarrow L^2(\Omega_A^{\ell}, \mathbb{P}_A^{\ell}) \big )$-continuous, and \\
  \begin{equation} \label{g38}
  \|\Phi_{\textbf{\emph{k}}}(\textbf{\emph{x}})\|_{L^{\infty}} \leq K^{\ell} \|\textbf{\emph{x}}\|_2, \ \ \ \textbf{\emph{x}} \in l^2(A^{\ell}),
  \end{equation}\\
  where $K > 1$ is the absolute constant in \eqref{g36}.\\\\
 (2) \ \  Let  $m$ be a positive integer, and let  $\mathcal{U} = (S_1, \ldots, S_n)$ be a covering sequence of $[m]$ with $I_{\mathcal{U}} \geq 2$.  Let
 \begin{equation} \label{z1}
\textbf{\emph{k}}_i = (k_j(\mathcal{U}): j \in S_i), \ \ i \in [n],
\end{equation}\\\ 
where $k_j(\mathcal{U})$ is the incidence of $j$ in $\mathcal{U}$, defined in \eqref{g32}.  Then, for $\textbf{\emph{x}}_i \in l^2(A^{S_i}), \ i \in [n]$,\\\  
  \begin{equation} \label{g41}
  \begin{split}
\eta_{\mathcal{U}}(\textbf{\emph{x}}_1, \ldots, \textbf{\emph{x}}_n) & =  \convolution_{\mathcal{U}}\big(\Phi_{\textbf{\emph{k}}_1}(\textbf{\emph{x}}_1),  \ldots,   \Phi_{\textbf{\emph{k}}_n}(\textbf{\emph{x}}_n)\big),\\\\
 & = \sum_{\bm{\chi} \in \widehat{\Omega}_A^m}\widehat{\Phi_{\textbf{\emph{k}}_1}(\textbf{\emph{x}}_1)}\big(\pi_{S_1}(\bm{\chi}) \big) \cdots \widehat{\Phi_{\textbf{\emph{k}}_n}(\textbf{\emph{x}}_n)}\big(\pi_{S_n}(\bm{\chi})\big). \\\\
\end{split}
 \end{equation}
\end{theorem}
\begin{proof}
 The map  
\begin{equation*}
\Phi_{\textbf{k}}: l^2(A^{\ell}) \rightarrow L^{\infty}(\Omega_A^{\ell}, \mathbb{P}_A^{\ell})
  \end{equation*}\\\
is constructed by successive applications of $\Phi_{k_i}$.  Given $\textbf{k} = (k_1, \ldots, k_{\ell}) \in \mathbb{N}^{\ell}$ with $k_i \geq 2$ \ ($i = 1, \ldots, \ell$), we formally define
\begin{equation} \label{b19a}
\begin{split}
 \Phi_{\textbf{k}}(\textbf{x})  &:= \big(\Phi_{k_{\ell}} \circ \cdots \circ \Phi_{k_1}\big)(\textbf{x}) \\\
 & := \Phi_{k_{\ell}}\big( \cdots \big(\Phi_{k_1}(\textbf{x})\big) \cdots \big), \ \ \ \ \textbf{x} \in l^2(A^{\ell}), 
 \end{split}
  \end{equation}\\\
where  the $\Phi_{k_i}$  of Lemma \ref{L6} are iteratively applied  to "slices" of vectors in $l^2(A^{\ell - i +1})$,  whose coordinates are elements in  $L^{\infty}(\Omega_A^{i-1},\mathbb{P}_A^{i-1})$. 

The construction of $\Phi_{\textbf{k}}$ is by induction on $\ell$.   For $\ell = 1$ and $k \geq 2$,  we take the $\Phi_{k}$ provided by Lemma \ref{L6}.  For $\ell > 1$, we assume the $(\ell -1)$-step:  that for every $(\ell - 1)$-tuple 
\begin{equation*}
\textbf{k} = (k_1, \ldots, k_{\ell-1}) \in \mathbb{N}^{\ell-1}
\end{equation*}
with $k_i \geq 2$ ($i = 1, \ldots, \ell -1$),
we have  
 \begin{equation} \label{za4}
\Phi_{\textbf{k}} : l^2(A^{\ell-1}) \rightarrow L^{\infty}(\Omega_A^{\ell-1}, \mathbb{P}_A^{\ell-1}),
  \end{equation}
 such that  $\Phi_{\textbf{k}}$  is  $\big ( l^2(A^{\ell-1}) \rightarrow L^2(\Omega_A^{\ell -1}, \mathbb{P}_A^{\ell -1}) \big )$-continuous, and \\
  \begin{equation} \label{ga38}
  \|\Phi_{\textbf{k}}(\textbf{x})\|_{L^{\infty}} \leq K^{\ell-1} \|\textbf{x}\|_2, \ \ \ \textbf{x} \in l^2(A^{\ell-1}),
  \end{equation}\\
  where $K > 1$ is the absolute constant in \eqref{g36}.  Now let 
  \begin{equation*}
\textbf{k} = (k_1, \ldots, k_{\ell}) \in \mathbb{N}^{\ell},
\end{equation*}
with $k_i \geq 2$  ($i = 1, \ldots, \ell$), and 
\begin{equation} \label{z7}
\textbf{k}^{\prime} = (k_1, \ldots, k_{\ell -1}).
\end{equation}
For $\textbf{x} \in l^2(A^{\ell})$ and fixed $\alpha \in A$, define $\textbf{x}^{({\alpha})} \in l^2(A^{\ell-1})$ by
\begin{equation}\label{za6}
\textbf{x}^{({\alpha})}(\alpha_1, \ldots, \alpha_{\ell -1}) = \textbf{x}(\alpha, \alpha_1, \ldots, \alpha_{\ell -1}), \ \ \ \ (\alpha_1,\ldots, \alpha_{\ell-1}) \in l^2(A^{\ell-1}).
\end{equation}
 (The definition in \eqref{za6}  is temporary:  $\textbf{x}^{({\alpha})}$ is not the same as $\textbf{x}^{(j)}$ defined in \eqref{g2} or \eqref{gg2}.)  We apply the map $\Phi_{\textbf{k}^{\prime}}$ (provided by the $(\ell -1)$-step) to $\textbf{x}^{({\alpha})}$,  and thus obtain
\begin{equation}
  \|\Phi_{\textbf{k}^{\prime}}\big(\textbf{x}^{({\alpha})}\big)\|_{L^{\infty}(\Omega_A^{\ell-1}, \mathbb{P}_A^{\ell-1})} \ \leq \ K^{\ell-1} \|\textbf{x}^{({\alpha})}\|_2.
  \end{equation}\\
  Then for almost all $(\omega_1, \ldots, \omega_{\ell-1}) \in (\Omega_A^{\ell-1}, \mathbb{P}_A^{\ell-1})$, we have\\\
  \begin{equation}\label{za5}
  \begin{split}
  \sum_{\alpha \in A} \big |\Phi_{\textbf{k}^{\prime}}\big(\textbf{x}^{({\alpha})}\big)(\omega_1, \ldots, \omega_{\ell-1}) \big|^2  \ &\leq \ K^{2(\ell-1)}  \sum_{\alpha \in A} \|\textbf{x}^{({\alpha})}\|^2_{l^2(A^{\ell-1})}\\\\
  & \leq  \ K^{2(\ell-1)} \|\textbf{x}\|^2_{l^2(A^{\ell})}.\\\\
  \end{split}
  \end{equation}
  We then apply  $\Phi_{k_{\ell}}$ to $\Phi_{\textbf{k}^{\prime}}\big(\textbf{x}^{(\cdot)}\big)(\omega_1, \ldots, \omega_{\ell-1})$ (Lemma \ref{L6}), and thus obtain\\\
  \begin{equation} \label{g39}
 \Phi_{\textbf{k}}(\textbf{x}) := \Phi_{k_{\ell}} \big(\Phi_{\textbf{k}^{\prime}}\big(\textbf{x}^{(\cdot)}\big) \ \in \ L^{\infty}(\Omega_A^{\ell}, \mathbb{P}_A^{\ell}),
 \end{equation}
 with the estimate
 \begin{equation}\label{ga39}
 \|\Phi_{\textbf{k}}(\textbf{x}) \|_{L^{\infty}} \leq K^{\ell} \|\textbf{x}\|_2.
 \end{equation}

 The  $\big(l^2(A^{\ell}) \rightarrow L^2(\Omega_A^{\ell}, \mathbb{P}_A^{\ell})\big)$-continuity of $\Phi_{\textbf{k}}$ follows also by induction on $\ell \geq 1.$  The case $\ell = 1$ is Lemma \ref{L6}.  Let $\ell > 1$, and assume continuity in the case $\ell - 1.$  To prove the inductive step, we assume that a sequence $(\textbf{x}_j)$ in $l^2(A^{\ell})$ converges to $\textbf{x}$ in the $l^2(A^{\ell})$-norm, and   proceed to verify that there is a subsequence $(\textbf{x}_{j_i})$ such that $\Phi_{\textbf{k}}(\textbf{x}_{j_i})$  converges to $\Phi_{\textbf{k}}(\textbf{x})$ in the $L^2(\Omega_A^{\ell},\mathbb{P}_A^{\ell})$-norm. 
 
 For fixed $\alpha \in A$, by the assumption and by the definition in \eqref{za6}, we have $\textbf{x}_j^{(\alpha)}$ converging to $\textbf{x}^{(\alpha)}$ in $l^2(A^{\ell - 1})$.  Therefore, by the induction hypothesis,
 \begin{equation}
 \Phi_{\textbf{k}^{\prime}}\big(\textbf{x}_j^{({\alpha})}\big) \  \underset{j \to \infty}{\longrightarrow} \ \Phi_{\textbf{k}^{\prime}}\big(\textbf{x}^{({\alpha})}\big)  \ \ \ \text{in} \ L^2(\Omega_A^{\ell -1}, \mathbb{P}_A^{\ell -1}),
 \end{equation}
 with $\textbf{k}^{\prime}$ defined in \eqref{z7}.  By dominated convergence, 
 \begin{equation}
 \sum_{\alpha \in A} \int_{\bm{\omega} \in \Omega_A^{\ell -1}}\big | \Phi_{\textbf{k}^{\prime}}\big(\textbf{x}_j^{({\alpha})}\big)(\bm{\omega}) - \Phi_{\textbf{k}^{\prime}}\big(\textbf{x}^{({\alpha})}\big)(\bm{\omega})\big|^2 \  \mathbb{P}_A^{\ell -1}(d \bm{\omega}) \underset{j \to \infty}{\longrightarrow} 0,
 \end{equation}
and therefore, by interchanging sum and integral, 
\begin{equation}
  \int_{\bm{\omega} \in \Omega_A^{\ell -1}} \bigg (\sum_{\alpha \in A}\big | \Phi_{\textbf{k}^{\prime}}\big(\textbf{x}_j^{({\alpha})}\big)(\bm{\omega}) - \Phi_{\textbf{k}^{\prime}}\big(\textbf{x}^{({\alpha})}\big)(\bm{\omega})\big|^2 \bigg) \mathbb{P}_A^{\ell -1}(d \bm{\omega}) \underset{j \to \infty}{\longrightarrow} 0.
 \end{equation}
 Therefore, there exists a subsequence $(j_i: i \in \mathbb{N})$ with the property that for almost every $\bm{\omega} \in (\Omega_A^{\ell-1}, \mathbb{P}_A^{\ell -1}),$
\begin{equation}
\Phi_{\textbf{k}^{\prime}}\big(\textbf{x}_{j_i}^{(\cdot)}\big)(\bm{\omega}) \ \underset{i \to \infty}{\longrightarrow} \ \Phi_{\textbf{k}^{\prime}}\big(\textbf{x}^{(\cdot)}\big)(\bm{\omega}) \ \ \ \text{in} \ \ l^2(A).
\end{equation}
Therefore, by $(l^2 \rightarrow L^2)$-continuity in the case $\ell = 1$ (Lemma \ref{L6}), for almost every $\bm{\omega} \in (\Omega_A^{\ell-1}, \mathbb{P}_A^{\ell -1}),$
\begin{equation}
\int_{\omega \in \Omega_A} \big| \Phi_{k_{\ell}} \big(\Phi_{\textbf{k}^{\prime}}\big(\textbf{x}_{j_i}^{(\cdot)}\big)(\bm{\omega})\big) -  \Phi_{k_{\ell}} \big(\Phi_{\textbf{k}^{\prime}}\big(\textbf{x}^{(\cdot)}\big)(\bm{\omega})\big) \big|^2 \  \mathbb{P}_A(d\omega) \ \underset{i \to \infty}{\longrightarrow} \ 0.
\end{equation}
Therefore, by dominated convergence and the definition of $\Phi_{\textbf{k}}$ (cf. \eqref{g39}), we conclude\\
 \begin{equation} \label{g40}
 \|\Phi_{\textbf{k}}(\textbf{x}_{j_i}) - \Phi_{\textbf{k}}(\textbf{x})\|_{L^2}^2 \ \underset{i \to \infty}{\longrightarrow} \ 0.
 \end{equation}\\

Next, we prove  \eqref{g41} by induction on $m$ (cf. proof of Lemma \ref{L5}).  The case $m=1$ is Lemma \ref{L6}.   Suppose $m > 1$, and that $\mathcal{U} = (S_1, \ldots, S_n)$ is a covering sequence of $[m]$ with $I_{\mathcal{U}} \geq 2$.  For $i \in ]n]$, let
  \begin{equation*}
  S_i^{\prime} = S_i \setminus \{m\}, 
  \end{equation*}
  and
  \begin{equation*}
  \textbf{k}_i^{\prime} = (k_j(\mathcal{U}^{\prime}): j \in S_i^{\prime}).
  \end{equation*}
  Then,  $\mathcal{U}^{\prime} = (S_1^{\prime}, \ldots, S_n^{\prime})$ is a covering sequence of $[m-1]$ with  $I_{\mathcal{U}^{\prime}} \geq 2$.  Let $$T := \{i:  m \in S_i\}$$   (as per \eqref{ba12}), and note that if $i \notin T$, then  $S_i^{\prime} = S_i$ and $ \textbf{k}_i^{\prime} = \textbf{k}_i$.   Let $$\emph{\bf{x}}_1 \in l^2(A^{S_1}), \ldots, \emph{\bf{x}}_n \in l^2(A^{S_n}).$$ For $i \in T$ and $u \in A$, define $\emph{\bf{x}}_i^{(u)} \in l^2(A^{S_i^{\prime}})$ by 
\begin{equation*}
\emph{\bf{x}}_i^{(u)}(\bm{\alpha}) = \emph{\bf{x}}_i(\bm{\alpha},u),  \ \ \ \bm{\alpha} \in A^{S_i^{\prime}}.
\end{equation*}  
 (Cf. \eqref{za6}.)  Then, by the induction hypothesis and Lemma \ref{L7}, 
  \begin{equation} \label{g43}
  \begin{split}
 & \eta_{\mathcal{U}}(\textbf{x}_1, \ldots, \textbf{x}_n) =  \\\\
   &   \ \ \ =  \sum_{u \in A} \  \sum_{\bm{\alpha} \in A^{[m-1]}}  \ \prod_{i \in [n] \setminus T}\emph{\bf{x}}_i\big(\pi_{S_i^{\prime}}(\bm{\alpha})\big) \  \prod_{i \in T}\emph{\bf{x}}_i^{(u)}\big(\pi_{S_i^{\prime}}(\bm{\alpha}) \big)\\\\
   &   \ \ \ =  \sum_{u \in A} \  \sum_{\bm{\chi} \in \widehat{\Omega}_A^{m-1}} \prod_{i \in [n] \setminus T}\big(\Phi_{\textbf{k}_i}(\textbf{x}_i)\big)^{\wedge}\big(\pi_{S_i}(\bm{\chi}) \big) \ \prod_{i \in T} \big(\Phi_{\textbf{k}_i^{\prime}}(\textbf{x}_i^{(u)})\big)^{\wedge}\big(\pi_{S_i^{\prime}}(\bm{\chi})\big)\\\\
&   \ \ \ =  \sum_{\bm{\chi} \in \widehat{\Omega}_A^{m-1}} \prod_{i \in [n] \setminus T}\big(\Phi_{\textbf{k}_i}(\textbf{x}_i)\big)^{\wedge}\big(\pi_{S_i}(\bm{\chi}) \big) \  \sum_{u \in A} \  \prod_{i \in T}\big(\Phi_{\textbf{k}_i^{\prime}}(\textbf{x}_i^{(u)})\big)^{\wedge}\big(\pi_{S_i^{\prime}}(\bm{\chi})\big).
\end{split}
\end{equation}\\\
For $\bm{\chi} \in \widehat{\Omega}_A^{m-1}$ and $i \in T$, define (for typographical convenience) $\textbf{v}_{i,\bm{\chi}} \in l^2(A)$  by
\begin{equation*}
\textbf{v}_{i,\bm{\chi}}(u) = \bigg(\Phi_{\textbf{k}_i^{\prime}}(\textbf{x}_i^{(u)})\bigg)^{\wedge}\big(\pi_{S_i^{\prime}}(\bm{\chi})\big).
\end{equation*}\\\
From Lemma \ref{L6} and the recursive definition of $\Phi_{\textbf{k}_i}$, we obtain\\\
\begin{equation} \label{g44}
\begin{split}
\sum_{u \in A} \  \prod_{i \in T} \bigg(\Phi_{\textbf{k}_i^{\prime}}(\textbf{x}_i^{(u)})\bigg)^{\wedge}\big(\pi_{S_i^{\prime}}(\bm{\chi})\big)  & = \sum_{u \in A} \  \prod_{i \in T}  \textbf{v}_{i,\bm{\chi}}(u) \\\\
& = \sum_{\gamma \in \widehat{\Omega}_A} \prod_{i \in T} \bigg(\Phi_{k_m}(\textbf{v}_{i,\bm{\chi}}) \bigg)^{\wedge}(\gamma) \\
(\text{by Lemma \ref{L6}, with \ }  k = k_m := k_m(\mathcal{U}) = |T|)\\\\
& =  \sum_{\gamma \in \widehat{\Omega}_A} \prod_{i \in T} \big(\Phi_{\textbf{k}_i}(\textbf{x}_i)\big)^{\wedge}(\pi_{S_i}((\bm{\chi},\gamma)))\\
(\text{by the recursive definition of \ } \Phi_{\textbf{k}_i}).
\end{split}
\end{equation}\\\
 Substituting  \eqref{g44} in \eqref{g43}, we obtain\\\\
\begin{equation*}
\begin{split}
 \eta_{\mathcal{U}}(\textbf{x}_1, \ldots, \textbf{x}_n) & =  \sum_{\bm{\chi} \in \widehat{\Omega}_A^{m-1}} \prod_{i \in [n] \setminus T}\big(\Phi_{\textbf{k}_i}(\textbf{x}_i)\big)^{\wedge}\big(\pi_{S_i}(\bm{\chi}) \big) \ \  \sum_{\gamma \in \widehat{\Omega}_A} \prod_{i \in T} \big(\Phi_{\textbf{k}_i}(\textbf{x}_i)\big)^{\wedge}(\pi_{S_i}\big((\bm{\chi},\gamma)\\\\
 & = \sum_{\bm{\chi} \in \widehat{\Omega}_A^m} \ \prod_{i=1}^n\ \big(\Phi_{\textbf{k}_i}(\textbf{x}_i)\big)^{\wedge}\big(\pi_{S_i}(\bm{\chi}) \big),
 \end{split}
 \end{equation*}
 and thus the integral representation in \eqref{g41}.
\end{proof}
\vskip0.5cm

\begin{corollary} \label{C5}
For every integer  $m \geq 1$, and every covering sequence \ $\mathcal{U} = (S_1, \ldots, S_n )$ of $[m]$ with $I_{\mathcal{U}} \geq 2$,  the multilinear functional $\eta_{\mathcal{U}}$ (defined in \eqref{g45}) is projectively continuous, and
\begin{equation} \label{g50}
\|\eta_{\mathcal{U}}\|_{\tilde{V}_n({\bf{B}}^{[\mathcal{U}]})} \leq K^{\beta_{\mathcal{U}}},
\end{equation}
where $$\emph{\bf{B}}^{[\mathcal{U}]} := B_{l^2(A^{S_1})} \times \cdots \times B_{l^2(A^{S_n})},$$
  \begin{equation}\label{z18}
  \beta_{\mathcal{U}} := \sum_{j=1}^m k_j(\mathcal{U}) \  \  \big(= \sum_{i=1}^n|S_i|  \ \big),
  \end{equation}
  and $K > 1$ is the constant in \eqref{approxn}.
  \end{corollary} \ \\  
  
 \begin{remark}  \label{R7} \ \\\
\em{\bf{i}.} \ We illustrate the proof of Theorem \ref{T4} in the case  $m = 3$ and \ $$\mathcal{U} = \big( \{1,2\}, \{2,3\}, \{1,3\} \big),$$  whence  ${\ell} = 2$, and
\begin{equation} \label{z2}
\begin{split}
{\bf{k}}_1 &= \big(k_1(\mathcal{U}), k_2(\mathcal{U}) \big ) = (2,2)\\\  {\bf{k}}_2 &= \big(k_2(\mathcal{U}), k_3(\mathcal{U}) \big ) = (2,2) \\\ {\bf{k}}_3 &= \big(k_1(\mathcal{U}), k_3(\mathcal{U}) \big ) = (2,2).\\
\end{split}
\end{equation} 
(Cf. \eqref{z1}.) \\

 First we construct the map 
\begin{equation} \label{z3}
\Phi_{(2,2)}: l^2(A^2) \rightarrow L^{\infty}(\Omega_A^2, \mathbb{P}_A^2).
\end{equation}
  For  ${\bf{x}} \in l^2(A^2)$ and fixed $\alpha \in A$,  we apply $\Phi_2$ (supplied by Lemma \ref{L6} with $k = 2$)  to ${\bf{x}}^{(\alpha)} :=  {\bf{x}}(\alpha, \cdot) \in l^2(A)$, and deduce that for almost all $\omega \in (\Omega_A, \mathbb{P}_A)$
  \begin{equation} \label{z8}
 \sum_{\alpha \in A} \|\Phi_2 \big(\textbf{x}^{(\alpha)} \big)(\omega)  \|_{l^2(A)}^2
  \leq  \sum_{\alpha \in A} K^2  \|\textbf{x}^{(\alpha)}  \|_2^2 = K^2\|\textbf{x}\|_2^2. 
 \end{equation}\\\ 
 (Cf. \eqref{za5}.)  Now apply $\Phi_2$ to  $\Phi_2 \big(\textbf{x}^{(\cdot)} \big)(\omega) \in l^2(A)$ for almost all $\omega \in (\Omega_A, \mathbb{P}_A)$,  and obtain 
 \begin{equation}
 \Phi_{(2,2)}({\bf{x}}) := \Phi_2 \big(\Phi_2 (\textbf{x}^{(\cdot)}) \big) \in L^{\infty}(\Omega_A^2, \mathbb{P}_A^2),
 \end{equation}\\\ 
 with the estimate
 \begin{equation} \label{z9}
 \|\Phi_{(2,2)}(\textbf{x})\|_{L^{\infty}} \leq K \esssup_{\omega \in \Omega_A}\big  \| \Phi_2 \big(\textbf{x}^{(\cdot)}\big)(\omega)\big \|_{l^2(A)} \leq K^2 \|\textbf{x}\|_2.
 \end{equation}\\\
 (Cf. \eqref{g39} and \eqref{ga39}.)\\

Next, to verify the $(l^2 \rightarrow L^2)$-continuity of $\Phi_{(2,2)}$, we prove that if  $(\textbf{x}_j)$ is a sequence in $l^2(A^2)$ converging to $\textbf{x}$ in the $l^2(A^2)$-norm, then there is a subsequence $(\textbf{x}_{j_i})$ such that $\Phi_{(2,2)}(\textbf{x}_{j_i})$  converges to $\Phi_{(2,2)}(\textbf{x})$ in the $L^2(\Omega_A^2,\mathbb{P}_A^2)$-norm.   First, because ${\bf{x}}_j^{(\alpha)}$ converges to ${\bf{x}}^{(\alpha)}$ in the $l^2(A)$-norm for every $\alpha \in l^2(A)$, we obtain from the $(l^2 \rightarrow L^2)$-continuity  of $\Phi_2$, 
 \begin{equation}
\int_{\omega \in \Omega_A} \big | \Phi_2\big ({\bf{x}}_j^{(\alpha)} \big)(\omega) - \Phi_2\big ({\bf{x}}^{(\alpha)} \big)(\omega) \big |^2 \  \mathbb{P}_A(d\omega)  \underset{j \to \infty}{\longrightarrow} 0.
  \end{equation}\\
  Therefore, by dominated convergence (via \eqref{z8}), and by interchange of sum and integral,  we obtain
  \begin{equation} \label{z12c}
  \int_{\omega \in \Omega_A}\bigg ( \sum_{\alpha \in A} \big | \Phi_2\big ({\bf{x}}_j^{(\alpha)} \big)(\omega) - \Phi_2\big ({\bf{x}}^{(\alpha)} \big)(\omega) \big |^2 \bigg) \mathbb{P}_A(d\omega)  \underset{j \to \infty}{\longrightarrow} 0.
  \end{equation}\\\
Therefore, there exists a subsequence $(j_i)$ such that for almost all $\omega \in (\Omega_A, \mathbb{P}_A)$
\begin{equation*}
\sum_{\alpha \in A} \big | \Phi_2\big ({\bf{x}}_{j_i}^{\alpha)} \big)(\omega) - \Phi_2\big ({\bf{x}}^{(\alpha)} \big)(\omega) \big |^2  \underset{i \to \infty}{\longrightarrow} 0.
\end{equation*}
Therefore, by a second application of the $(l^2 \rightarrow L^2)$-continuity of $\Phi_2$,  we have for almost all $\omega_2 \in (\Omega_A, \mathbb{P}_A)$
\vskip 0.1cm 
\begin{equation*} 
\int_{\omega_1 \in \Omega_A}\big |\Phi_2\big(\Phi_2\big ({\bf{x}}_{j_i}^{(\cdot)} \big)(\omega_2)\big)(\omega_1) - \Phi_2\big(\Phi_2\big ({\bf{x}}^{(\cdot)} \big)(\omega_2)\big)(\omega_1) \big |^2  \ \mathbb{P}_A(d\omega_1) \ \underset{i \to \infty}{\longrightarrow} 0, 
\end{equation*}\ \\
and by dominated convergence we conclude \ \\ 
\begin{equation*} 
 \begin{split}
& \|\Phi_{(2,2)}(\textbf{x}_{j_i}) - \Phi_{(2,2)}(\textbf{x})\|_{L^2}^2  \\\
 &= \int_{\omega_2 \in \Omega_A} \left (\int_{\omega_1 \in \Omega_A}\big |\Phi_2\big(\Phi_2\big ({\bf{x}}_{j_i}^{(\cdot)} \big)(\omega_2)\big)(\omega_1) - \Phi_2\big(\Phi_2\big ({\bf{x}}^{(\cdot)} \big)(\omega_2)\big)(\omega_1)\right ) \mathbb{P}_A(d \omega_2) \ \underset{i \to \infty}{\longrightarrow} 0.
 \end{split}
 \end{equation*} \ \\\
  Finally, we verify \eqref{g41}  by three successive applications of   \eqref{g34} with $k=2$ (in Lemma \ref{L6}).  Specifically, for ${\bf{x}}, \ {\bf{y}}, \ {\bf{z}}$ \  in \ $l^2(A^2)$, 
 \begin{equation} \label{z12}
 \begin{split}
\eta_{\mathcal{U}}({\bf{x}},{\bf{y}},{\bf{z}}) \ & = \  \sum_{\alpha_1,\alpha_2,\alpha_3}{\bf{x}}(\alpha_1,\alpha_2){\bf{y}}(\alpha_2,\alpha_3){\bf{z}}(\alpha_1,\alpha_3)\\\\
&= \ \sum_{\alpha_3}  \ \sum_{\alpha_2} \ {\bf{y}}(\alpha_2,\alpha_3) \ \sum_{\alpha_1} \ {\bf{x}}(\alpha_1,\alpha_2){\bf{z}}(\alpha_1,\alpha_3)\\\\
&= \  \ \sum_{\alpha_3}  \ \sum_{\alpha_2} \ {\bf{y}}(\alpha_2,\alpha_3) \ \int_{\omega_1} \Phi_2 \big( {\bf{x}}(\cdot, \alpha_2) \big)(\omega_1) \ \Phi_2 \big( {\bf{z}}(\cdot, \alpha_3) \big)(\omega_1) \ d{\omega_1}\\\\
&= \ \sum_{\alpha_3} \ \int_{\omega_1} \bigg( \Phi_2 \big( {\bf{z}}(\cdot, \alpha_3) \big)(\omega_1) \ \sum_{\alpha_2} \ {\bf{y}}(\alpha_2,\alpha_3) \ \Phi_2 \big( {\bf{x}}(\cdot, \alpha_2) \big)(\omega_1) \bigg)d{\omega_1}\\\\
& = \ \sum_{\alpha_3} \ \int_{\omega_1} \Phi_2 \big( {\bf{z}}(\cdot, \alpha_3)(\omega_1) \bigg( \int_{\omega_2}  \Phi_2 \big( {\bf{y}}(\cdot, \alpha_3)(\omega_2)\ \Phi_{(2,2)}({\bf{x}})(\omega_1,\omega_2) \ d\omega_2 \bigg) d\omega_1\\\\
& = \   \int_{\omega_1} \int_{\omega_2}\Phi_{(2,2)}({\bf{x}})(\omega_1,\omega_2) \bigg ( \sum_{\alpha_3}\Phi_2 \big( {\bf{y}}(\cdot, \alpha_3)(\omega_2) \ \Phi_2 \big( {\bf{z}}(\cdot, \alpha_3)(\omega_1) \bigg ) d \omega_2 d \omega_1 \\\\
& = \ \int_{\omega_1} \int_{\omega_2}\Phi_{(2,2)}({\bf{x}})(\omega_1,\omega_2) \bigg( \int_{\omega_3} \Phi_{(2,2)}({\bf{y}})(\omega_2,\omega_3) \ \Phi_{(2,2)}({\bf{z}})(\omega_1,\omega_3) \ d\omega_3 \bigg )d \omega_2 d \omega_1 \\\\
& = \ \int_{\omega_1} \int_{\omega_2} \int_{\omega_3}\Phi_{(2,2)}({\bf{x}})(\omega_1,\omega_2) \ \Phi_{(2,2)}({\bf{y}})(\omega_2,\omega_3) \ \Phi_{(2,2)}({\bf{z}})(\omega_1,\omega_3) \ d\omega_3d\omega_2d\omega_1\\\\
& = \convolution_{\mathcal{U}} \big (\Phi_{(2,2)}({\bf{x}}),\Phi_{(2,2)}({\bf{y}}),\Phi_{(2,2)}({\bf{z}}) \big ),
 \end{split}
 \end{equation}
 which is the needed integral representation of  $\eta_{\mathcal{U}}({\bf{x}}, {\bf{y}},{\bf{z}}).$  
 Therefore (cf. Corollary \ref{C5}), \ \\
 \begin{equation}
 \|\eta_{\mathcal{U}}\|_{\tilde{V}_3(B_{l^2}^3)} \leq K^{\beta_{\mathcal{U}}} = K^6.
 \end{equation} \ \\\
 {\bf{ii}.} \ The homogeneity of the "base" map $\Phi_k$ (of Lemma \ref{L6}) implies that the "amalgam" map  $$\Phi_{{\bf{k}}} := \Phi_{(k_1,\ldots,k_{\ell})}: \ l^2(A^{\ell}) \ \rightarrow \ L^{\infty}(\Omega_A^{\ell}, \mathbb{P}_A^{\ell})$$ (of Theorem \ref{T4}), when restricted to a tensor product of Hilbert spaces, admits a natural factorization.  We illustrate this in the case $\ell = 2,$ and ${\bf{k}} = (2,2)$.  Let ${\bf{x}}_1 \in l^2_{\mathbb{R}}(A)$, ${\bf{x}}_2 \in l^2_{\mathbb{R}}(A)$, and consider the elementary tensor ${\bf{x}}_1\otimes{\bf{x}}_2 \in l^2_{\mathbb{R}}(A^2)$, i.e., 
 \begin{equation}
 ({\bf{x}}_1\otimes{\bf{x}}_2)(\alpha_1,\alpha_2) = {\bf{x}}_1(\alpha_1) {\bf{x}}_2(\alpha_2), \ \ \ (\alpha_1,\alpha_2) \in A^2.
 \end{equation}
 To compute $\Phi_{(2,2)}( {\bf{x}}_1\otimes{\bf{x}}_2)$, first fix $\alpha_1 \in A$ and apply $\Phi_2$ to ${\bf{x}}_1(\alpha_1){\bf{x}}_2$  (a scalar ${\bf{x}}_1(\alpha)$ multiplying a vector ${\bf{x}}_2$), thus obtaining
 \begin{equation}
 \Phi_2\big ({\bf{x}}_1(\alpha_1){\bf{x}}_2 \big) = {\bf{x}}_1(\alpha_1) \Phi_2({\bf{x}}_2) \in L^{\infty}(\Omega_A, \mathbb{P}_A)
 \end{equation}
 (by the homogeneity of $\Phi_2$). Then for almost all $\omega_2 \in (\Omega_A,\mathbb{P}_A)$, apply  $\Phi_2$ to $\big(\Phi_2({\bf{x}}_2)\big)(\omega_2) {\bf{x}}_1$  (a scalar $\big(\Phi_2({\bf{x}}_2)\big)(\omega_2)$ multiplying a vector ${\bf{x}}_1$), thus obtaining 
 \begin{equation}
\Phi_2\bigg(\big(\Phi_2({\bf{x}}_2)\big)(\omega_2) {\bf{x}}_1\bigg) = \big(\Phi_2({\bf{x}}_2)\big)(\omega_2) \Phi_2({\bf{x}}_1)
 \end{equation}
 (again by homogeneity).  Therefore,
 \begin{equation}
 \Phi_{(2,2)}( {\bf{x}}_1\otimes{\bf{x}}_2) = \Phi_2({\bf{x}}_1)\otimes\Phi_2({\bf{x}}_2).
 \end{equation}
 
 Next, take the linear span $l^2(A)\otimes l^2(A)$ (\emph{algebraic tensor product}) of elementary tensors in $l^2(A^2)$, wherein two elements are equal if they determine the same linear functional on $l^2(A^2)$. Consider the projective tensor norm on it,
\begin{equation}
\big \| \sum_j {\bf{x}}_j \otimes {\bf{y}}_j \big \|_{\widehat{\otimes}} := \inf \bigg \{ \sum_j \|{\bf{x}}_j^{\prime}\|_2 \|{\bf{y}}_j^{\prime}\|_2: \sum_j {\bf{x}}_j^{\prime} \otimes {\bf{y}}_j^{\prime} =   \sum_j {\bf{x}}_j \otimes {\bf{y}}_j\bigg \}, \ \ \  \sum_j {\bf{x}}_j \otimes {\bf{y}}_j \in l^2(A)\otimes l^2(A) 
\end{equation} \ \\ 
 (the \emph{greatest cross norm}), and let $l^2(A) \widehat{\otimes} l^2(A)$ (\emph{projective tensor product})  be the completion of  $l^2(A)\otimes l^2(A)$ under this norm.   Similarly, we have the projective tensor product 
\begin{equation} 
 L^{\infty}(\Omega_A, \mathbb{P}_A)\widehat{\otimes}L^{\infty}(\Omega_A, \mathbb{P}_A) := \bigg \{f \in L^{\infty}(\Omega_A^2, \mathbb{P}_A^2):  f = \sum_j g_j\otimes h_j, \ \  \sum_j \|g_j\|_{L^{\infty}} \| h_j\|_{L^{\infty}} < \infty \bigg \},
 \end{equation}
 normed by 
 \begin{equation}
 \|f\|_{L^{\infty} \widehat{\otimes} L^{\infty}} :=\inf  \bigg \{\sum_j \|g_j\|_{L^{\infty}} \| h_j\|_{L^{\infty}}: f = \sum_j g_j\otimes h_j \bigg \}, \ \ \ f \in L^{\infty}(\Omega_A, \mathbb{P}_A)\widehat{\otimes}L^{\infty}(\Omega_A, \mathbb{P}_A).
 \end{equation}
 We conclude:   if \  $\sum_j {\bf{x}}_j \otimes {\bf{y}}_j \in l^2(A) \widehat{\otimes} l^2(A)$,  then \ \\
 \begin{equation}
 \Phi_{(2,2)}\big(\sum_j {\bf{x}}_j \otimes {\bf{y}}_j \big) = \sum_j \Phi_2({\bf{x}}_j)\otimes \Phi_2({\bf{y}}_j), 
 \end{equation}
 and 
 \begin{equation}
 \big \|\Phi_{(2,2)}\big(\sum_j {\bf{x}}_j \otimes {\bf{y}}_j \big) \big \|_{L^{\infty} \widehat{\otimes} L^{\infty}} \leq K^2  \big \| \sum_j {\bf{x}}_j \otimes {\bf{y}}_j \big \|_{l^2\widehat{\otimes}l^2} \ .
 \end{equation}
 \end{remark}
\vskip0.5cm

  \subsection{The left-side inequality in \eqref{g46}} \ \ 
 To verify suffciency in Theorem \ref{T3}, we start with $\theta \in \tilde{\mathcal{V}}_{\mathcal{U}}(A^m)$.  By Proposition \ref{P5}, for every $\epsilon > 0$ there exist $$\mu \in M(\Omega_{A^{S_1}} \times \cdots \times \Omega_{A^{S_n}})$$ such that
  \begin{equation*}
 \theta(\bm{\alpha}) = \hat{\mu}(r_{\pi_{S_1}(\bm{\alpha})}\otimes \cdots \otimes r_{\pi_{S_n}(\bm{\alpha})}), \  \ \ \bm{\alpha} \in A^m,
  \end{equation*}
  and 
  \begin{equation} \label{g51}
  \|\mu\|_M  \leq \|\theta\|_{\tilde{\mathcal{V}}_{\mathcal{U}}} + \epsilon.
  \end{equation}\\\
  For $\emph{\bf{x}}_1 \in l^2(A^{S_1}), \ldots, \emph{\bf{x}}_n \in l^2(A^{S_n})$, denote\\
  \begin{equation}
  G(\textbf{x}_1, \dots, \textbf{x}_n)  = \sum_{\bm{\alpha} \in A^m} \emph{\bf{x}}_1\big(\pi_{S_1}(\boldsymbol{\alpha}) \big) \cdots \emph{\bf{x}}_n\big(\pi_{S_n}(\boldsymbol{\alpha}) \big)r_{\pi_{S_1}(\bm{\alpha})} \otimes \cdots \otimes r_{\pi_{S_n}(\bm{\alpha})}.
  \end{equation}\\\
By Lemma \ref{L5}, $G(\textbf{x}_1, \dots, \textbf{x}_n)$ is a continuous function on $\Omega_{A^{S_1}} \times \cdots \times \Omega_{A^{S_n}}$ with absolutely convergent Walsh series.  Therefore, by the usual Parseval formula \eqref{Par},
\begin{equation}\label{g48}
  \int_{\bm{\zeta} \in \bm{\Omega}_{{\bf{A}}^{[\mathcal{U}]}}}G(\textbf{x}_1, \dots, \textbf{x}_n)(\bm{\zeta}) \  \mu(d \bm{\zeta}) = \eta_{\mathcal{U},\theta}(\textbf{x}_1, \ldots, \textbf{x}_n),
\end{equation}
where
\begin{equation*}
\bm{\Omega}_{{\bf{A}}^{[\mathcal{U}]}} := \Omega_{A^{S_1}} \times \cdots \times \Omega_{A^{S_n}}.
\end{equation*}\\\
For a set $B$, \ ${\bf{x}} \in l^2(B)$, and $\zeta \in \Omega_B   \ ( \ = \{-1,1\}^B$ \ ), \ define ${\bf{x}} \bullet \zeta \in l^2(B)$ by
\begin{equation} \label{z19}
({\bf{x}}\bullet \zeta)(\beta) := {\bf{x}}(\beta)\zeta(\beta) \ (\ = {\bf{x}}(\beta)r_{\beta}(\zeta) \ ), \ \ \ \beta \in B,
\end{equation}
whence    
\begin{equation} \label{g52}
\|\textbf{x} \bullet \zeta)\|_2 = \|\textbf{x}\|_2.
\end{equation}
Then by \eqref{g41} (the case $\theta = 1$), for all\\ $$\bm{\zeta} = (\zeta_1, \ldots, \zeta_n) \in \Omega_{A^{S_1}} \times \cdots \times \Omega_{A^{S_n}} = \bm{\Omega}_{{\bf{A}}^{[\mathcal{U}]}},$$\\ we have
\begin{equation} \label{g49}
G(\textbf{x}_1, \dots, \textbf{x}_n)(\bm{\zeta}) = \convolution_{\mathcal{U}}\big(\Phi_{\textbf{k}_1}(\textbf{x}_1\bullet \zeta_1),  \ldots,   \Phi_{\textbf{k}_n}(\textbf{x}_n \bullet \zeta_n)\big).
\end{equation}\\\
Substituting \eqref{g49} in \eqref{g48}, we obtain
\begin{equation} \label{corr}
 \eta_{\mathcal{U},\theta}(\textbf{x}_1, \ldots, \textbf{x}_n) =  \int_{\bm{\zeta} \in \bm{\Omega}_{{\bf{A}}^{[\mathcal{U}]}}} \bigg( \convolution_{\mathcal{U}}\big(\Phi_{\textbf{k}_1}(\textbf{x}_1\bullet \zeta_1),  \ldots,   \Phi_{\textbf{k}_n}(\textbf{x}_n \bullet \zeta_n)\big) \bigg) \mu(d \bm{\zeta}),
 \end{equation}
 which implies the desired integral representation of $\eta_{\mathcal{U},\theta}$.  Combining \eqref{g50},  \eqref{g51} and \eqref{g52}, we obtain the left side inequality in \eqref{g46}. 
 
 \begin{remark} \label{R9}
 \em{To illustrate ideas, we run the proof of the left side of \eqref{g46} in the archetypal case  $$\mathcal{U} = \big ( \{1,2\}, \{2,3\},\{1,3\}\big ).$$  Let  $\theta \in \tilde{\mathcal{V}}_{\mathcal{U}}(A^3)$,}  and (by Proposition \ref{P5}) obtain a measure  $\mu \in M(\Omega_{A^2} \times \Omega_{A^2} \times \Omega_{A^2})$ such that 
 \begin{equation}
 \theta(\alpha_1,\alpha_2,\alpha_3) = \hat{\mu}(r_{(\alpha_1,\alpha_2)}\otimes r_{(\alpha_2,\alpha_3)}\otimes  r_{(\alpha_1,\alpha_3)}), \ \ \ \  (\alpha_1, \alpha_2,\alpha_3) \in A^3.
 \end{equation}
Then by Parseval's formula, for ${\bf{x}}, \ {\bf{y}}, \ {\bf{z}}$ \  in \ $l^2(A^2)$,\\\\ 
\begin{equation} \label{z20}
\begin{split}
\eta_{\mathcal{U},\theta}&({\bf{x}}, {\bf{y}}, {\bf{z}}) = \\\\
  \int_{(\Omega_{A^2})^3}\bigg(\sum_{(\alpha_1, \alpha_2, \alpha_3) \in A^3}&{\bf{x}}(\alpha_1,\alpha_2){\bf{y}}(\alpha_2,\alpha_3){\bf{z}}(\alpha_1,\alpha_3)  r_{(\alpha_1,\alpha_2)}\otimes r_{(\alpha_2,\alpha_3)}\otimes  r_{(\alpha_1,\alpha_3)} \bigg) d \mu = \\\\
 &\int_{(\Omega_{A^2})^3}\bigg(\eta_{\mathcal{U}}\big({\bf{x}}\bullet \zeta_1, \  {\bf{y}}\bullet \zeta_2, \  {\bf{z}}\bullet \zeta_3\big) \bigg) \mu(d \zeta_1, \zeta_2, \zeta_3),
\end{split}
\end{equation}\\\
where \ ${\bf{x}}\bullet \zeta, \  \  {\bf{y}}\bullet \zeta, \ \  {\bf{z}}\bullet \zeta$ \ are the "random" vectors in $l^2(A^2)$, \ $\zeta \in \Omega_{A^2}$, as defined in \eqref{z19}.  Then, by applying \eqref{z12} (in the case $\theta = 1$) to the integrand in \eqref{z20}, we deduce \\\
\begin{equation}
\eta_{\mathcal{U},\theta}({\bf{x}}, {\bf{y}}, {\bf{z}})
= \int_{(\zeta_1,\zeta_2,\zeta_3) \in \Omega_{A^2}} \bigg(\convolution_{\mathcal{U}} \big (\Phi_{(2,2)}({\bf{x}} \bullet \zeta_1),\Phi_{(2,2)}({\bf{y}}\bullet \zeta_2),\Phi_{(2,2)}({\bf{z}} \bullet \zeta_3) \big )\bigg)\mu(d\zeta_1,d\zeta_2, d\zeta_3).
\end{equation}\\\

 \end{remark}
 
 \subsection{The right-side inequality in \eqref{g46}} \ \  To establish necessity in Theorem \ref{T3}, we verify
 \begin{equation} \label{z21}
 \|\theta\|_{\tilde{\mathcal{V}}_{\mathcal{U}}(A^m)} \leq \|\eta_{\mathcal{U},\theta}\|_{\tilde{\mathcal{V}}_n(\emph{\bf{B}}^{[\mathcal{U}]})},
\end{equation}
where 
\begin{equation*}
\emph{\bf{B}}^{[\mathcal{U}]} := B_{l^2(A^{S_1})} \times \cdots \times B_{l^2(A^{S_n})}.
\end{equation*}
 For a set $E$, let $\hat{E} = \{\hat{e}\}_{e \in E}$ denote the standard basis of $l^2(E)$;  that is,  for $e \in E$ and $e^{\prime} \in E$, 
 \begin{equation*}
\hat{e}(e^{\prime}) =  \left \{
\begin{array}{ccc}
1 & \quad  \text{if} \quad e = e^{\prime}  \\
0 & \quad   \text{ if} \quad  e \not = e^{\prime}. 
\end{array} \right. 
\end{equation*}
 Specifically in our setting, for $S \subset [m]$ and $\bm{\alpha} = (\alpha_j: j \in S) \in A^S$, we have
 \begin{equation*}
 \hat{\bm{\alpha}} = (\hat{\alpha}_j: j \in S) \ ;
 \end{equation*}
 that is, for $\bm{\alpha}^{\prime} = (\alpha^{\prime}_j: j \in S) \in A^S$,\\\
 \begin{equation*}
 \begin{split}
 \hat{\bm{\alpha}}(\bm{\alpha}^{\prime}) &= \prod_{j \in S} \hat{\alpha}_j(\alpha^{\prime}_j)\\\\
  &=  \left \{
\begin{array}{ccc}
1 & \quad  \text{if} \quad \bm{\alpha} = \bm{\alpha}^{\prime} \\
0 & \quad    \text{if} \quad \bm{\alpha} \not =  \bm{\alpha}^{\prime}. 
\end{array} \right.\\\\
\end{split}
 \end{equation*}\\
 Restricting $\eta_{\mathcal{U},\theta}$ to $\Hat{A}^{S_1} \times \cdots \times \hat{A}^{S_n} := \hat{\textbf{A}}^{[\mathcal{U}]}$, we have\\
\begin{equation} \label{g55}
\|\eta_{\mathcal{U},\theta}\|_{\tilde{\mathcal{V}}_n(\hat{\textbf{A}}^{[\mathcal{U}]})} \leq  \|\eta_{\mathcal{U},\theta}\|_{\tilde{\mathcal{V}}_n(\emph{\bf{B}}^{[\mathcal{U}]})}.
\end{equation}\\
For $\hat{\bm{\alpha}} = (\hat{\bm{\alpha}}_1, \ldots , \hat{\bm{\alpha}}_n) \in \Hat{\textbf{A}}^{[\mathcal{U}]}$, we have

 \begin{equation} \label{g54}
 \begin{split}
 \eta_{\mathcal{U},\theta}(\hat{\bm{\alpha}}) & = \sum_{\bm{\alpha}^{\prime} \in A^m}  \theta(\bm{\alpha}^{\prime}) \  \hat{\bm{\alpha}}_1\big(\pi_{S_1}(\bm{\alpha}^{\prime})\big) \cdots  \hat{\bm{\alpha}}_n\big(\pi_{S_n}(\bm{\alpha}^{\prime})\big) \\\\
& =  \left \{
\begin{array}{ccc}
\theta(\bm{\alpha}) & \quad  \text{if} \quad \bm{\alpha} \in  \textbf{A}^{\mathcal{U}} \\
0 & \quad    \text{if} \quad \bm{\alpha} \notin  \textbf{A}^{\mathcal{U}}, 
\end{array} \right.  
\end{split}
\end{equation}\\
where $\textbf{A}^{\mathcal{U}}$ is the fractional Cartesian product in \eqref{g27} with $A$ in place of  $X_i,  \ i \in [n]$.  Combining \eqref{g54} and \eqref{g55}, we conclude (via \eqref{a5} in Remark \ref{R5}.ii)\\
\begin{equation} \label{g56}
 \|\theta\|_{\tilde{\mathcal{V}}_{\mathcal{U}}(A^m)} = \|\eta_{\mathcal{U},\theta}\|_{\tilde{\mathcal{V}}_{\mathcal{U}}(\Hat{A}^m)} \leq \|\eta_{\mathcal{U},\theta}\|_{\tilde{\mathcal{V}}_n (\hat{\textbf{A}}^{[\mathcal{U}]})} \leq  \|
  \eta_{\mathcal{U},\theta}\|_{\tilde{\mathcal{V}}_n(\emph{\bf{B}}^{[\mathcal{U}]})}.
\end{equation}
\vskip0.60cm

\begin{remark} \label{R8}  
 \ \em{For $n \geq 3$, except for covering sequences of type $\big (\overbrace{ \{1\}, \ldots, \{1\} }^n \big )$, all other covering sequences \ $\mathcal{U} = (S_1, \ldots, S_n)$ with $I_{\mathcal{U}} \geq 2$, satisfy $\alpha(\mathcal{U}) > 1$.  For such \ $\mathcal{U}$,  by Theorem \ref{T3}, there exist  bounded $n$-linear functionals $\eta_{\mathcal{U}, \theta}$ that are not projectively bounded;  see Remark \ref{R5}.i and \eqref{a3} therein.

For example, if $n = 3$, then we have exactly four types of trilinear functionals $\eta_{\mathcal{U}_i, \theta}$ ($i = 1, 2, 3, 4$), based on the four  standard covering sequences 
\begin{equation*}
\begin{split}
\mathcal{U}_1 &= \big(\{1\}, \{1\},\{1\} \big ),\\\
\mathcal{U}_2 &= \big(\{1,2\}, \{2,3\},\{1,3\} \big ),\\\
\mathcal{U}_3 &= \big(\{1,2,4\}, \{2,3,4\},\{1,3,4\} \big ),\\\
\text{and}\\\
\mathcal{U}_4 &= \big(\{1,2,4\}, \{2,3,4\},\{1,3\} \big ).\\\
\end{split}
\end{equation*}
(See Remark \ref{R6}.ii, and \eqref{st3} therein.)  By \eqref{a3}, by Theorem \ref{T3}, and because $\alpha(\mathcal{U}_1) = 1$, the trilinear functional $\eta_{\mathcal{U}_1, \theta}$ is projectively continuous for every $\theta \in l^{\infty}(A)$, whereas by \eqref{st4}, there exist bounded trilinear functionals $\eta_{\mathcal{U}_i, \theta}$ ($i = 2, 3, 4$) that are not projectively bounded.   Bounded trilinear functionals based on  $\mathcal{U}_2$  that were \emph{not} projectively bounded appeared first in \cite{varopoulos1974inequality}.
}
\end{remark}

\subsection{Multilinear extensions of the Grothendieck inequality} \ The "averaging" argument used to verify \eqref{i6}  $\Rightarrow$  \eqref{i1} in Proposition \ref{P0} can be analogously used, via the left-side of \eqref{g46}, to prove the following multilinear extension of the Grothendieck inequality.  For every integer $n \geq 2$ and $m \geq 1$, and every covering sequence $\mathcal{U} = \big (S_1, \ldots, S_n \big)$ of $[m]$ with $I_{\mathcal{U}} \geq 2$, there is a constant \  $\mathcal{K}_{\mathcal{U}}$,
\begin{equation}\label{op1}
1 < \mathcal{K}_{\mathcal{U}} \leq  K^{\sum_{i=1}^n|S_i|},
\end{equation}
 where $K > 1$ is the constant in \eqref{approxn}, such that for all $\theta \in \tilde{\mathcal{V}}_{\mathcal{U}}(A^m)$ and every finitely supported scalar $n$-array  $\big(a_{j_1\ldots j_n} \big)_{(j_1, \ldots, j_n) \in \mathbb{N}^n}$,\\\
 \begin{equation}\label{ge}
 \begin{split} 
 \sup \bigg \{\big |\sum_{j_1, \ldots, j_n}a_{j_1\ldots j_n}   \eta_{\mathcal{U},\theta}({\bf{x}}_{1,j_1}, \ldots, {\bf{x}}_{n,j_n}) \big |:  ({\bf{x}}_{1,j_1}, \ldots, & {\bf{x}}_{n,j_n}) \in B_{l^2(A^{S_1})} \times \cdots \times B_{l^2(A^{S_n})}\bigg \}\\\
  \leq  \ \mathcal{K}_{\mathcal{U}} \|\theta\|_{\tilde{\mathcal{V}}_{\mathcal{U}}(A^m)}\sup \bigg \{\big |\sum_{j_1, \ldots, j_n}a_{j_1\ldots j_n} s_{1,j_1} \cdots s_{n,j_n}\big |&: (s_{1,j_1}, \ldots, s_{n,j_n})   \in [-1,1]^n \bigg \}.\\\\
 \end{split}
 \end{equation}\\
 
 The inequality in \eqref{ge}, derived here as a consequence of \\ 
\begin{equation}
 \|\eta_{\mathcal{U},\theta}\|_{\tilde{\mathcal{V}}_n(\emph{\bf{B}}^{[\mathcal{U}]})} \leq K^{\sum_{i=1}^n|S_i|} \|\theta\|_{\tilde{\mathcal{V}}_{\mathcal{U}}(A^m)},
\end{equation}\\
seems weaker than the left-side inequality in \eqref{g46}. Namely,  by applying the duality \\
\begin{equation}
\begin{split}
\bigg(C_{{\emph{\bf{R}}}_{ \emph{\bf{X}}^{[n]}}}({\bf{\Omega}}_{ \emph{\bf{X}}^{[n]}} )\bigg)^{\convolution} & = B({\emph{\bf{R}}}_{ \emph{\bf{X}}^{[n]}}) =  \tilde{\mathcal{V}}_n( \emph{\bf{X}}^{[n]}),\\\\
& \text{with} \ \ \ \emph{\bf{X}}^{[n]} = B_{l^2(A^{S_1})} \times \cdots \times  B_{l^2(A^{S_n})} := \emph{\bf{B}}^{[\mathcal{U}]},\\\\
\end{split}
\end{equation}\\
(cf. \eqref{dua1}, \eqref{dua2}, Proposition \ref{P4}), we obtain that the multilinear inequality in \eqref{ge} is equivalent to the existence of a complex measure $$\lambda \in M(\Omega_{B_{l^2(A^{S_1})}} \times \cdots \times  \Omega_{B_{l^2(A^{S_n})}})$$  with $\|\lambda\|_M = \mathcal{K}_{\mathcal{U}} \|\theta\|_{\tilde{\mathcal{V}}_{\mathcal{U}}(A^m)}$, such that 
 \begin{equation} \label{mult}
 \begin{split}
 \eta_{\mathcal{U},\theta} ({\bf{x}}_1, \ldots, {\bf{x}}_n) &=  \int_{\Omega_{B_{l^2(A^{S_1})}} \times \cdots \times  \Omega_{B_{l^2(A^{S_n})}}}r_{{\bf{x}}_1}\otimes \cdots \otimes r_{{\bf{x}}_n} \ d \lambda\\\\
 &= \widehat{\lambda}(r_{{\bf{x}}_1}\otimes \cdots \otimes r_{{\bf{x}}_n}),
  \ \ \ \ \ \ ({\bf{x}}_1, \ldots, {\bf{x}}_n) \in B_{l^2(A^{S_1})} \times \cdots \times  B_{l^2(A^{S_n})}.\\\
 \end{split}
 \end{equation}
This integral representation of $\eta_{\mathcal{U},\theta}$  is ostensibly weaker than the  integral representation of $\eta_{\mathcal{U},\theta}$ in \eqref{corr}; see Problem \ref{Q10}.\\

\begin{remark} \label{R10} \em{\ A simple $n$-linear extension of the Grothendieck inequality ($n \geq 2$), based on  $\theta = 1$ and $$\mathcal{U} = \big ( \{1\}, \ldots, \{n-1\}, \{1, \ldots, n-1\} \big),$$ had appeared in \cite{Blei:1977uq}, and a subsequent notion of projective boundedness with a prototype of Theorem \ref{T3}, dealing specifically with covering sequences  $\mathcal{U} = (S_1, \ldots, S_n)$ of $[m]$ such that $$I_{\mathcal{U}} \geq 2 \ \ \ \text{and} \ \ \ |S_i| =  k > 0, \ \ i \in [n],$$  appeared  later in  \cite{Blei:1979fk}.  

The dependence of $\mathcal{K}_{\mathcal{U}}$ on $k$ and $n$ had been shown in \cite{Blei:1979fk} to be  $\mathcal{O}(n^{kn})$, which turned out to be far from optimal; e.g., compare this growth with the exponential bounds in \eqref{op1}.  To illustrate the significance of $\mathcal{K}_{\mathcal{U}}$'s growth, we take
\begin{equation}\label{hs1}
\mathcal{U}_n = \big(\{1,2\}, \{2,3\},\ldots, \{1,n\}\big), \ \ n \geq 3,
\end{equation}
and then note, via a result in \cite{davie1973quotient},  that the algebra of Hilbert-Schmidt operators  (the Hilbert space $l^2(\mathbb{N}^2)$ with matrix multiplication) is a \emph{Q-algebra} (a quotient of a uniform algebra) if and only if 
\begin{equation}\label{hs2}
\mathcal{K}_{\mathcal{U}_n} = \mathcal{O}\big(K^n\big),
\end{equation}\\
 for some $K > 1$.  That the growth of \ $\mathcal{K}_{\mathcal{U}_n}$ is indeed given by  \eqref{hs2} was  verified, independently, in  \cite{tonge1978neumann} and \cite{kaijser1977} -- in both by an adaptation of a proof of the Grothendieck inequality given in \cite{pisier1978grothendieck}.  Subsequently in \cite{carne1980banach}, the  inequalities involving $\eta_{\mathcal{U}}$ in  \cite{Blei:1979fk}  were reproved and $\mathcal{K}_{\mathcal{U}}$'s growth was shown to be   $ \mathcal{O}\big(K^{kn}\big)$, via an inductive scheme involving the classical Grothendieck inequality.   
}
\end{remark}

\section{\bf{Some loose ends}} \label{s11}

 \subsection{$\tilde{\mathcal{V}}_2(X \times Y)$ \ vs. \ $\tilde{V}_2(X \times Y)$, \ \  $\mathcal{G}_2(X \times Y)$ \ vs. \ $G_2(X \times Y)$} \  \   If $X$ and $Y$ are topological spaces, then
 \begin{equation} \label{inc}
\begin{split}
\tilde{V}_2(X \times Y) \subset C_b(X \times Y) \cap \tilde{\mathcal{V}}_2(X \times Y),\\\
G_2(X \times Y) \subset C_b(X \times Y) \cap \mathcal{G}_2(X \times Y).\\
\end{split}
\end{equation}
\begin{problem} \label{q1}
Are the inclusions in \eqref{inc}  proper inclusions? 
\end{problem}
\noindent 
 See \S3.1 for definitions, and Remark \ref{R3} for a brief discussion.
 
 \subsection{$\mathcal{G}_2(X \times Y) =  \tilde{\mathcal{V}}_2(X \times Y)$ \  vs. \  $G_2(X \times Y) = \tilde{V}_2(X \times Y)$} \ \ The first equality is the Grothendieck theorem for discrete spaces $X$ and $Y$, as per Theorem \ref{T1}, and the second is its "upgraded" version for topological spaces  $X$ and $Y$, as per Corollary \ref{C1}.  
\begin{problem} \label{Q9}
 Are the two equalities distinguishable by the respective constants associated with them? 
 \end{problem}
 \noindent
   Namely, we have the universal Grothendieck constant 
 \begin{equation}
\mathcal{K}_G := \sup\big \{\|f\|_{\tilde{\mathcal{V}}_2(X \times Y)}: f \in B_{\mathcal{G}_2(X \times Y)} \big \},
 \end{equation} 
where $X$ and $Y$ are infinite sets with no \emph{a priori} structures.  Otherwise, if $X$ and $Y$ are topological spaces, then 
\begin{equation}
\mathcal{K}_{GC}(X\times Y) := \sup\big \{\|f\|_{\tilde{V}_2(X \times Y)}: f \in B_{G_2(X \times Y)} \big \} < \infty.
 \end{equation} 
 The problem becomes:  are there topological spaces  $X$ and $Y$, such that
\begin{equation}
\mathcal{K}_G < \mathcal{K}_{GC}(X\times Y)?
\end{equation}

 \subsection{Projective boundedness vs. projective continuity} \label{q10} \ Every projectively continuous functional is projectively bounded, and in all known instances (supplied by Theorem \ref{T3}), every projectively bounded functional is also projectively continuous.
  \begin{problem}[cf. Problem \ref{q1}] \label{Q10}
 Is every projectively bounded multilinear functional on a Hilbert space also projectively continuous?
 \end{problem}
 
 \begin{problem}[cf.  \eqref{lcc}, Theorem \ref{T3}, Problem \ref{Q9}] \label{num}
For integer $m > 0$,  let \ $\mathcal{U} = \{S_1, \ldots, S_n\}$  be a covering sequence of $[m]$ with $I_{\mathcal{U}} \geq 2.$  Let $A$ be an infinite set, and let $\eta_{\mathcal{U}}$  be the $n$-linear functional on $l^2(A^{S_1})\times \cdots \times l^2(A^{S_n})$ defined in \eqref{g45}.  Let \ $\emph{\bf{B}}^{[\mathcal{U}]} := B_{l^2(A^{S_1})} \times \cdots \times B_{l^2(A^{S_n})}.$   Is
\begin{equation}
\|\eta_{\mathcal{U}}\|_{\tilde{\mathcal{V}}_n({\bf{B}}^{[\mathcal{U}]})} < \|\eta_{\mathcal{U}}\|_{\tilde{V}_n({\bf{B}}^{[\mathcal{U}]})}?
\end{equation}\\
\end{problem}

\subsection{A characterization of projective boundedness} \ Let $\eta$ be a bounded $n$-linear functional on a Hilbert space $H$ with an orthonormal basis $A$, and let $ \theta_{A, \eta}$ be its kernel relative to $A$, as defined in \eqref{ker}.  If $\eta$ is projectively bounded, then $ \theta_{A, \eta} \in B(R_A^n)$, i.e., there exists a complex measure $\lambda \in M(\Omega_A^n)$ such that
\begin{equation}
 \eta(\alpha_1, \ldots, \alpha_n) = \widehat{\lambda}(r_{\alpha_1}\otimes \cdots \otimes r_{\alpha_n}), \ \ \ \ \ (\alpha_1, \ldots, \alpha_n) \in A^n.
\end{equation}
(See Remark \ref{R5}.ii.)  Whether the converse holds is an open question:  
\begin{problem} \label{q11}
Suppose $ \theta_{A, \eta} \in B(R_A^n)$.  Is $\eta$ projectively bounded?
\end{problem}

\bibliographystyle{apalike}
\bibliography{blei}
\vspace{.3in}
\end{document}